\documentclass[11pt,a4paper, fullpage]{amsart}

\usepackage[centertags]{amsmath}
\usepackage{latexsym,amsthm,amssymb,verbatim}
\usepackage{epsf}
\usepackage[scanall]{psfrag}
\usepackage{graphicx, color}

\newcommand{\appsection}[1]{\let\oldthesection\thesection
  \renewcommand{\thesection}{Appendix \oldthesection}
  \section{#1}\let\thesection\oldthesection}

\newtheorem{defin}{Definition}[section]
\newtheorem{prop-def}[defin]{Proposition-Definition}
\newtheorem{lem}[defin]{Lemma}
\newtheorem{thm}[defin]{Theorem}
\newtheorem{remark}[defin]{Remark}
\newtheorem{cor}[defin]{Corollary}

\newtheorem{conv}[defin]{Convention}
\newtheorem{ex}[defin]{Example}
\newtheorem{claim}{Claim}
\newcommand{\e}{ \hfill $\diamond$}

\begin{document}

\title[Algorithmic Problems]{Algorithmic Problems in Amalgams of Finite Groups: Conjugacy and Intersection Properties}%
\author{L.Markus-Epstein}\footnote{Supported in part at the Technion by a fellowship of the Israel Council for Higher Education}%
\address{Department of Mathematics \\
Technion \\
Haifa 32000, Israel}%
\email{epstin@math.biu.ac.il}%
\maketitle

\begin{abstract}

Geometric methods proposed by Stallings \cite{stal} for treating
finitely generated subgroups of free groups  were successfully
used to solve a wide collection of decision problems for free
groups and their subgroups \cite{b-m-m-w, kap-m, mar_meak, m-s-w,
mvw, rvw, ventura}.

In the present paper we employ the generalized Stallings' methods,
developed by the author in \cite{m-foldings}, to solve various
algorithmic problems concerning finitely generated subgroups of
amalgams of finite groups.

\end{abstract}

\pagestyle{headings}
%


\section{Introduction}

This paper continues the line of \cite{m-algI} and
\cite{m-kurosh}. The primary goal of the sequence of these three
papers is to solve effectively (by finding an algorithm) various
decision problems concerning finitely generated subgroups of
amalgams of finite groups.

Decision (or \emph{algorithmic}) problems is one of the classical
subjects of combinatorial group theory originating in the three
fundamental decision problems posed by Dehn \cite{dehn} in 1911:
the word problem, the conjugacy problem and the isomorphism
problem. As is well known (the reader is referred to
\cite{miller71, miller92} for a survey on decision problems for
groups),  these problems are theoretically undecidable in general.
Thus the celebrated Novikov-Boone theorem asserts that the word
problem is undecidable (p.88 in \cite{l_s}). However restrictions
to some particular classes of groups may yield surprisingly good
results. Remarkable examples include the solvability of the word
problem in one-relator groups (Magnus, see II.5.4 in \cite{l_s})
and in hyperbolic groups (Gromov, see 2.3.B in \cite{gro}).

In free groups a big success in this direction is due to the
geometrical methods proposed by Stallings \cite{stal} in the early
80's. Recall that Stallings  showed that every finitely generated
subgroup of a free group is canonically represented by a minimal
immersion of a bouquet of circles.  Using the graph theoretic
language, the results of \cite{stal} can be restated as follows. A
finitely generated subgroup of a free group is canonically
represented by a finite labelled graph which can be constructed
algorithmically by a so called process of \emph{Stallings'
foldings} (\emph{Stallings' folding algorithm}). Moreover, this
algorithm is quadratic in the size of the input \cite{kap-m,
m-s-w}. See \cite{tuikan} for a faster implementation of this
algorithm.

This approach reviled as extremely useful to solve algorithmic
problems in free groups. See \cite{b-m-m-w, mar_meak, m-s-w, mvw,
rvw, ventura} for  examples of the applications of the Stallings'
approach in free groups, and \cite{kap-w-m, kmrs, m_w, schupp} for
the applications in some other classes of groups. Note that
Stallings' ideas were recast in a combinatorial graph theoretic
way in the remarkable survey paper of Kapovich and Myasnikov
\cite{kap-m}, where these methods were applied systematically to
study the subgroup structure of free groups.

Our recent results \cite{m-foldings} show that Stallings' methods
can be completely generalized to the class of  amalgams of finite
groups. Along the current paper we refer to this generalization of
Stallings' folding algorithm as the \emph{generalized Stallings'
folding algorithm}. Its description is included in the Appendix.
Let us emphasize that the generalized Stallings' algorithm is
quadratic in the size of the input, which yields  a quadratic time
solution of the membership problem in amalgams of finite groups
(see \cite{m-foldings}).

We employ these generalized Stallings' methods  to answer a
collection of algorithmic questions concerning finitely generated
subgroups of amalgams of finite groups, which extends the results
presented in \cite{kap-m}. Our results include polynomial
solutions for the following algorithmic problems (which are known
to be unsolvable in general \cite{miller71, miller92}) in amalgams
of finite groups:
\begin{itemize}
\item computing subgroup presentations,
%
\item detecting triviality of a given subgroup,
\item the freeness problem,
\item  the finite index problem,
\item the separability problem,
\item the conjugacy problem,
\item the normality,
\item the intersection problem,
\item the malnormality problem,
\item the power problem,
\item reading off Kurosh decomposition for finitely generated
subgroups of free products of finite groups.
\end{itemize}

These results are spread out between three papers: \cite{m-algI,
m-kurosh} and the current one. In \cite{m-kurosh} free products of
finite groups are considered, and an efficient procedure to read
off a Kurosh decomposition is presented.

The splitting between \cite{m-algI} and the current paper was done
with the following idea in mind. It turn out that some subgroup
properties, such as computing of a subgroup presentation and
index, as well as detecting of freeness and normality, can be
obtained directly by an analysis of the corresponding subgroup
graph.
 Solutions of others require some additional
constructions. Thus, for example, intersection properties can be
examined via product graphs, and separability needs constructions
of a pushout of graphs.

In \cite{m-algI}  algorithmic problems of the first type are
presented: the computing of subgroup presentations, the freeness
problem and the finite index problem. The separability problem is
also included there, because it is closely related with the other
problems presented in \cite{m-algI}. The rest of the algorithmic
problems are introduced in the current paper.


The paper is organized as follows. The Preliminary Section
includes the description of the basic notions  used along the
present paper. Readers familiar with  amalgams, normal words in
amalgams and labelled graphs can skip it. The next section
presents a summary of the results from \cite{m-foldings} which are
essential for our algorithmic purposes. It describes the nature
and the properties of the subgroup graphs constructed by the
generalized Stallings' folding  algorithm in \cite{m-foldings}.
The rest of the sections are titled by the names of various
algorithmic problems and present definitions (descriptions) and
solutions of the corresponding algorithmic problems. The relevant
references to other papers considering similar problems and a
rough analysis of the complexity of the presented solutions
(algorithms) are provided. In contrast with papers that establish
the exploration of the complexity of decision problems as their
main goal (for instance, \cite{generic-case, average-case,
tuikan}), we do it rapidly (sketchy) viewing in its analysis a way
to emphasize the effectiveness  of our methods.

\subsection*{Other Methods}  \
There have been a number of papers, where methods, not based on
Stallings' foldings, have been presented. One can use these
methods to treat finitely generated subgroups of amalgams of
finite groups. A topological approach can be found in works of
Bogopolskii \cite{b1, b2}. For the automata theoretic approach,
see papers of Holt and Hurt \cite{holt-decision, holt-hurt},
papers of Cremanns, Kuhn, Madlener and Otto \cite{c-otto,
k-m-otto}, as well as the recent paper of Lohrey and Senizergues
\cite{l-s}.

However the methods for treating finitely generated subgroups
presented in the above papers were applied to some particular
subgroup property. No one of these papers has as its goal a
solution of various algorithmic problems, which we consider as our
primary aim.  Moreover, similarly to the case of free groups (see
\cite{kap-m}), our combinatorial approach seems to be the most
natural one for this purpose.


\section{Acknowledgments}

I wish to deeply thank to my PhD advisor Prof. Stuart W. Margolis
for introducing me to this subject, for his help  and
encouragement throughout  my work on the thesis. I owe gratitude
to Prof. Arye Juhasz for his suggestions and many useful comments
during the writing of this paper. I gratefully acknowledge a
partial support at the Technion by a fellowship of the Israel
Council for Higher Education.

%
\section{Preliminaries} \label{section:Preliminaries}
\subsection*{Amalgams}

Let $G=G_1 \ast_{A} G_2$ be a free product of $G_1$ and $G_2$ with
amalgamation, customary, an \emph{amalgam} of $G_1$ and $G_2$.
We assume that the (free) factors are given by the  finite group
presentations
\begin{align} G_1=gp\langle X_1|R_1\rangle, \ \ G_2=gp\langle
X_2|R_2\rangle \ \ {\rm such \ that} \ \ X_1^{\pm} \cap
X_2^{\pm}=\emptyset. \tag{\text{$1.a$}}
\end{align}
 $A= \langle Y   \rangle$ is a group such that there exist two
monomorphisms
\begin{align}
\phi_1:A \rightarrow G_1 \ {\rm and } \ \phi_2:A \rightarrow G_2.
\tag{\text{$1.b$}}
\end{align}
Thus $G$ has a finite group presentation
\begin{align}
G=gp\langle X_1,X_2 | R_1, R_2, \phi_1(a)=\phi_2(a), \; a \in Y
\rangle. \tag{\text{$1.c$}}
\end{align}

We  put $X=X_1 \cup X_2$,  $R=R_1 \cup R_2 \cup
\{\phi_1(a)=\phi_2(a) \; | \; a \in Y \} $. Thus $G=gp\langle
X|R\rangle$.

As is well known \cite{l_s, m-k-s, serre}, the free factors embed
in $G$. It enables us to identify $A$ with its monomorphic image
in each one of the free factors. Sometimes in order to make the
context clear we use \fbox{$G_i \cap A$}
\footnote{Boxes are used for emphasizing purposes only.}
to denote the monomorphic image of $A$ in $G_i$ ($i \in \{1,2\}$).

Elements of $G=gp \langle X |R \rangle$ are equivalence classes of
words. However it is customary to blur the distinction between a
word $u$ and the equivalence class containing $u$. We will
distinguish between  them by using different equality signs:
\fbox{``$\equiv$''} for the equality of two words and
\fbox{``$=_G$''} to denote the equality of two elements of $G$,
that is the equality of two equivalence classes. Thus in
$G=gp\langle x \; | \; x^4 \rangle$, for example, $x \equiv x$ but
$x \not\equiv x^{-3}$, while $x=_G x^{-3}$.



\subsection*{Normal Forms}
Let $G=G_1 \ast_A G_2$. A word $g_1g_2 \cdots g_n \in G$ is
\emph{in normal form} (or, simply, it is a \emph{normal word}) if:
\begin{enumerate}
    \item [(1)] $g_i \neq_G 1$ lies in one of the  factors, $G_1$ or $G_2$,
    \item [(2)] $g_i$ and $g_{i+1}$ are in different factors,
    \item [(3)] if $n \neq 1$, then $g_i \not\in A$.
\end{enumerate}
We call the sequence $(g_1, g_2, \ldots, g_n)$ a
 \emph{normal decomposition} of the element $g \in  G $, where $g=_G g_1g_2 \cdots g_n$.

Any $g \in G$ has a representative in a normal form (see, for
instance, p.187 in \cite{l_s}).  If $g \equiv g_1g_2 \cdots g_n $
is in normal form and $n>1$, then the Normal Form Theorem (IV.2.6
in \cite{l_s}) implies that $g \neq_G 1$. The number $n$ is unique
for a given element $g$ of $G$ and it is called the \emph{syllable
length} of $g$. We denote it $l(g)$. We use $|g|$ to denote the
length of $g$ as a word in $X^*$.


\subsection*{Labelled graphs}
Below we follow the notation of \cite{gi_sep, stal}.

A graph $\Gamma$ consists of two sets $E(\Gamma)$ and $V(\Gamma)$,
and two functions $E(\Gamma)\rightarrow E(\Gamma)$  and
$E(\Gamma)\rightarrow V(\Gamma)$: for each $e \in E$ there is an
element $\overline{e} \in E(\Gamma)$ and an element $\iota(e) \in
V(\Gamma)$, such that $\overline{\overline{e}}=e$ and
$\overline{e} \neq e$.

The elements of $E(\Gamma)$ are called \textit{edges}, and an $e
\in E(\Gamma)$ is a \emph{direct edge} of $\Gamma$, $\overline{e}$
is the \emph{reverse (inverse) edge} of $e$.

The elements of $V(\Gamma)$ are called \textit{vertices},
$\iota(e)$ is the \emph{initial vertex} of $e$, and
$\tau(e)=\iota(\overline{e})$ is the \emph{terminal vertex} of
$e$. We call them the \emph{endpoints} of the edge $e$.

A  \emph{path of length $n$} is  a sequence of $n$ edges $p=e_1
\cdots  e_n $ such that $v_i=\tau(e_i)=\iota(e_{i+1})$ ($1 \leq
i<n$).  We call $p$ a \emph{path from $v_0=\iota(e_1)$ to
$v_n=\tau(e_n)$}. The \emph{inverse} of the path $p$ is
$\overline{p}=\overline{e_n} \cdots \overline{e_1}$. A path of
length 0 is the \emph{empty path}.

We say that the graph $\Gamma$ is \emph{connected} if $V(\Gamma)
\neq \emptyset$ and any two vertices  are joined by a path. The
path $p$ is \emph{closed} if $\iota(p)=\tau(p)$, and it is
\emph{freely reduced} if $e_{i+1} \neq \overline{e_i}$ ($1 \leq i
<n$). $\Gamma$ is a \emph{tree} if it is a connected graph and
every closed freely reduced path in $\Gamma$ is empty.

A \emph{subgraph} of $\Gamma$ is a graph $C$  such that $V(C)
\subseteq V(\Gamma)$ and $E(C) \subseteq E(\Gamma)$. In this case,
by abuse of language, we write $C\subseteq \Gamma$.
Similarly, whenever we write $\Gamma_1 \cup \Gamma_2$ or $\Gamma_1
\cap \Gamma_2$,  we always mean that the set operations are, in
fact,  applied to the vertex sets and the edge sets of the
corresponding graphs.


A \emph{labelling} of $\Gamma$ by the set $X^{\pm}$ is a function
$$lab: \: E(\Gamma)\rightarrow X^{\pm}$$ such that for each $e \in
E(\Gamma)$, $lab(\overline{e}) \equiv (lab(e))^{-1}$.

The last equality enables one, when representing the labelled
graph $\Gamma$ as a directed diagram,  to represent only
$X$-labelled edges, because $X^{-1}$-labelled edges can be deduced
immediately from them.

A graph with a labelling function is called a \emph{labelled (with
$X^{\pm}$) graph}.  The only graphs considered in the present
paper are labelled graphs.

A labelled graph is called \emph{well-labelled} if
$$\iota(e_1)=\iota(e_2), \; lab(e_1) \equiv lab(e_2)\ \Rightarrow \
e_1=e_2,$$ for each pair of edges $e_1, e_2 \in E(\Gamma)$. See
Figure \ref{fig: labelled, well-labelled graphs}.

\begin{figure}[!h]
\psfrag{a }[][]{$a$} \psfrag{b }[][]{$b$} \psfrag{c }[][]{$c$}
\psfrag{e }[][]{$e_1$}
\psfrag{f }[][]{$e_2$}
\psfragscanon \psfrag{G }[][]{{\Large $\Gamma_1$}}
\psfragscanon \psfrag{H }[][]{{\Large $\Gamma_2$}}
\psfragscanon \psfrag{K }[][]{{\Large $\Gamma_3$}}
\includegraphics[width=\textwidth]{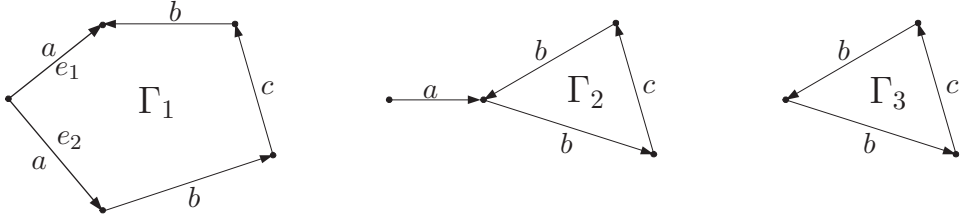}
\caption[The construction of $\Gamma(H_1)$]{ \footnotesize {The
graph $\Gamma_1$ is labelled with $\{a,b,c\} ^{\pm}$, but it is
not well-labelled. The graphs $\Gamma_2$ and $\Gamma_3$ are
well-labelled with $\{a,b,c\} ^{\pm}$.}
 \label{fig: labelled, well-labelled graphs}}
\end{figure}

If a finite graph $\Gamma$ is not well-labelled then a process of
iterative identifications of each pair  $\{e_1,e_2\}$ of distinct
edges with the same initial vertex and the same label to a single
edge yields a well-labelled graph. Such identifications are called
\emph{foldings}, and the whole process is known as the process of
\emph{Stallings' foldings} \cite{b-m-m-w, kap-m, mar_meak, m-s-w}.

 Thus the graph $\Gamma_2$ on Figure
\ref{fig: labelled, well-labelled graphs}  is obtained from the
graph $\Gamma_1$ by folding the edges $e_1$ and $e_2$ to a single
edge labelled by $a$.

Notice that the graph $\Gamma_3$ is obtained from the graph
$\Gamma_2$ by removing the edge labelled by $a$ whose initial
vertex has degree 1. Such an edge is called a \emph{hair}, and the
above procedure is used to be called \emph{``cutting hairs''}.


The label of a path $p=e_1e_2 \cdots e_n$ in $\Gamma$, where $e_i
\in E(\Gamma)$, is the word $$lab(p) \equiv lab(e_1)\cdots
lab(e_n) \in (X^{\pm})^*.$$ Notice that the label of the empty
path is the empty word. As usual, we identify the word $lab(p)$
with the corresponding element in $G=gp\langle X | R \rangle$. We
say that $p$ is   a \emph{normal path} (or $p$ is a path in
\emph{normal form}) if $lab(p)$ is a normal word.

If $\Gamma$ is a well-labelled graph then a path $p$ in $\Gamma$
is freely reduced if and only if $lab(p)$ is a freely reduced
word.
Otherwise $p$  can be converted into a freely reduced path $p'$ by
iterative  removals of the subpaths   $e\overline{e}$
(\emph{backtrackings}) (\cite{mar_meak, kap-m}).  Thus
$$\iota(p')=\iota(p), \ \tau(p')=\tau(p) \ \; {\rm and } \ \; lab(p)=_{FG(X)} lab(p'),$$ where \fbox{$FG(X)$} is a free group
with a free basis $X$. We say that $p'$ is obtained from $p$ by
\emph{free reductions}.


If $v_1,v_2 \in V(\Gamma)$ and $p$ is a path in $\Gamma$ such that
$$\iota(p)=v_1, \ \tau(p)=v_2 \ {\rm and } \ lab(p)\equiv u,$$
then, following the automata theoretic notation, we simply write
\fbox{$v_1 \cdot u=v_2$} to summarize this situation, and  say
that the word $u$ is \emph{readable} at $v_1$ in $\Gamma$.

A pair \fbox{$(\Gamma, v_0)$}  consisting  of the graph $\Gamma$
and the \emph{basepoint} $v_0$ (a distinguished vertex of the
graph $\Gamma$) is called  a \emph{pointed graph}.

Following the notation of Gitik (\cite{gi_sep}) we denote the set
of all closed paths in $\Gamma$ starting at $v_0$  by
\fbox{$Loop(\Gamma, v_0)$},  and the image of $lab(Loop(\Gamma,
v_0))$ in $G=gp\langle X | R \rangle$  by \fbox{$Lab(\Gamma,
v_0)$}. More precisely,
$$Loop(\Gamma, v_0)=\{ p \; | \; p {\rm  \ is \ a \ path \ in \ \Gamma \
with} \ \iota(p)=\tau(p)=v_0\}, $$
$$Lab(\Gamma,v_0)=\{g \in G \; | \; \exists p \in Loop(\Gamma,
v_0) \; : \; lab(p)=_G g \}.$$


It is easy to see that $Lab(\Gamma, v_0)$ is a subgroup of $G$
(\cite{gi_sep}). Moreover, $Lab(\Gamma,v)=gLab(\Gamma,u)g^{-1}$,
where $g=_G lab(p)$, and $p$ is a path in $\Gamma$ from $v$ to $u$
(\cite{kap-m}).
%
%
%
If $V(\Gamma)=\{v_0\}$ and $E(\Gamma)=\emptyset$ then we assume
that $H=\{1\}$.



We say that $H=Lab(\Gamma, v_0)$ is \emph{the subgroup of $G$
determined by the graph $(\Gamma,v_0)$}. Thus any pointed graph
labelled by $X^{\pm}$, where $X$ is a generating set of a group
$G$, determines a subgroup of $G$. This argues the use of the name
\emph{subgroup graphs} for such graphs.

\subsection*{Morphisms of Labelled Graphs} \label{sec:Morphisms
Of Well-Labelled Graphs}

Let $\Gamma$ and $\Delta$ be graphs labelled with $X^{\pm}$. The
map $\pi:\Gamma \rightarrow \Delta$ is called a \emph{morphism of
labelled graphs}, if $\pi$ takes vertices to vertices, edges to
edges, preserves labels of direct edges and has the property that
$$ \iota(\pi(e))=\pi(\iota(e)) \ {\rm and } \
\tau(\pi(e))=\pi(\tau(e)), \ \forall e\in E(\Gamma).$$
An injective morphism of labelled graphs is called an
\emph{embedding}. If $\pi$ is an embedding then we say that the
graph $\Gamma$ \emph{embeds} in the graph $\Delta$.


A \emph{morphism of pointed labelled graphs} $\pi:(\Gamma_1,v_1)
\rightarrow (\Gamma_2,v_2)$  is a morphism of underlying labelled
graphs $ \pi: \Gamma_1\rightarrow \Gamma_2$ which preserves the
basepoint $\pi(v_1)=v_2$. If $\Gamma_2$ is well-labelled then
there exists at most one such morphism (\cite{kap-m}).


\begin{remark}[\cite{kap-m}] \label{unique isomorphism}
{\rm  If two pointed well-labelled (with $X^{\pm}$) graphs
$(\Gamma_1,v_1)$ and $(\Gamma_2,v_2)$  are isomorphic, then there
exists a unique isomorphism $\pi:(\Gamma_1,v_1) \rightarrow
(\Gamma_2,v_2)$. Therefore $(\Gamma_1,v_1)$ and $(\Gamma_2,v_2)$
can be identified via $\pi$. In this case we sometimes write
$(\Gamma_1,v_1)=(\Gamma_2,v_2)$.} \e
\end{remark}

The notation $\Gamma_1=\Gamma_2$ means that there exists an
isomorphism between these two graphs. More precisely, one can find
$v_i \in V(\Gamma_i)$ ($i \in \{1,2\}$) such that
$(\Gamma_1,v_1)=(\Gamma_2,v_2)$ in the sense of Remark~\ref{unique
isomorphism}.


\begin{lem}[\cite{kap-m}] \label{morphism of graphs}
Let $(\Gamma_1,v_1)$ and $(\Gamma_2,v_2)$ be pointed graphs
well-labelled with $X^{\pm}$ such that  $degree(v) >1$
\footnote{Recall $degree(v)=|\{e \in E(\Gamma_i) \; | \;
\iota(e)=v \ {\rm or } \ \tau(e)=v\}|$.}
, for all $v \in V(\Gamma_i) \setminus \{v_i\}$ ($ i \leq
\{1,2\}$).

Then $Lab(\Gamma_1,v_1) \leq Lab(\Gamma_2,v_2)$ if and only if
there exists a unique morphism $\pi:(\Gamma_1,v_1) \rightarrow
(\Gamma_2,v_2)$. \e
\end{lem}

\section{Subgroup Graphs}

The current section is devoted to the discussion on subgroup
graphs constructed by the generalized Stallings' folding
algorithm. The main results of \cite{m-foldings} concerning these
graphs (more precisely, Theorem 7.1, Lemma 8.6, Lemma 8.7, Theorem
8.9 and Corollary 8.11 in \cite{m-foldings}), which are essential
for the present paper, are summarized in Theorem~\ref{thm:
properties of subgroup graphs} below. All the missing notations
are explained along the rest of the present section.



\begin{thm} \label{thm: properties of subgroup graphs}
Let $H=\langle h_1, \cdots, h_k \rangle$ be a finitely generated
subgroup of an amalgam of finite groups $ G=G_1 \ast_A G_2$.

Then there is an algorithm (\underline{the generalized Stallings'
folding  algorithm}) which  constructs a finite labelled graph
$(\Gamma(H),v_0)$ with the following properties:
\begin{itemize}
\item[(1)] $ {Lab(\Gamma(H),v_0)}= {H}. $

\item[(2)] Up to isomorphism, $(\Gamma(H),v_0)$ is  a unique
\underline{reduced precover} of $G$ determining $H$.

\item[(3)] $(\Gamma(H),v_0)$ is the {\underline{normal core}} of
$(Cayley(G,H), H \cdot 1)$.

\item[(4)] A {\underline{normal word}} $g \in G$ is in $H$ if and
only if it labels a closed path in $\Gamma(H)$ starting at $v_0$,
that is $v_0 \cdot g=v_0$.

\item[(5)] Let $m$  be the sum of the lengths of words $h_1,
\ldots h_n$. Then the algorithm computes $(\Gamma(H),v_0)$ in time
$O(m^2)$.
Moreover, $|V(\Gamma(H))|$ and  $|E(\Gamma(H))|$ are proportional
to $m$.

\end{itemize}
\end{thm}

\begin{cor}
Theorem~\ref{thm: properties of subgroup graphs} (4) provides a
solution of the \underline{membership problem} for finitely
generated subgroups of amalgams of finite groups.
\end{cor}

Throughout the present paper the notation \fbox{$(\Gamma(H),v_0)$}
is always used for the finite labelled graph  constructed by the
generalized Stallings' folding  algorithm for a finitely generated
subgroup $H$ of an amalgam of finite groups $G=G_1 \ast_A G_2$.
%



\subsection*{Definition of Precovers:} The notion of
\emph{precovers} was defined by Gitik in \cite{gi_sep} for
subgroup graphs of amalgams. Below we present its definition and
list some  basic properties. In doing so, we rely on the notation
and results obtained in \cite{gi_sep}. The discussion of precovers
which are \emph{reduced}  come later in
Section~\ref{sec:ConjugacyProblem}.

Let $\Gamma$ be a graph labelled with $X^{\pm}$, where $X=X_1 \cup
X_2$ is the generating set of $G=G_1 \ast_A G_2$  given by
(1.a)-(1.c).
We view $\Gamma$ as a two colored graph: one color for each one of
the generating sets $X_1$ and $X_2$ of the factors $G_1$ and
$G_2$, respectively.

The vertex $v \in V(\Gamma)$ is called \emph{$X_i$-monochromatic}
if all the edges of $\Gamma$ incident with $v$ are labelled with
$X_i^{\pm}$, for some $i \in \{1,2\}$. We denote the set of
$X_i$-monochromatic vertices of $\Gamma$ by $VM_i(\Gamma)$ and put
$VM(\Gamma)= VM_1(\Gamma) \cup VM_2(\Gamma)$.

We say that a vertex $v \in V(\Gamma)$ is \emph{bichromatic} if
there exist edges $e_1$ and $e_2$ in $\Gamma$ with
$$\iota(e_1)=\iota(e_2)=v \ {\rm and} \ lab(e_i) \in X_i^{\pm}, \ i \in \{1,2\}.$$
The  set of bichromatic vertices of $\Gamma$ is denoted by
$VB(\Gamma)$.

A subgraph of $\Gamma$ is called \emph{monochromatic} if it is
labelled only with $X_1^{\pm}$ or only with $X_2^{\pm}$. An
\emph{$X_i$-monochromatic component} of $\Gamma$ ($i \in \{1,2\}$)
is a maximal connected subgraph of $\Gamma$ labelled with
$X_i^{\pm}$, which contains at least one edge.
Thus monochromatic components of $\Gamma$ are graphs determining
subgroups of the factors, $G_1$ or $G_2$.

We say that a graph $\Gamma$ is \emph{$ G$-based} if any path $p
\subseteq \Gamma$ with $lab(p)=_G 1$ is closed. Thus if $\Gamma$
is $G$-based then, obviously, it is well-labelled with $X^{\pm}$.

\begin{defin}[Definition of Precover] A $G$-based  graph $\Gamma$
is a \emph{precover} of $G$ if each $X_i$-monochromatic
component of $\Gamma$ is a \emph{cover} of $G_i$  ($i \in
\{1,2\}$).
\end{defin}

Following the terminology of Gitik (\cite{gi_sep}), we use the
term \emph{``covers of $G$''} for \emph{relative (coset) Cayley
graphs} of $G$ and denote by \fbox{$Cayley(G,S)$} the coset Cayley
graph of $G$ relative to the subgroup $S$ of
$G$.\footnote{Whenever the notation $Cayley(G,S)$ is used, it
always means that $S$ is a subgroup of the group $G$ and the
presentation of $G$ is fixed and clear from the context. }
If $S=\{1\}$, then $Cayley(G,S)$ is the \emph{Cayley graph} of $G$
and the notation \fbox{$Cayley(G)$} is used.

Note that the use of the term ``covers'' is adjusted by the  well
known fact that a geometric realization of a coset Cayley graph of
$G$ relative to some $S \leq G$ is a 1-skeleton of a topological
cover corresponding to $S$ of the standard 2-complex representing
the group $G$ (see \cite{stil}, pp.162-163).

\begin{conv}
By the above definition, a precover doesn't have to be a connected
graph. However along this paper we restrict our attention only to
connected precovers. Thus any time this term
 is used, we always mean that the corresponding graph
is connected unless it is stated otherwise.

We follow the convention that a graph $\Gamma$ with
$V(\Gamma)=\{v\}$ and $E(\Gamma)=\emptyset$ determining the
trivial subgroup (that is $Lab(\Gamma,v)=\{1\}$) is a (an empty)
precover of $G$.  \e
\end{conv}

\begin{ex}
{\rm
 Let $G=gp\langle x,y | x^4, y^6, x^2=y^3 \rangle=\mathbb{Z}_4 \ast_{\mathbb{Z}_2} \mathbb{Z}_6$.

Recall that $G$ is isomorphic to $SL(2,\mathbb{Z})$ under the
homomorphism
$$x\mapsto \left(
\begin{array}{cc}
0 & 1 \\
-1 & 0
\end{array}
\right), \ y \mapsto \left(
\begin{array} {cc}
0 & -1\\
1 & 1
\end{array}
 \right).$$
The graphs $\Gamma_1$ and $\Gamma_3$ on Figure \ref{fig:Precovers}
are examples of precovers of $G$ with one monochromatic component
and two monochromatic components, respectively.

Though the $\{x\}$-monochromatic component of the graph $\Gamma_2$
is a cover of $\mathbb{Z}_4 $ and the $\{y\}$-monochromatic
component is a cover of $\mathbb{Z}_6$, $\Gamma_2$ is not a
precover of $G$, because it is not a $G$-based graph. Indeed, $v
\cdot (x^2y^{-3})=u$, while $x^2y^{-3}=_G 1$.

The graph $\Gamma_4$ is not a precover of $G$ because its
$\{x\}$-monochromatic components are not covers of  $\mathbb{Z}_4
$. }\e
\end{ex}
\begin{figure}[!h]
\psfrag{x }[][]{$x$} \psfrag{y }[][]{$y$} \psfrag{v }[][]{$v$}
\psfrag{u }[][]{$u$}
\psfrag{w }[][]{$w$}
\psfrag{x1 - monochromatic vertex }[][]{{\footnotesize
$\{x\}$-monochromatic vertex}}
\psfrag{y1 - monochromatic vertex }[][]{\footnotesize
{$\{y\}$-monochromatic vertex}}
\psfrag{ bichromatic vertex }[][]{\footnotesize {bichromatic
vertex}}
\psfragscanon \psfrag{G }[][]{{\Large $\Gamma_1$}}
\psfragscanon \psfrag{K }[][]{{\Large $\Gamma_2$}}
\psfragscanon \psfrag{H }[][]{{\Large $\Gamma_3$}}
\psfragscanon \psfrag{L }[][]{{\Large $\Gamma_4$}}
\includegraphics[width=\textwidth]{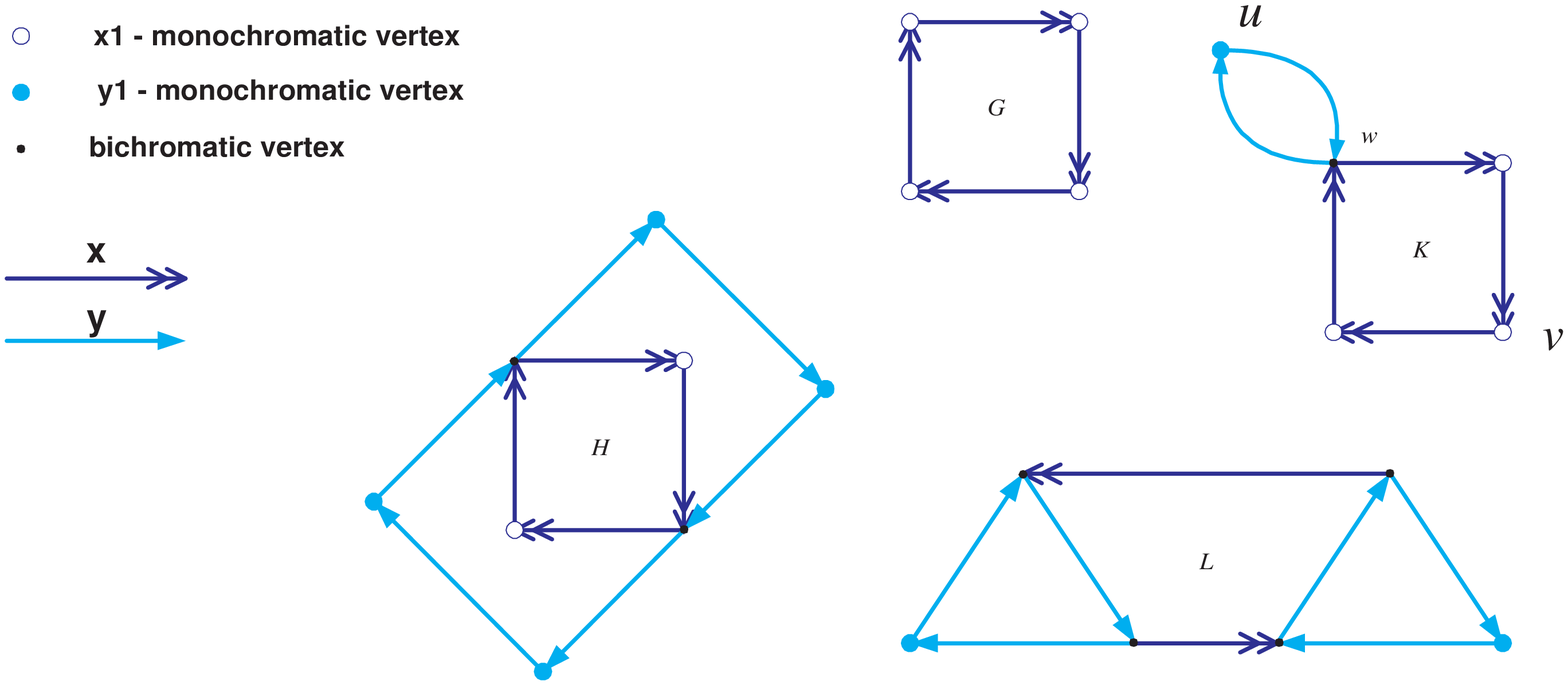}
\caption{ \label{fig:Precovers}}
\end{figure}


A graph $\Gamma$ is \emph{$x$-saturated} at $v \in V(\Gamma)$, if
there exists $e \in E(\Gamma)$ with $\iota(e)=v$ and $lab(e)=x$
($x \in X$). $\Gamma$ is \emph{$X^{\pm}$-saturated} if it is
$x$-saturated for each $x \in X^{\pm}$ at each  $v \in V(\Gamma)$.

\begin{lem}[Lemma 1.5 in \cite{gi_sep}] \label{lemma1.5}
Let $G=gp\langle X|R \rangle$ be a group and let $(\Gamma,v_0)$ be
a graph well-labelled with $X^{\pm}$. Denote $Lab(\Gamma,v_0)=S$.
Then
\begin{itemize}
    \item $\Gamma$ is $G$-based if and only if it can be embedded in $(Cayley(G,S), S~\cdot~1)$,
    \item $\Gamma$ is $G$-based and $X^{\pm}$-saturated if and only if it is isomorphic to  \linebreak[4] $(Cayley(G,S), S
    \cdot~1)$.~
    \footnote{We write $S \cdot 1$ instead of the usual $S1=S$ to distinguish this vertex of $Cayley(G,S)$ as the basepoint of the
    graph.}
\end{itemize}
\end{lem}

\begin{cor}
If $\Gamma$ is a precover of $G$ with $Lab(\Gamma,v_0)=H \leq G$
then  $\Gamma$ is a subgraph of $Cayley(G,H)$.
\end{cor}

Thus a precover of $G$ can be viewed as a part of the
corresponding cover  of $G$, which explains the use of the term
``precovers''.

\begin{remark}[\cite{m-foldings}] \label{remark: morphism of precovers}
{\rm Let $\phi: \Gamma \rightarrow \Delta$ be a morphism of
labelled graphs. If $\Gamma$ is a precover of $G$, then
$\phi(\Gamma)$ is a precover of $G$ as well. }\e
\end{remark}


\subsection*{Precovers are Compatible:}

A graph $\Gamma$  is called \emph{compatible at a bichromatic
vertex} $v$ if for any monochromatic path $p$ in $\Gamma$ such
that $\iota(p)=v$ and $lab(p) \in A$ there exists a monochromatic
path $t$ of a different color in $\Gamma$ such that $\iota(t)=v$,
$\tau(t)=\tau(p)$ and $lab(t)=_G lab(p)$. We say that $\Gamma$ is
\emph{compatible} if it is compatible at all bichromatic vertices.

\begin{ex}
{\rm The graphs $\Gamma_1$ and $\Gamma_3$ on Figure
\ref{fig:Precovers} are compatible. The graph $\Gamma_2$ does not
possess this property because $w \cdot x^{2}=v$, while $w \cdot
y^3=u$. $\Gamma_4$ is not compatible as well.} \e
\end{ex}

\begin{lem} [Lemma 2.12 in \cite{gi_sep}] \label{lemma2.12}
If $\Gamma$ is a compatible graph, then for any  path $p$ in
$\Gamma$ there exists a path $t$ in normal form  such that
$\iota(t)=\iota(p), \ \tau(t)=\tau(p) \ {\rm and} \ lab(t)=_G
lab(p).$
\end{lem}

\begin{remark} [Remark 2.11 in \cite{gi_sep}] \label{remark:
precovers are compatible} {\rm Precovers are compatible. \hfill
$\diamond$}
\end{remark}

The following can be taken as another definition of precovers.

\begin{lem} [Corollary2.13 in \cite{gi_sep}]  \label{corol2.13}
Let $\Gamma$ be a compatible graph. If all  $X_i$-components of
$\Gamma$ are $G_i$-based, $i \in \{1,2\}$, then $\Gamma$ is
$G$-based. In particular, if each $X_i$-component of $\Gamma$ is a
cover of $G_i$, $i \in \{1,2\}$, and $\Gamma$ is compatible, then
$\Gamma$ is a precover of $G$.
\end{lem}


\subsection*{Normal Core and Canonicity:}

\begin{defin} \label{def: normal core}
A vertex of $Cayley(G,H)$ is called \underline{essential} if there
exists a normal path closed at $H \cdot 1$ that goes through it.

The \underline{normal core} $(\Delta, H \cdot 1)$ of $Cayley(G,H)$
is the restriction of $Cayley(G,H)$ to the set of all
\underline{essential vertices}.
\end{defin}

\begin{remark} \label{rm core=union of closed paths_finite grps}
{\rm
    Note that the normal core $(\Delta, H \cdot 1)$ can be viewed as the union
    of all normal paths closed at $H \cdot 1$  in $(Cayley(G,H), H \cdot 1)$.
    Thus $(\Delta, H \cdot 1)$ is a connected graph with  basepoint $H \cdot
    1$.

    Moreover, $V(\Delta)=\{H \cdot 1\}$ and  $E(\Delta)=\emptyset$ if and only if $H$ is the trivial
    subgroup. Indeed, $H$ is not trivial iff there exists $1
    \neq g \in H$ in  normal form iff $g$ labels a normal path in $Cayley(G,H)$
    closed at $H \cdot 1$, iff  $E(\Delta) \neq \emptyset$.}

    \e
\end{remark}

Therefore the normal core of $Cayley(G,H)$ depends on $H$ itself
and not on the set of subgroup generators, which, by
Theorem~\ref{thm: properties of subgroup graphs} (3), implies the
canonicity of the construction of $(\Gamma(H),v_0)$ by the
generalized Stallings' folding  algorithm. This provides a
solution of the Membership Problem for finitely generated
subgroups of amalgams of finite groups given by Theorem~\ref{thm:
properties of subgroup graphs} (4).


\subsection*{Complexity Issues:}
As were noted in \cite{m-foldings}, the complexity of the
generalized Stallings' algorithm is quadratic in the size of the
input, when we assume that all the information concerning the
finite groups $G_1$, $G_2$, $A$ and the amalgam $G=G_1 \ast_{A}
G_2$ given via $(1.a)$, $(1.b)$ and $(1.c)$ (see
Section~\ref{section:Preliminaries}) is not a part of the input.
We also assume that the Cayley graphs and all the relative Cayley
graphs of the free factors are given for ``free'' as well.

Otherwise, if the group presentations of the free factors $G_1$
and $G_2$, as well as the monomorphisms between the amalgamated
subgroup $A$ and the free factors are a part of the input (the
\emph{uniform version} of the algorithm) then we have to build the
groups $G_1$ and $G_2$, that is to construct their Cayley graphs
and relative Cayley graphs.

Since we assume that the groups $G_1$ and $G_2$ are finite,  the
Todd-Coxeter algorithm and the Knuth Bendix algorithm  are
suitable \cite{l_s, sims, stil} for these purposes. Then the
complexity of the construction depends on the group presentation
of $G_1$ and $G_2$ we have: it could be even exponential in the
size of the presentation \cite{cdhw73}. Therefore the generalized
Stallings algorithm, presented in \cite{m-foldings}, with these
additional constructions could take time exponential in the size
of the input.

Thus each uniform algorithmic problem for $H$ whose solution
involves the construction of the subgroup graph $\Gamma(H)$ may
have an exponential complexity in the size of the input.

The primary goal of the complexity analysis introduced along the
current paper is to estimate our graph theoretical methods. To
this end,  we assume that all the algorithms along the present
paper have the following ``given data''.
\begin{description}
    \item[GIVEN] : Finite groups $G_1$, $G_2$, $A$ and the amalgam
$G=G_1 \ast_{A} G_2$ given via $(1.a)$, $(1.b)$ and $(1.c)$.\\
We assume that the Cayley graphs and all the relative Cayley
graphs of the free factors are given.
\end{description}


\section{The Conjugacy Problem}
\label{sec:ConjugacyProblem}

The \emph{conjugacy problem} for subgroups of a group $G$ asks to
answer whether or not given subgroups of $G$ are conjugate. Below
we solve this problem for finitely generated subgroups of amalgams
of finite groups, using subgroup graphs constructed by the
generalized Stallings' algorithm.

Our results extend the analogous ones  obtained for finitely
generated subgroups of free groups by Kapovich and Myasnikov in
\cite{kap-m}. We start by discussing of this analogy. Throughout
the present section we assume that $G=G_1 \ast_A G_2$ is an
amalgam of finite groups.

The  solution of the \emph{conjugacy problem} for finitely
generated subgroups of free groups, presented in \cite{kap-m},
involve a construction of a special graph $Type$ which is a
\emph{core graph} with respect to each of its vertices. Thus it
posses the property that $H, K \leq_{f.g.} FG(X)$ conjugate if and
only if $Type(\Gamma_H)=Type(\Gamma_K)$.

The extended  definition of $Type$ in the case of amalgams of
finite groups as well as a discussion of its properties are
introduced in Section~\ref{subsection:Type}.
Theorem~\ref{thm:Conjugates=>EqType} gives a connection between
$Types$ of conjugate subgroups, which provides a solution of the
conjugacy problem for subgroups in amalgams of finite groups. The
algorithm is presented along with the proof of
Corollary~\ref{algorithm : conjugate subgroups}. The complexity
analysis shows that this algorithm is quadratic in the size of the
input.

\medskip

In \cite{stal} Stallings  defined a \emph{core-graph} to be a
connected graph which has at least one edge, and each of whose
edges belongs to at least one cyclically reduced circuit. He noted
that every connected graph with a non-trivial fundamental group
contains a \emph{core} where the fundamental group is
concentrated, and the original graph consists of this core with
various trees \emph{hangings} on.  Thus given a connected graph
$\Gamma$ which has at least one edge, one can obtain its core by
the process of ``\emph{shaving off trees}''.

In \cite{kap-m} the Stallings' notion of a \emph{core-graph} were
split  into two aspects: a core with respect to some vertex
(\emph{the basepoint})  and a core with respect to any of its
vertices. The first notion corresponds to the subgroup graph
$(\Gamma_S,v_0)$ of $S \leq_{f.g.} F(X)$ constructed by Stallings'
algorithm \cite{stal}, while the second one defines
$Type(\Gamma_S)$.
Thus   $\Gamma_S$  can be obtained from $Cayley(FG(X),S)$  by a
``partial shaving procedure'', which preserves the basepoint $S
\cdot 1$. The ``full shaving procedure'' yields $Type(\Gamma_S)$.
Moreover, $Type(\Gamma_S)$ can be obtained from the subgroup
$\Gamma_S$ by the iterative  erasure of the unique sequence of
\emph{spurs} (\emph{spur} is an edge one of whose endpoints has
degree 1) starting from the basepoint $v_0$ of $\Gamma_S$.

An analog of $(\Gamma_S,v_0)$ in amalgams of finite groups is the
subgroup graph $(\Gamma(H), v_0)$  constructed by the generalized
Stallings' algorithm, where  $H \leq_{f.g.}
 G_1 \ast_A G_2$. By Theorem~\ref{thm: properties of subgroup
graphs} (3), $(\Gamma(H),v_0)$ is the \emph{normal core} of
$(Cayley(G,H), H \cdot 1)$, that is  the union of all normal paths
in $(Cayley(G,H), H \cdot 1)$ closed at $H \cdot 1$. That is, it
is a sort of a core graph with respect to the basepoint $H \cdot
1$.


An analog of a \emph{spur} in subgroup graphs of finitely
generated subgroups of amalgams of finite groups is a
\emph{redundant component}.
The notion of \emph{redundant component} were defined in
\cite{m-foldings}. However in the present context its more
convenient to use the name \emph{redundant component w.r.t.} the
basepoint $v_0$  for that notion defined in \cite{m-foldings}, and
to keep the name \emph{redundant component} for the following.

\begin{defin} \label{def: redundant component}
Let $\Gamma$ be a precover of $G$.
Let $C$ be a $X_i$-monochromatic component of $\Gamma$ ($i \in
\{1,2\}$).  $C$ is  \underline{redundant}  if one of the following
holds.
\begin{enumerate}
\item [(1)] $C$ is the unique monochromatic component of $\Gamma$
(that is $\Gamma=C$) and $Lab(C,v)=\{1\}$ (equivalently, by Lemma
\ref{lemma1.5}, $C$ is isomorphic to $Cayley(G_i)$), where $v \in
V(C)$.
\item [(2)] $\Gamma$ has at least two distinct monochromatic
components and the following holds.

Let $\vartheta \in VB(C)$. Let $K=Lab(C,\vartheta)$ (equivalently,
by Lemma \ref{lemma1.5}, $(C,\vartheta)=(Cayley(G_i,K), K \cdot
1)$).

Then $K \leq A$ and $VB(C)=A(\vartheta)$.
\footnote{Recall that $A(\vartheta)=\{\vartheta \cdot a \; | \; a
\in A\}$ is the \emph{$A$-orbit} of $\vartheta$ in $V(C)$ by the
right action of $A$ on $V(C)$. Since $A_{\vartheta}=K$, the
condition $VB(C) = A(\vartheta)$ can be replaced by its
computational analogue $|VB(C)|=[A:K]$.}
\end{enumerate}

C is  \underline{redundant w.r.t.} the vertex $u\in V(\Gamma)$  if
$C$ is redundant and $u \in V(C)$ implies $u \in VB(C)$ and $K =
\{1\}$.

\end{defin}

\begin{remark} \label{remark: RemovingRedundantCom}
{\rm Similarly to the removing of spurs from graphs representing
finitely generated subgroups  of free group, in the case of
amalgams of finite groups the erasing of redundant components
w.r.t. $v_0$ from $(\Gamma, v_0)$ doesn't change the subgroup
defined by this pointed graph (see Lemma 6.17 in
\cite{m-foldings}).

Namely, if $\Gamma'$ is the graph obtained from $\Gamma $, by
erasing of a monochromatic component  which is redundant w.r.t.
$v_0$, then $Lab(\Gamma',v_0')=Lab(\Gamma , v_0)$, where $v_0'$ is
the image of $v_0$ in $\Gamma'$.

} \e
\end{remark}

The following example attempts to give an intuition of  what
happens in the covering space corresponding to the subgroup $H
\leq G$ of the standard 2-complex representing $G$, when we remove
redundant monochromatic components from a subgraph of
$Cayley(G,H)$, which is the 1-skeleton of this covering space.

\begin{ex}
{\rm

Let $G=gp\langle x,y | x^4, y^6, x^2=y^3 \rangle=G_1 \ast_A G_2$,
where $G_1=gp\langle x | x^4 \rangle$, $G_2=gp\langle y | y^6
\rangle$ and $A=\langle x^2 \rangle=\langle y^3 \rangle$.

Assume that all the redundant monochromatic components are
isomorphic to either $Cayley(G_1)$ or $Cayley(G_2)$.
Hence a removing of  a redundant component from $Cayley(G,H)$ is
expressed in the covering space by removing  a 2-cell with the
boundary path $x^4$ (or $y^6$) and two 2-cells with the boundary
path $x^2y^{-3}$. One can imagine this process as  ``smashing of
bubbles'', see Figure \ref{fig: Bubbles}.

However even if a redundant component is isomorphic to
$Cayley(G_i,S)$, where $\{1\} \neq S \leq G_i$, $i \in \{1,2\}$,
the ``bubbles intuition'' fails as well as in the cases when the
factor groups are not cyclic. That is now the parts removed from
the covering space hardly resemble bubbles, while the motivation
for their removing remains the same.

Here the  common intuition with Stallings' construction: we
``smash bubbles'' instead of ``shaving off trees'', which can be
thought of as  an iterative erasure  of spurs. } \e
\end{ex}
\begin{figure}[!htb]
\begin{center}
\psfrag{x }[][]{\footnotesize $x$} \psfrag{y }[][]{\footnotesize
$y$}
\includegraphics[width=0.5\textwidth]{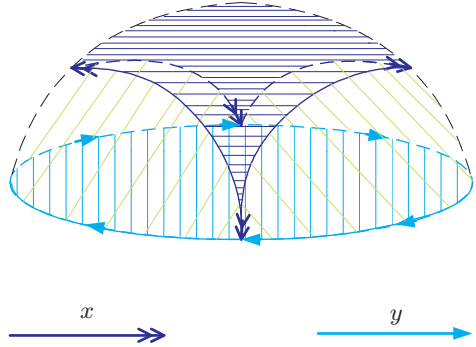}
\caption[A bubble]{{\footnotesize A bubble.}
 \label{fig: Bubbles}}
\end{center}
\end{figure}

The subgroup graph  $(\Gamma(H), v_0)$ is a unique finite
\emph{reduced precover} of $G$, by Theorem~\ref{thm: properties of
subgroup graphs} (2). Now we are ready to recall the precise
definition of this term.

\begin{defin}[Definition~6.18  in \cite{m-foldings}]
\label{def:ReducedPrecover}
A precover $(\Gamma,v_0)$ of $G$ is  \underline{reduced}  if the
following holds.
\begin{itemize}
 \item[(i)]
$(\Gamma,v_0)$ has no redundant components w.r.t. $v_0$.
 \item[(ii)]
$Lab(C_0,v_0) \cap A \neq \{1\}$ implies $v_0 \in VB(\Gamma)$,
where $C_0$ is a monochromatic component of $\Gamma$ such that
$v_0 \in V(C_0)$.
\end{itemize}
\end{defin}

Roughly speaking,
the reduced precover $(\Gamma(H), v_0)$ can be obtained from
$(Cayley(G,H), H \cdot 1)$ by removing of all redundant components
w.r.t. the basepoint $H \cdot 1$. Intuitively, in analogy with
\cite{kap-m}, the graph obtained from $(Cayley(G,H), H \cdot 1)$
by erasing of all redundant components is  $Type(\Gamma(H))$.
Moreover, $Type(\Gamma(H))$ can be obtained from the graph
$\Gamma(H)$ by the iterative  erasure of the unique sequence of
redundant components starting from $C_0$ such that $v_0 \in
V(C_0)$. Some special cases occur  when $H$ is a subgroup of a
factor, $G_1$ or $G_2$, of $G$.



\subsection{Type} \label{subsection:Type}

Consider $(\Gamma(H), v_0)$, where $H$ is a finitely generated
subgroup of an amalgam of finite groups $G=G_1 \ast_A G_2$.

As is mentioned in the introductory  part, a definition of
$Type(\Gamma(H))$  largely relies  on the definition of
$\Gamma(H)$. To this end we start by presenting some properties of
reduced precovers based on the results obtained in
\cite{m-foldings}.

\begin{lem}[Lemma 6.21 in \cite{m-foldings}] \label{lem:PropertiesRedPrecovers}
Let $(\Gamma,v)$ be a precover of $G$ with no redundant components
w.r.t. $v$. Let $H=Lab(\Gamma,v)$.

If $(\Gamma,v)$ is not a reduced precover of $G$, then $Lab(C ,v )
\cap A=S \neq \{1\}$ and $v \in VM_i(\Gamma)$, where $C$ is a
$X_i$-monochromatic component of $\Gamma$ such that $v \in V(C)$
($i \in \{1,2\}$).

Moreover,  $ (\Gamma(H),v_0)=(\Gamma \ast_{\{v \cdot a=Sa \; | \;
a\in A\}} Cayley(G_j, S),\vartheta)$,where $1\leq i \neq j \leq 2$
and $\vartheta$ is the image of $v$ (equivalently, of $S \cdot 1$)
in the amalgam graph.
\end{lem}
\begin{cor} \label{cor:AmalgamGraphEmbedding}
Let $ \Gamma $ be a precover of $G$. Let $C$ be a
$X_i$-monochromatic component of $\Gamma$ and let $v \in VM_i(C)$
($i \in \{1,2\}$).

Then   the graph $\Delta=\Gamma \ast_{\{v \cdot a=Sa \; | \; a\in
A\}} Cayley(G_j, S)$, where $S=Lab(C ,v ) \cap A $, satisfies
\begin{itemize}
 \item $  \Gamma $ and $Cayley(G_j, S)$ embeds in $\Delta$ ($1 \leq i \neq j \leq 2$),
 \item $Lab(\Delta,\vartheta)=Lab(\Gamma,v)$, where $\vartheta$ is the image of
 $v$ in $\Delta$.
\end{itemize}
\end{cor}
%
%
%
%
%

\begin{lem} \label{lem:FormOfRedPrecover}
Each of the following holds.
\begin{itemize}
 \item[(i)] \underline{$H=\{1\}$} if and only if  $V(\Gamma(H))=\{v_0\}$,
 $E(\Gamma(H))=\emptyset$.
 \item[(ii)]  \underline{$H \leq G_i$ and $H \cap A=\{1\}$}
 if and only if $\Gamma(H)$ consists of a unique $X_i$-monochromatic
 component:
 $(\Gamma(H),v_0)=(Cayley(G_i,H),H \cdot 1)$ ($i \in \{1,2\}$).
 \item[(iii)] \underline{$H \leq A$} if and only if
$(\Gamma(H), v_0)=(\Delta, \vartheta)$, where \\
$\Delta=Cayley(G_1,H) \ast_{\{Ha \; | \; a \in A\}} Cayley(G_2,H)$
and $\vartheta$ is the image of $H \cdot 1$ in $\Delta$.

 \item[(iv)] If \underline{$H \nleq G_i$ for all $i \in \{1,2\}$}
 then  $C_0 \subseteq \Gamma(H)$ is a redundant
 component if and only if $v_0 \in V(C_0)$ and
there exists $u_0 \in VB(C_0)$ such that $VB(C_0)=A(u_0)$ and
$Lab(C_0,u_0) \leq A$. Moreover, $C_0$ is a unique redundant
component of $\Gamma(H)$.

\end{itemize}

\end{lem}
\begin{proof}
By Theorem~\ref{thm: properties of subgroup graphs} (2), the
reduced precover $(\Gamma(H),v_0)$  is unique up to isomorphism.
By the Definition~\ref{def:ReducedPrecover}, a graph $(\Delta,u)$
such that $V(\Delta)=\{u\}$ and $E(\Delta)=\emptyset$ is a reduced
precover of $G$, which satisfies $Lab(\Delta,u)=\{1\}$. Therefore
$(\Delta,u)=(\Gamma(H),v_0)$. This gives the ``if'' direction of
(i). Similar arguments prove the ``if'' direction of (ii) and
(iii).

The converse of (i) is trivial.  The opposite direction of (ii) is
true, because, by Lemma~\ref{lemma1.5},
$(\Gamma,v)=(Cayley(G_i,H),H \cdot 1)$ ($i \in \{1,2\}$) implies
$Lab(\Gamma,v)=H$. Moreover, by
Definition~\ref{def:ReducedPrecover}, $(\Gamma,v)$ is a reduced
precover of $G$ with $Lab(\Gamma,v)=H$ whenever $H \leq G_i$ such
that $H \cap A=\{1\}$.

To prove the  converse   of (iii), let $\Gamma=C_1 \ast_{\{v_1
\cdot a=v_2 \cdot a \; | \; a \in A\}} C_2$, where
$(C_i,v_i)=(Cayley(G_i,H),H \cdot 1)$  and $H \leq A$ $(i \in
\{1,2\})$. Let $v$ be the image of $v_i$  in $\Gamma$. Hence
$Lab(C_i,v_i) \leq Lab(\Gamma,v)$.
Now we need the following result from \cite{m-foldings}.
\begin{claim} \label{claim: red precover}
Let $(\Gamma,v)$ be a precover of $G$. Let $C$ be a
$X_i$-monochromatic component of $\Gamma$. Then the followings are
equivalent.
\begin{itemize}
 \item $u_1 \cdot a=u_2$ implies $a \in A$, for all $u_1,u_2 \in
 VB(C)$.
 \item $VB(C)=A(\vartheta)$ and $Lab(C,\vartheta)  \leq
A$, for all $\vartheta \in VB(C)$.
\end{itemize}
\end{claim}
Thus $u_1 \cdot a=u_2$ implies $a \in A$, for all $u_1,u_2 \in
 VB(C_1)=VB(C_2)$. Therefore no normal words of syllable length
 greater than 1 label normal paths in $\Gamma$ closed at $v$.
Hence if $g \in Lab(\Gamma,v)$ and $p$ is a normal path in
$\Gamma$ closed at $v$ such that $lab(p) \equiv g$ then either
$p\subseteq C_1$ or $p \subseteq C_2$. Thus $g \equiv lab(p) \in
H$. Therefore $Lab(\Gamma,v)=H \leq A$. By Theorem~\ref{thm:
properties of subgroup graphs} (2), $(\Gamma,v)=(\Gamma(H),v_0)$.

The statement of (iv) is an immediate consequence of
Definition~\ref{def:ReducedPrecover} and Definition~\ref{def:
redundant component}.

To prove the uniqueness of $C_0$ assume  that there exists another
redundant component $D$ in $\Gamma(H)$ such that $v_0 \in V(D)$
and there exists $u \in VB(D)$ such that $VB(D)=A(u)$ and
$Lab(D,u) \leq A$. Thus, without loss of generality, one can
assume that $C_0$ is a $X_1$-monochromatic component and $D$ is a
$X_2$-monochromatic component. Hence $v_0 \in VB(C_0) \cap VB(D)$.
Therefore $A(u_0)=A(v_0)=A(u)$.
Hence $VB(C_0)=VB(D)$. Since the graph $\Gamma(H)$ is
well-labelled, this implies that $C_0$ and $D$ are the only
monochromatic components of $\Gamma(H)$.

Moreover, since $v_0 \in A(u_0)$, there is $a \in A$ such that
$v_0=u_0 \cdot a$. Hence $Lab(C_0,v_0)=aLab(C_0,u_0)a^{-1} \leq
A$. Similarly, $Lab(D,v_0) \leq A$. Thus
$Lab(C_0,v_0)=A_{v_0}=Lab(D,v_0)$.
\footnote{$A_{v_0}=Lab(C_0,v_0)\cap A$ is the
\emph{$A$-stabilizer} of $v_0$ by the right action of $A$ on
$V(C_0)$.}
Therefore, by (iii), $H=A_{v_0} \leq A$, which contradicts the
assumption of (iv).

\end{proof}

\begin{remark}
{\rm In (iii), the graphs $Cayley(G_1,H)$ and $Cayley(G_2,H)$
embeds
 in $\Delta$, by
Corollary~\ref{cor:AmalgamGraphEmbedding}, }\e
\end{remark}

\begin{cor}[The Triviality Problem] \label{cor:TrivialityProblem}
Let $h_1, \cdots, h_n \in G$.

Then there is an algorithm which decides whether or not the
subgroup $H=\langle h_1, \cdots, h_n \rangle$ is trivial.
\end{cor}
\begin{proof}
We first construct the pointed graph $(\Gamma(H),v_0)$, using the
generalized Stallings' folding  algorithm.

By Theorem~\ref{thm: properties of subgroup graphs} (2),
$(\Gamma(H),v_0)$ is a reduced precover of $G$. Therefore, by
Lemma~\ref{lem:FormOfRedPrecover}(i), $H=Lab(\Gamma(H),v_0)=\{1\}$
if and only if $V(\Gamma(H))=v_0$ and $E(\Gamma(H))=\emptyset$.
\end{proof}

\begin{remark}[Complexity]
{\rm To detect the triviality of a subgroup $H$ given by a set of
generators it takes the same time as to construct the subgroup
graph $\Gamma(H)$. By Theorem~\ref{thm: properties of subgroup
graphs} (5), it is $O(m^2)$,  where $m$ is the sum of the lengths
of words $h_1, \ldots h_n$. } \e
\end{remark}

\begin{lem} \label{lem:DefGammaM}
Let $H$ be a finitely generated subgroup of $G$ such that $H \nleq
G_i$ ($i \in \{1,2\}$).

If $\Gamma(H)$ has a redundant component, then there exists a
unique sequence of alternating monochromatic components $C_0,
\cdots C_{m-1}$ of $\Gamma(H)$ such that the graph $\Gamma_m$,
obtained from $\Gamma(H)$ by the iterative erasure of the above
sequence, has no redundant components.
\end{lem}
\begin{proof}
By Lemma~\ref{lem:FormOfRedPrecover} (iv), $\Gamma(H)$ has the
unique redundant component $C_0$ which satisfies $v_0 \in V(C_0)$
and there exists $u_0 \in VB(C_0)$ such that $VB(C_0)=A(u_0)$ and
$Lab(C_0,u_0) \leq A$.

Let $\Gamma_1$ be the graph   obtained from  $\Gamma(H)$ by
removing of the component $C_0$. That is
$$VM_i(\Gamma_1)=VM_i(\Gamma(H)) \setminus VM_i(C_0), \; \
VM_j(\Gamma_1)=VM_j(\Gamma(H)),$$ $$VB(\Gamma_1)=VB(\Gamma(H))
\setminus VB(C_0)  \ {\rm and} \ E(\Gamma_1)=E(\Gamma(H))
\setminus E(C_0).$$ The resulting graph $\Gamma_1$ is, obviously,
a finite precover of $G$. If $\Gamma_1$ has no redundant
components then $m=1$.

Otherwise there exists a unique $X_j$-monochromatic component of
$\Gamma_{1}$ ($1 \leq i \neq j \leq 2$) which is redundant.
Indeed, $\Gamma(H)$ has a unique $X_j$-monochromatic component
$C_1$  such that $C_0 \cap C_1=VB(C_0)$ ($1 \leq i \neq j \leq
2$). By abuse of notation, we identify the component $C_1$ of
$\Gamma(H)$ with its image in $\Gamma_1$. Thus
$$VB_{\Gamma_1}(C_1)=VB_{\Gamma(H)}(C_1) \setminus VB(C_0).$$
Therefore, $C_1$ is a $X_j$-monochromatic redundant component of
$\Gamma_1$ if and only if there exists a vertex $u_1 \in
VB_{\Gamma(H)}(C_1) \setminus VB(C_0)$ such that $Lab(C_1,u_1)
\leq A$ and $VB_{\Gamma(H)}(C_1)=A(u_1) \cup VB(C_0)=A(u_1) \cup
A(u_0)$.


Since the graph $\Gamma(H)$ is finite, continuing in such manner
one can find the unique sequence
\begin{align}
C_0,C_1, \ldots, C_{m-1} \tag{\text{$\ast$}}\label{sequence of
alternating components}
\end{align}

 of $X_{j_i}$-monochromatic components of $\Gamma(H)$
(see Figure \ref{fig: type construction}) such that the following
holds.


\begin{figure}[!b]
\begin{center}
\psfragscanon \psfrag{A }[][]{{\Large $\Gamma(H)$}}
\psfrag{x }[][]{\footnotesize $x$} \psfrag{y }[][]{\footnotesize
$y$} \psfrag{v }[][]{\small $v_0$}
\psfrag{A3 }[][]{\Large $\Gamma_3$} \psfrag{c0 }[][]{\large $C_0$}
\psfrag{c1 }[][]{ $C_1$} \psfrag{c2 }[][]{ $C_2$} \psfrag{c3
}[][]{ $C_3$}
\psfrag{u1 }[][]{\small $u_1$} \psfrag{u2 }[][]{\small $ u_2$}
\psfrag{u0 }[][]{\small $u_0$}
\psfrag{u3 }[][]{\small $u_3$} \psfrag{u4 }[][]{\small $ u_4$}
\includegraphics[width=\textwidth]{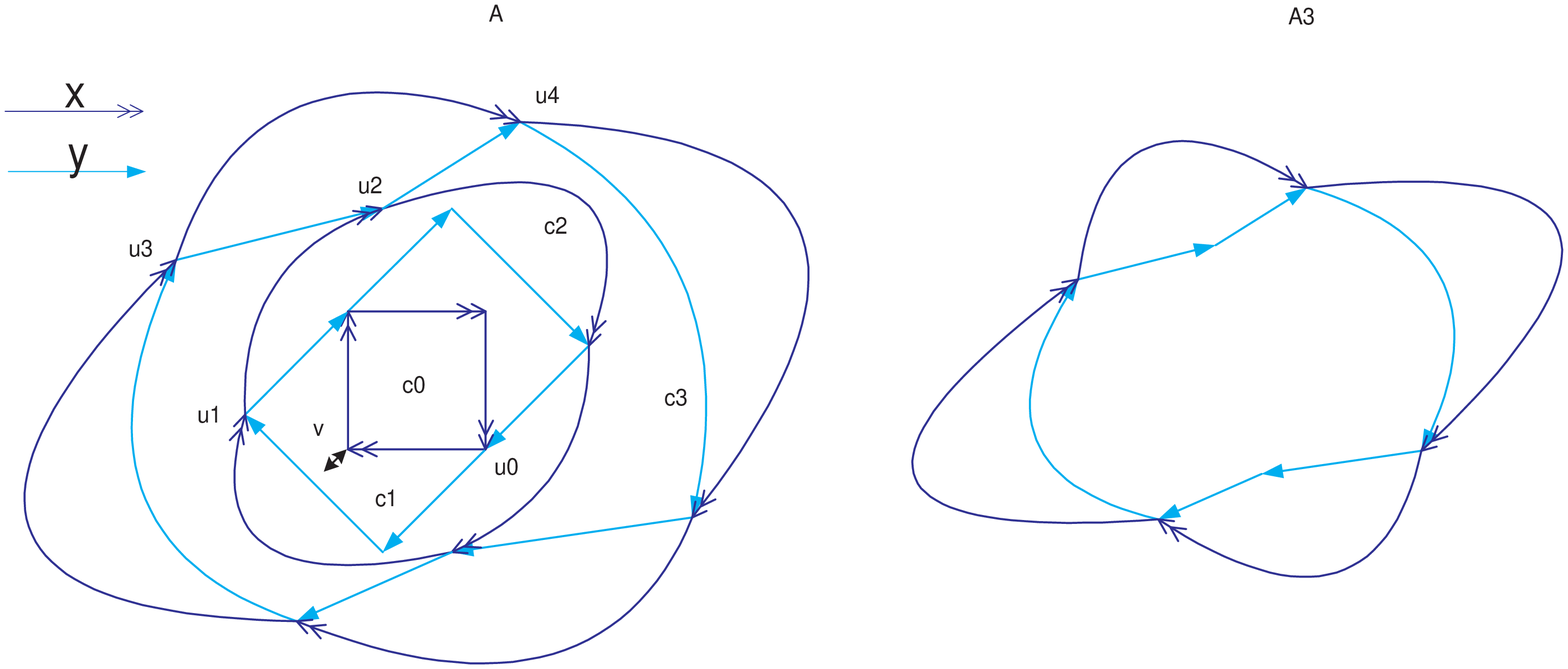}
\caption[The intuition for $Type(\Gamma(H))$, case
1]{{\footnotesize  Let $H \leq_{f.g} G_1 \ast_A G_2 \simeq Z_4
\ast_{Z_2} Z_6$, where $G_1= gp\langle x | x^4 \rangle$, $G_2=
gp\langle y | y^6 \rangle$ and  $A=\langle x^2\rangle=\langle
y^3\rangle$. \\ Thus $C_0,C_1,C_2 $ is the unique sequence of
alternating monochromatic components in the graph $\Gamma(H)$ such
that $\Gamma_3$ has no redundant components. \\ \underline{In
$\Gamma(H)$}:
$(C_0,u_0)=Cayley(G_1)$ and $ VB(C_0) =A(u_0)$;
$(C_1,u_1)=Cayley(G_2)$ and $ VB(C_1) =A(u_0) \cup A(u_1)$;
$(C_2,u_2)=Cayley(G_1)$ and $ VB(C_2) =A(u_1) \cup A(u_2)$;
$(C_3,u_3)=Cayley(G_2)$,   but $VB(C_3)=A(u_2) \cup A(u_3) \cup
A(u_4)$. Thus $Type(\Gamma(H))=\Gamma_3$.}
 \label{fig: type construction}}
\end{center}
\end{figure}


\begin{enumerate}
    \item [(1)] $1 \leq j_i \neq j_{i+1} \leq 2$ for all $0 \leq i \leq
    m-1$.
    \item [(2)] $v_0 \in V(C_0)$ and there exists $u_0
    \in VB(C_0)$ such that $Lab(C_0,u_0) \leq A$ and
    $VB(C_0)=A(u_0)$.
    \item [(3)] For all $1 \leq i \leq m-1$, there exists $u_i
    \in VB(C_i) \setminus VB(C_{i-1})$ such that $Lab(C_i,u_i) \leq A$ and
    $VB(C_i)=A(u_{i-1}) \cup A(u_{i})$.
    \item[(4)] The graph $\Gamma_m$, obtained from  $\Gamma(H)$ by
    the
iterative  removal of sequence ($\ast$), has no redundant
components.
\end{enumerate}
\end{proof}

Following the   notation of Lemma~\ref{lem:DefGammaM} we define.

\begin{defin}[Definition of Type] \label{def:Type}
Let $H$ be a finitely generated subgroup of $G=G_1 \ast_A G_2$.

If $H \leq G_i$ or $\Gamma(H)$ has no redundant components then
$Type(\Gamma(H))=\Gamma(H)$.
Otherwise  $Type(\Gamma(H))=\Gamma_m$.
\end{defin}


\begin{lem}[Properties of $Type(\Gamma(H))$]
\label{lem:PropertiesType} { \ }
\begin{itemize}
 \item[(i)] $Type(\Gamma(H))$ is a finite nonempty precover of
 $G$. \\

Let $v \in V(Type(\Gamma(H)))$. Let $K=Lab(Type(\Gamma(H)),v)$.

 \item[(ii)] $K \neq \{1\}$.
 \item[(iii)] $Lab(\Gamma(H),v)=K$.
 \item[(iv)]
If  $H, K  \nleq A$  then $Type(\Gamma(K))=Type(\Gamma(H))$.\\
If  $H  \leq A$  then  $Type(\Gamma(H))=C_1 \ast_{\{v_1 \cdot
a=v_2 \cdot a \; | \; a \in A\}} C_2$ and
$$Type(\Gamma(K))=\left\{%
\begin{array}{ll}
   C_l, & \hbox{$  K \nleq A$;} \\
    C_l \ast_{\{v \cdot a=K \cdot a \; | \; a \in A\}}
Cayley(G_j,K), & \hbox{ $K \leq A$,} \\
\end{array}%
\right.    $$  where $(C_i,v_i)=(Cayley(G_i,H),H \cdot 1)$, for
all $i \in \{1,2\}$, $Lab(C_l,v)=K$ and  $1\leq l \neq j \leq 2$.
\end{itemize}
\end{lem}
\begin{proof} If $Type(\Gamma(H))=\Gamma(H)$ then the statement of (i)-(iii) is
trivial. Therefore, without loss of generality, we can assume that
$Type(\Gamma(H))=\Gamma_m$, where $\Gamma_m$ is obtained from
$\Gamma(H)$ by the iterative  removal of the unique sequence
($\ast$) of alternating monochromatic components
$$C_0,C_1, \ldots, C_{m-1}.$$

By the construction, $Type(\Gamma(H))=\Gamma_m$ is a finite
precover of $G$. Assume that $\Gamma_m$ consists of a unique
monochromatic component $C_m$, that is $\Gamma_m=C_m$, then $C_m$
is not redundant. Indeed, $|VB_{\Gamma_m}(C_m)|=0$ (see Figure
\ref{fig: type construction Gamma=C}), hence $VB_{\Gamma(H)}(C_m)
=A(u_{m-1})$. Since $v_0 \not\in V(C_m)$ and $(\Gamma(H),v_0)$ is
a reduced precover of $G$, that is has no redundant components
w.r.t. $v_0$, this is possible if and only if $\{1\} \neq
Lab(C_{m},u_{m-1}) \nleq A.$ This completes the proof of (i) and
(ii).


%
\begin{figure}[!h]
\begin{center}
\psfragscanon \psfrag{A }[][]{{\Large $\Gamma(H)$}}
\psfrag{x }[][]{\footnotesize $x$} \psfrag{y }[][]{\footnotesize
$y$} \psfrag{v0 }[][]{\small $v_0$}
\psfrag{a }[][]{\footnotesize $a$} \psfrag{b }[][]{\footnotesize
$b$}
\psfrag{B }[][]{\Large $\Gamma_1$} \psfrag{c0 }[][]{\large $C_0$}
\psfrag{c1 }[][]{\large $C_1$}
\psfrag{u1 }[][]{\small $u_1$} \psfrag{u0 }[][]{\small $u_0$}
\includegraphics[width=0.95\textwidth]{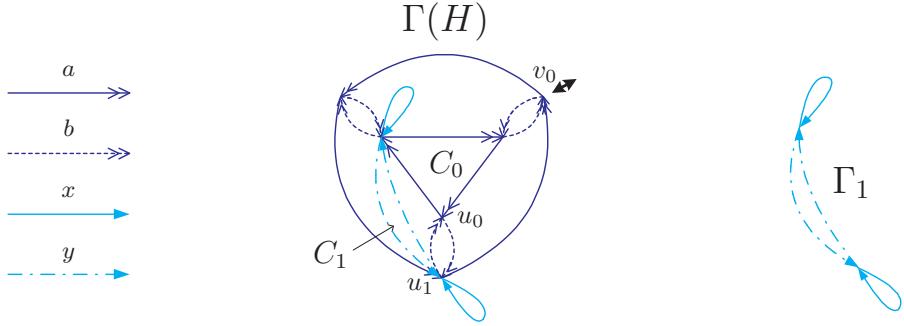}
\caption[The intuition for $Type(\Gamma(H))$, case
2]{{\footnotesize Example of the sequence $C_0,C_1$ of alternating
monochromatic components in the graph $\Gamma(H)$, where $H \leq
{f.g.} G_1 \ast_A G_2 \simeq S_3 \ast_{Z_2} S_3$, where $G_1 =
gp\langle a,b \; | \; a^3, \; b^2, \; ab=ba^2 \rangle$, $G_2
=gp\langle x,y \; | \; x^3, \; y^2, \; xy=yx^2 \rangle$,
 and  $A =\langle ab \rangle =\langle yx  \rangle$. \underline{In
$\Gamma(H)$}:
$(C_0,u_0)=Cayley(G_1)$ and $ VB(C_0) =A(u_0)$;
$ VB(C_1) = VB(C_0) $, $Lab(C_1, v)=\langle x \rangle \nleq A$.
Thus $Type(\Gamma(H))=\Gamma_2=C_1$.}
 \label{fig: type construction Gamma=C}}
\end{center}
\end{figure}

  Since each monochromatic component $C_i$ ($1
\leq i \leq m-1$) is redundant in $\Gamma_i$ (which is the graph
obtained from $\Gamma(H)$ by the iterative  removal of $C_0,
\cdots, C_{i-1}$) w.r.t. some $v \in V( \Gamma_m) \subseteq
V(\Gamma(H))$, we conclude, by Remark~\ref{remark:
RemovingRedundantCom}, that $Lab(\Gamma(H),v)=Lab(\Gamma_m,v)$. We
get (iii).

To prove (iv), assume first that $H  \leq A$.  Therefore, by
Definition~\ref{def:Type} and by Lemma~\ref{lem:FormOfRedPrecover}
(iii), $Type(\Gamma(H))=\Gamma(H)=C_1 \ast_{\{v_1 \cdot a=v_2
\cdot a \; | \; a \in A\}} C_2$, where $(C_i, v_i)=(Cayley(G_i,H
), H \cdot 1)$ ($i \in \{1,2\}$). Without loss of generality,
assume that $v_0 \neq v \in V(C_1)$. Therefore $C_2$ is redundant
w.r.t. $v$. Hence, by Lemma~\ref{lem:PropertiesRedPrecovers},
$$ \Gamma(K) =\left\{%
\begin{array}{ll}
    C_1, & \hbox{$K \cap A=\{1\}$;} \\
    C_1 \ast_{\{v \cdot a=S \cdot a \; | \; a \in A\}} Cayley(G_2,S), & \hbox{$K \cap A=S \neq \{1\}$.} \\
\end{array}%
\right.$$ Therefore
$$Type(\Gamma(K))=\left\{%
\begin{array}{ll}
    C_1, & \hbox{$K \nleq A$;} \\
    C_1 \ast_{\{v \cdot a=K \cdot a \; | \; a \in A\}} Cayley(G_2,K), & \hbox{$K \leq A$.} \\
\end{array}%
\right.$$
Assume now that $H,K \nleq A$. Thus combining
Definition~\ref{def:Type} and Lemma~\ref{lem:FormOfRedPrecover},
we conclude that $Type(\Gamma(H)) $  has no redundant components.
If $(Type(\Gamma(H)), v)$ is a finite reduced precover of $G$
then, by Theorem~\ref{thm: properties of subgroup graphs} (2),
$(Type(\Gamma(H)),v)=(\Gamma(K),u_0)$.

Otherwise, by Lemma~\ref{lem:PropertiesRedPrecovers},
$$(\Gamma(K), u_0)=Type(\Gamma(H)) \ast_{\{v \cdot a=S \cdot a \;
| \; a \in A\}} Cayley(G_j,S),$$
where $S=Lab(C,v)\cap A \neq \{1\}$ and $C$ is a
$X_i$-monochromatic component of $Type(\Gamma(H))$ such that $v
\in V(C)$ ($1 \leq i \neq j \leq 2$).
Since $K \nleq A$,   the component $D= Cayley(G_j,S)$   is
redundant in $\Gamma(K)$. Therefore  $ Type(\Gamma(H) )=
Type(\Gamma(K) )$.
This completes the proof.

\end{proof}


\begin{ex}
{\rm Concerning the subgroups $H_1$ and $H_2$ from Example
\ref{example: graphconstruction} and their subgroup graphs
$\Gamma(H_1)$ and $\Gamma(H_1)$ presented on Figures \ref{fig:
example of H=xy}
 and  \ref{fig: example of H=xy^2x, yxyx}, we compute
that $Type(\Gamma(H_1))=\Gamma(H_1)$ and
$Type(\Gamma(H_2))=\Gamma(H_2)$.} \e
\end{ex}

\medskip


\subsection{Conjugate Subgroups}
\label{subsec:ConjugateSubgroups}

\begin{lem} \label{lem:EqType=>Conjugates}
Let $H$ and $K$ be nontrivial subgroups of $G$  such that
$Type(\Gamma(H))=Type(\Gamma(K))$. Then $H$ is conjugate to $K$ in
$G$.
\end{lem}
\begin{proof}
Suppose that $Type(\Gamma(H))=Type(\Gamma(K))=\Gamma$. Let $v \in
V(\Gamma) \subseteq V(\Gamma(H))$. Hence the subgroup
$Lab(\Gamma(H),v)$ is conjugate to the subgroup
$Lab(\Gamma(H),v_0)$.
By Lemma~\ref{lem:PropertiesType} (iii),
$Lab(\Gamma(H),v)=Lab(\Gamma,v)$. Therefore the subgroup
$Lab(\Gamma,v)$ is conjugate to the subgroup
$Lab(\Gamma(H),v_0)=H$. By symmetric arguments, the subgroup
$Lab(\Gamma,v)$ is also conjugate to the subgroup $K$.
 Hence $H$ is conjugate to $K$. See
 Figure~\ref{fig:CommonType=>Conjugates}.
\end{proof}

\begin{figure}[!h]
 \begin{center}
\psfrag{H }[][]{$\Gamma(H)$} \psfrag{K }[][]{$\Gamma(K)$}
\psfrag{tk }[][]{$Type(\Gamma(H))=Type(\Gamma(K))$}

\psfrag{u }[][]{$u$} \psfrag{v }[][]{$v$}
\psfrag{v0 }[][]{$v_0$}  \psfrag{u0 }[][]{$u_0$}
\psfrag{g1 }[][]{$g_1$}  \psfrag{g2 }[][]{$g_2$} \psfrag{g3
}[][]{$g_3$}
\includegraphics[width=0.8\textwidth]{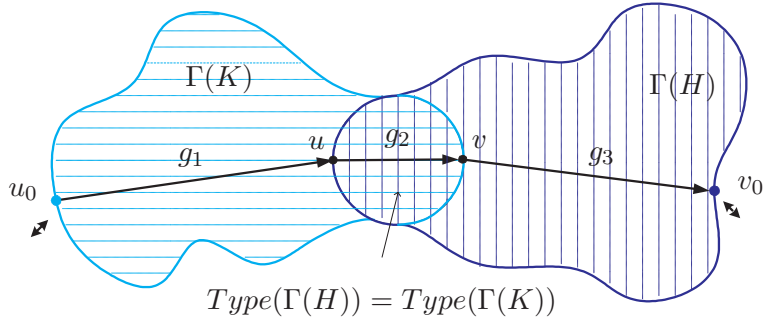}
\caption { {\footnotesize $K=gHg^{-1}$, where $g \equiv
g_1g_2g_3$.}  \label{fig:CommonType=>Conjugates}}
\end{center}
\end{figure}

\begin{thm} \label{thm:Conjugates=>EqType}
Let $H$ and $K$ be finitely generated subgroup of an amalgam of finite groups  $G=G_1 \ast_A G_2$.\\
Then $H$ is conjugate to $K$ in $G$ if and only if one of the
following holds
\begin{itemize}
 \item[(1)] $H, K  \nleq A$  and $Type(\Gamma(K))=Type(\Gamma(H))$.
 \item[(2)] $H  \leq A$, $Type(\Gamma(H))=C_1 \ast_{\{v_1 \cdot a=v_2
\cdot a \; | \; a \in A\}} C_2$ and
$$Type(\Gamma(K))=\left\{%
\begin{array}{ll}
   C_l, & \hbox{$  K \nleq A$;} \\
    C_l \ast_{\{v \cdot a=K \cdot a \; | \; a \in A\}}
Cayley(G_j,K), & \hbox{ $K \leq A$,} \\
\end{array}%
\right.    $$  where $(C_i,v_i)=(Cayley(G_i,H),H \cdot 1)$, for
all $i \in \{1,2\}$, $Lab(C_l,v)=K$, $v \in V(C_l)$ and  $1\leq l
\neq j \leq 2$.
\end{itemize}
\end{thm}
\begin{proof}
If (1) holds then, by Lemma~\ref{lem:EqType=>Conjugates}, $H$ is
conjugate to $K$ in $G$.

Assume that (2) holds and, without loss of generality, assume that
$l=1$.  Thus, by Lemma~\ref{lem:PropertiesType} (iii),
$Lab(\Gamma(K),v)=Lab(Type(\Gamma(K)),v)$. Therefore, by
Lemma~\ref{lem:PropertiesRedPrecovers},
$Lab(Type(\Gamma(K)),v)=Lab(C_1,v)$.
Therefore the subgroup $Lab(C_1,v)$ is conjugate to the subgroup
$Lab(\Gamma(K),u_0)=K$.

On the other hand, $Lab(\Gamma(H),v)=Lab(C_1,v)$, by
Remark~\ref{remark: RemovingRedundantCom}, because $C_2$ is
redundant w.r.t. $v$. Therefore the subgroup $Lab(C_1,v)$ is
conjugate to the subgroup $Lab(\Gamma(H),v_0)=H$. Thus $H$ and $K$
are conjugate subgroups of $G$.

\medskip

Let $\underline{K=g^{-1}Hg} $. Without loss of generality, assume
that $g \in G$ is a normal word. Let $g \equiv g_1g_2$, where
$g_1$ is the maximal prefix of the word $g$ such that there is a
path $p$ in $\Gamma(H)$ with $\iota(p)=v_0$ and $lab(p) \equiv
g_1$, where $v_0$ is the basepoint of $\Gamma(H)$. Let $v=\tau(p)
\in V(\Gamma(H))$. See Figure \ref{fig: TypeConjugacy1}.

If $g_2$ is the empty word then $g \equiv g_1$, and
$Lab(\Gamma(H),v)=g^{-1}Hg=K$.

If $v \in Type(\Gamma(H))$ then, by Lemma~\ref{lem:PropertiesType}
(iv), we are done.


%
 \begin{figure}[!h]
 \begin{center}
\psfrag{A }[][]{{\Large $\Gamma(H)$}}
\psfrag{v }[][]{\small $v$} \psfrag{p }[][]{$p$}
\psfrag{c0 }[][]{$C_0$} \psfrag{c1 }[][]{$C_1$}
\psfrag{c2 }[][]{$C_2$} \psfrag{c3 }[][]{$C_3$}
\psfrag{ci }[][]{$C_i$}
\includegraphics[width=0.5\textwidth]{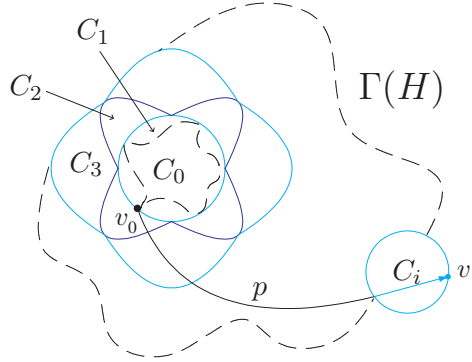}
\caption[The first auxiliary figure for the proof of
Theorem~\ref{conjugate subgroups finite grps}]{{\footnotesize The
closed connected curves represent monochromatic components of
different colors.  The broken curves denote the rest of the
graph.} \label{fig: TypeConjugacy1}}
\end{center}
\end{figure}
%
%

Assume now that $v \not\in V(Type(\Gamma(H))$. Therefore
$Type(\Gamma(H)) \neq \Gamma(H)$. Thus, by
Definition~\ref{def:Type}, $H \nleq G_i$  ($i \in \{1,2\}$).
Without loss of generality, we can assume that
$Type(\Gamma(H))=\Gamma_m$, where $\Gamma_m$ is obtained from
$\Gamma(H)$ by the iterative removal of the unique sequence
($\ast$) of alternating monochromatic components
$$C_0,C_1, \ldots, C_{m-1}.$$

Hence there exists $1 \leq i \leq m-1$ such that $v \in V(C_i)$.
Without loss of generality, we can assume that $C_i$ is a
$X_1$-monochromatic component.
Let $\Gamma_i$ be the graph  obtained from $\Gamma(H)$ by the
iterative  removal of the unique sequence $C_{0}, \ldots,
C_{i-1}.$
By Lemma~\ref{lem:PropertiesRedPrecovers}, we have either $
\Gamma(K) = \Gamma_{i} $ or $$ \Gamma(K) = \Gamma_{i} \ast_{\{v
\cdot a=S \cdot a \; | \; a \in A\}} Cayley(G_2,S),$$ where $S=K
\cap A$.

In the first case, since each component $C_j$ is redundant in
$\Gamma_j$ ($i  \leq j \leq m-1$), $Type(\Gamma(K))$ is obtained
from $\Gamma(K)$ by the iterative  erasure of the unique sequence
of alternating monochromatic components $ C_{i}, \ldots, C_{m-1}$.

In the second case, the component $D=Cayley(G_2,S)$ of $\Gamma(K)$
is redundant. Therefore $Type(\Gamma(K))$ is obtained from
$\Gamma(K)$ by the iterative  erasure of the unique sequence of
alternating monochromatic components $ D, C_{i}, \ldots, C_{m-1}$.
Therefore $K \nleq G_i$ ($i \in \{1,2\}$) and
$$Type(\Gamma(K))=\Gamma_m=Type(\Gamma(H)).$$

\medskip

Assume now that $g_2$ is a nonempty word.

We suppose first that $v \in V(Type(\Gamma(H)))$. Let $\Gamma'$ be
the graph obtained from $Type(\Gamma(H))$ by attaching to this
graph a ``stem'' $q$ at the vertex $v$, such that $lab(q) \equiv
g_2$. Thus $\iota(q)=v$ and we let $\tau(q)=v'$, see Figure
\ref{fig: TypeConjugacy2}.

Obviously, $Lab(\Gamma',v)=Lab(Type(\Gamma(H)),v)$. By
Lemma~\ref{lem:PropertiesType},
$$Lab(\Gamma',v)=Lab(Type(\Gamma(H)),v)=Lab(\Gamma(H),v)=g_1^{-1}Hg_1.$$ Therefore
$Lab(\Gamma', v')=g_2^{-1}Lab(\Gamma',v)g_2=g^{-1}Hg=K$.

 \begin{figure}[!h]
 \begin{center}
\psfrag{A }[][]{{\Large $\Gamma(H)$}}
\psfrag{v }[][]{$v$} \psfrag{p }[][]{$p$}
\psfrag{c0 }[][]{$C_0$} \psfrag{c1 }[][]{$C_1$}
\psfrag{c2 }[][]{$C_2$} \psfrag{c3 }[][]{$C_3$}
\psfrag{ci }[][]{$C_i$}
\psfrag{d1 }[][]{$D_1$} \psfrag{d2 }[][]{$D_2$}
\psfrag{q1 }[][]{$q_1$}

\psfrag{q2 }[][]{$q_2$}
\psfrag{qk }[][]{$q_k$}
\psfrag{v1 }[][]{$v_1$}

\psfrag{v2 }[][]{$v_2$}
\psfrag{vk }[][]{$v_k$}
\psfrag{v3 }[][]{$v_3$}

\psfrag{v' }[][]{$v'$}

\includegraphics[width=\textwidth]{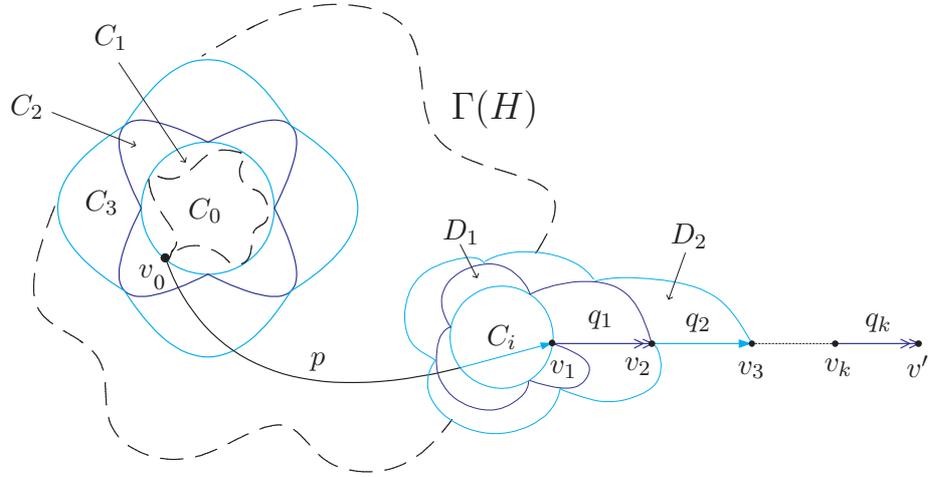}
\caption[The second auxiliary figure for the proof of
Theorem~\ref{conjugate subgroups finite grps}]{{\footnotesize The
closed connected curves represent monochromatic components of
different colors.  The broken curves denote the rest of the
graph.} \label{fig: TypeConjugacy2}}
\end{center}
\end{figure}
%
%

Let $q=q_1 \cdots g_k$ be a decomposition of $q$ into maximal
monochromatic paths. Let $v_i=\iota(q_i)$, $1 \leq i \leq k$. Thus
$v  =v_1$.

Now we need the following result from \cite{m-algI} (given along
with the proof of Claim 2 in \cite{m-algI}).
\begin{claim} \label{claim:PrecoverStem->Precover}
The graph $(\Gamma',v')$ can be embedded into a finite precover
$(\Gamma'',v'')$ of $G$ such that
$$\Gamma''=\left(\left(\left(\Gamma' \ast_{\{v_1 \cdot a | a \in A  \}}
D_1 \right) \ast_{\{v_2 \cdot a | a \in A  \}} D_2 \right) \cdots
\right) \ast_{\{v_{k} \cdot a | a \in A  \}} D_k,$$ where
\begin{itemize}
\item $(D_j,v_j)=Cayley(G_{i_j},S_j)$ ($1 \leq j \leq k$, $1 \leq
i_j \neq i_{j+1} \leq 2$),
\item $S_1=Lab(C,v) \cap A$, where $C$ is a $X_{i_1}$
monochromatic component of $Type(\Gamma(H))$ such that $v \in
V(C)$,
\item $S_{j+1}=Lab(Cayley(G_{i_j},S_j),v_{j+1}) \cap A$ ($1 \leq j
\leq k-1$),
\item the image of $q_j$ in $\Gamma''$ is a path in $D_j$,
\item $v''$ is the image of $v'$ in $\Gamma''$.
\end{itemize}
\end{claim}
Let $\Gamma'_j=\left(\left(\Gamma' \ast_{\{v_1 \cdot a | a \in
A\}} D_1\right)  \cdots \right) \ast_{\{v_{j } \cdot a | a \in A
\}} D_j$, for all $1 \leq j \leq k$. Thus $\Gamma''=\Gamma'_k$.

By Corollary~\ref{cor:AmalgamGraphEmbedding},
$Lab(\Gamma',v_1)=Lab(\Gamma'_1,v_1)$ and
$Lab(\Gamma'_j,v_j)=Lab(\Gamma'_{j-1},v_j )$.\footnote{By abuse of
notation, we identify the vertices $v_j \in V(\Gamma')$ with their
images in the graphs $\Gamma'_j$ ($1 \leq j \leq k$).}
Therefore
\begin{eqnarray}
 Lab(\Gamma'_j,v_{j+1}) & = &
(lab(q_j))^{-1}Lab(\Gamma'_{j},v_j)lab(q_j)=
(lab(q_j))^{-1}Lab(\Gamma'_{j-1},v_j )lab(q_j)   \nonumber \\
& = & (lab(q_1 \cdots q_j))^{-1}Lab(\Gamma'_{1},v_1 )lab(q_1
\cdots q_j) \nonumber \\
& = & (lab(q_1 \cdots q_j))^{-1} Lab(\Gamma',v_1)lab(q_1 \cdots
q_j). \nonumber
\end{eqnarray}
Thus
$$Lab(\Gamma'',v'')=Lab(\Gamma'_k,v'')=g_2^{-1}Lab(\Gamma',v_1)g_2=(g_1g_2)^{-1}H(g_1g_2)=K.$$
Moreover, by the construction, $\Gamma''$ is a precover of $G$
which has no redundant components w.r.t. $v''$. Hence, by
Lemma~\ref{lem:PropertiesRedPrecovers},  either $ \Gamma(K) =
\Gamma'' $ or $$ \Gamma(K) = \Gamma'' \ast_{\{v'' \cdot a=S \cdot
a \; | \; a \in A\}}  Cayley(G_l,S),$$ where $S=K \cap A$ and $ l
=i_{k-1}$.

By the  construction of $\Gamma''$ (see the proof of
Claim~\ref{claim:PrecoverStem->Precover}), $D_k, \ldots, D_1 $ is
the unique sequence of redundant components in $\Gamma''$ which
satisfies the conditions (1)-(3) from the description of sequence
($\ast$) (see the proof of Lemma~\ref{lem:DefGammaM}). Therefore,
in the first case, it should be erased from $\Gamma(K)$ along the
construction of $Type(\Gamma(K))$.

In the second case, the component $D=Cayley(G_2,S)$ of $\Gamma(K)$
is redundant. Therefore the sequence   $D, D_k, \ldots, D_1$
should be erased from $\Gamma(K)$ along the construction of
$Type(\Gamma(K))$.
Therefore, by Definition~\ref{def:Type}, $K \neq G_i$ ($i \in
\{1,2\}$). Moreover, if $Type(\Gamma(H))$ has no redundant
components, that is $H \nleq A$ then
$Type(\Gamma(K))=Type(\Gamma(H))$ if $Type(\Gamma(H))$. Otherwise
$$Type(\Gamma(H))=\Gamma(H)=C_1 \ast_{\{v_1 \cdot a=v_2 \cdot a \;
| \; a \in A\}} C_2,$$ where $(C_i,v_i)=(Cayley(G_i,H),H \cdot
1)$, for all $i \in \{1,2\}$. Without loss of generality assume
that $v \in V(C_1)$. Hence $Type(\Gamma(K))=C_1$.

If $v \not\in V(Type(\Gamma(H)))$ then $H \nleq G_i$ ($i \in
\{1,2\}$). We take $\Gamma'$ to be the graph obtained by gluing a
stem $q$ labelled by $g_2$ at $v \in V(C_i)$ to the graph
$\Gamma_i$, which is obtained from $\Gamma(H)$ by the  iterative
removal of the sequence $C_{0}, \ldots, C_{i-1} $ of redundant
components in $\Gamma(H)$. Combining the proofs of two previous
cases, namely $v \not\in V(Type(\Gamma(H)))$, $g_2=_G 1$ and $v
\in V(Type(\Gamma(H)))$, $g_2 \neq_G 1$, we conclude that $K \nleq
G_i$ ($i \in \{1,2\}$), and $Type(\Gamma(K))=Type(\Gamma(H))$.

 \end{proof}


\begin{cor} \label{algorithm : conjugate subgroups}
Let $h_1, \ldots ,h_s,k_1, \ldots, k_t \in G$. Then there is an
algorithm, which decides whether or not the subgroups
$$H=\langle h_1, \ldots ,h_s \rangle   \ \ {\rm and}  \ \ K=\langle k_1, \ldots,
k_t \rangle    \ \ (i \in \{1,2\})$$ are conjugate in $G$.
\end{cor}
\begin{proof} First we construct the graphs $(\Gamma(H),v_0)$ and
$(\Gamma(K),u_0)$, using the generalized Stallings' algorithm.
Then we compute $Type(\Gamma(H))$ and $Type(\Gamma(K))$ according
to the definition of $Type$. Now we verify if any of the
conditions from Theorem~\ref{thm:Conjugates=>EqType} are
satisfied.

Note that the verification of $\Delta_1=\Delta_2$ actually means
to check whether or not $\Gamma_1$ and $\Gamma_2$ are isomorphic.
This can be done by fixing a vertex $v \in V( \Delta_1)$ and
comparing for each vertex $w \in V(\Delta_2)$ the pointed graphs
$(\Delta_1,v)$ and $(\Delta_2, w)$, because by Remark \ref{unique
isomorphism}, such an isomorphism if it exists is unique.

Since morphisms of well-labelled graphs preserves endpoints and
labels, we can specify the above verification by fixing a
bichromatic vertex $v \in VB(\Delta_1)$ and comparing  the pointed
graphs $(\Delta_1,v)$ and $(\Delta, w)$, for each bichromatic
vertex $w \in VB(\Gamma_2)$.

\end{proof}

\begin{ex} {\rm
The subgroups $H_1$ and $H_2$ from Example \ref{example:
graphconstruction} (see Figures \ref{fig: example of H=xy}
 and  \ref{fig: example of H=xy^2x, yxyx}) are not
conjugate to each other, because their $Type$ graphs are not
isomorphic. Indeed, $Type(\Gamma(H_i))=\Gamma(H_i)$ for $i \in
\{1,2\}$,  but $|V(\Gamma(H_1))| \neq |V(\Gamma(H_2))|$. Hence
these graphs can not be isomorphic.} \e
\end{ex}


\subsection*{Complexity}

Let $m$ be the sum of the lengths of the words $h_1, \ldots h_s$,
and let $l$ be the sum of the lengths of the words $k_1, \ldots,
k_t$. By Theorem~\ref{thm: properties of subgroup graphs} (5), the
complexity of the construction of the graphs $\Gamma(H)$ and
$\Gamma(K)$ is $O(m^2)$ and $O(l^2)$, respectively.

The detecting of monochromatic components in the constructed
graphs takes $ \: O(|E(\Gamma(H))|) \: $ and $ \:
O(|E(\Gamma(K))|) \: $, that is  $O(m)$ and $O(l)$, respectively.
Since all the essential information about $A$, $G_1$ and  $G_2$ is
given and it is not a part of the input, verifications concerning
a particular monochromatic component of $\Gamma(H)$ or of
$\Gamma(K)$ takes $O(1)$. Therefore,   the complexity of the
construction of $Type(\Gamma(H))$ from $\Gamma(H)$ is
$O(|E(\Gamma(H))|)$, that is $O(m )$. Similarly, the complexity of
the construction of $Type(\Gamma(K))$ from $\Gamma(K)$ is
$O(|E(\Gamma(K))| )$, that is $O(l)$.

Now we are ready to verify an isomorphism of the obtained type
graphs. We can start by comparing the sizes of $V(\Gamma_1)$ and
$V(\Gamma_2)$ and of $E(\Gamma_1)$ and $E(\Gamma_2)$. If
$|V(\Gamma_1)|=|V(\Gamma_2)|$ and $|E(\Gamma_1)|=|E(\Gamma_2)|$
then we continue. Otherwise the graphs are not isomorphic.


Let $\Gamma_1=Type(\Gamma(H))$ and $\Gamma_2=Type(\Gamma(K))$. Let
$v \in VB(\Gamma_1)$ and $w \in VB(\Gamma_2)$. Thus, by
Definition~\ref{def:ReducedPrecover}, $(\Gamma_1,v)$ and
$(\Gamma_2,w)$ are finite reduced precovers of $G$.

Theorem~\ref{thm: properties of subgroup graphs} (2) implies that
the finite reduced precovers $(\Gamma_1,v)$ and $(\Gamma_2,w)$ are
isomorphic if and only if they are isomorphic via the morphism
$\mu$ of well-labelled pointed graphs, defined in the proof of
Lemma~\ref{lemma1.5} in \cite{gi_sep}. That is to check the
isomorphism between $(\Gamma_1,v)$ and $(\Gamma_2,w)$, we simply
have to check if $\mu$ is defined. Recall that $\mu:(\Gamma_1,v)
\rightarrow (\Gamma_2,w)$  satisfies
$$
\mu(\vartheta)=w \cdot x  \  \big( \forall \; \vartheta=v \cdot x
\in V(\Gamma_1) \big) \  \ {\rm and} \ \
\mu(e)=(\mu(\iota(e)),lab(e)) \ \big( \forall \; e \in
 E(\Gamma) \big).$$
Thus for all $\; \vartheta=v \cdot x \in V(\Gamma_1) \;$ we have
to check if $Star(v \cdot x, \Gamma_1)=Star(w \cdot x, \Gamma_2)$,
where the \emph{star} of the vertex $\sigma$ (see \cite{stal}) in
the graph $\Delta$ is the set $$Star(\sigma,\Delta)=\{e \in
E(\Delta) \; | \; \iota(e)=\sigma \}.$$
This procedure takes time proportional to $|E(\Gamma_1)|$, that is
proportional to $m$. Since in the worst case we have to repeat the
above procedure for all pointed graphs $(\Gamma_2,\omega)$, where
$\omega \in VB(\Gamma_2)$, the verification  of an isomorphism
between the graphs $\Gamma_1$ and $\Gamma_2$ takes $O
\big(|VB(\Gamma_2)| \cdot |E(\Gamma_1)| \big).$

Since $|VB(\Gamma_2)| \leq |V(\Gamma_2)|$ and, by
Theorem~\ref{thm: properties of subgroup graphs} (5),
$|V(\Gamma_2)|$ is proportional to $l$ and $|E(\Gamma_1)|$ is
proportional to $m$, the complexity of the algorithm given along
with the proof of Corollary \ref{algorithm : conjugate subgroups}
is $$O \big(m^2+l^2+ml \big) \; = \; O \big((m+l)^2 \big).$$ Thus
the above algorithm is quadratic in the size of the input.

Note that if the subgroups $H$ and $K$ are given by the  graphs
$\Gamma(H)$ and $\Gamma(K)$,  the complexity of the algorithm that
decides whether or not the subgroup $H$ and $K$ are conjugate in
$G$ is
$$O \big( |E(\Gamma(H))|^2+|E(\Gamma(K))|^2+ |VB(\Gamma_2)| \cdot
|E(\Gamma_1)| \big).$$

Note that, since our graphs are connected,  $|V(\Gamma_2)| \leq
|E(\Gamma_2)|$. Thus $|VB(\Gamma_2)|  \leq |V(\Gamma_2)| \leq
|E(\Gamma_2)| \leq |E(\Gamma(K))|$. Since $|E(\Gamma_1)| \leq
|E(\Gamma(H))|$, the complexity is
$$O \big( |E(\Gamma(H))|^2+|E(\Gamma(K))|^2+ |E(\Gamma(H))| \cdot
|E(\Gamma(K))| \big).$$ That is it is quadratic in the size of the
input $O \big( \big( |E(\Gamma(H))|+|E(\Gamma(K))| \big)^2 \big).$

%
%
\bigskip

\subsection{More Conjugacy Results}

\begin{thm} \label{K conjugate to a subgroup of H finite grps}
Let $H$ and $K$ be finitely generated subgroup of an amalgam of finite groups  $G=G_1 \ast_A G_2$.\\
Then there exists $g \in G$ such that $gKg^{-1} \leq H$ if and
only if one of the following holds.
\begin{itemize}
 \item[(1)] If $K  \nleq A$ then there exists a morphism of well-labelled graphs $$\pi:
Type(\Gamma(K)) \rightarrow Type(\Gamma(H)).$$
 \item[(2)] If $K  \leq A$ and $Type(\Gamma(K))=C_1 \ast_{\{v_1 \cdot a=v_2
\cdot a \; | \; a \in A\}} C_2$, where $(C_i,v_i)=(Cayley(G_i,K),K
\cdot 1)$ ($i \in \{1,2\}$), then there exists  a morphism of
well-labelled graphs $\pi: C_l \rightarrow Type(\Gamma(H)),$ for
some $l \in \{1,2\}$.
\end{itemize}

\end{thm}
\begin{proof} We begin with the following claim which allows to
assume that $\Gamma(H)=Type(\Gamma(H))$.
\begin{claim}
There exist $v \in V(\Gamma(H))$ such that
$Type(\Gamma(L))=\Gamma(L)$, where $L=Lab(Type(\Gamma(H)),v)$.
\end{claim}
\begin{proof}[Proof of the Claim] \
By Definition~\ref{def:Type}, either $Type(\Gamma(H))=\Gamma(H)$
or \linebreak[4] $Type(\Gamma(H))~=~\Gamma_m.$

If $\Gamma_m=C_m$ then we take $v \in
VB_{\Gamma_{m-1}}(C_{m-1})=VB_{\Gamma_{m-1}}(C_m)$. Thus
$L=Lab(C_m,v) \leq G_i$ ($i\in\{1,2\}$), and, by the proof of
Lemma~\ref{lem:DefGammaM},  $L  \nleq A$. Hence, by
Lemma~\ref{lem:FormOfRedPrecover}, $\Gamma(L)=C_m$. Thus, by
Definition~\ref{def:Type}, $\Gamma(L)=Type(\Gamma(L))$.

Assume now that $\Gamma_m$ has at least two distinct monochromatic
components. Let $v \in VB(\Gamma_m)$. By
Lemma~\ref{lem:PropertiesType}(i), $\Gamma_m$ is a finite nonempty
precover of $G$. Thus, since $\Gamma_m$ has no redundant
components and $v \in VB(\Gamma_m)$, we conclude that $\Gamma_m$
is a finite reduced precover. That is $(\Gamma_m,v)=(\Gamma(L),
u_0)$, where $u_0$ is the basepoint of $\Gamma(L)$. Therefore
$Type(\Gamma(L))=\Gamma(L)$.

\end{proof}

Let $y \in G$ such that $v=v_0 \cdot y$. Therefore $L=y^{-1}Hy$.
Thus $y^{-1}gK g ^{-1}y \leq L$. Hence, without loss of
generality,  we can assume that $\Gamma(H)=Type(\Gamma(H))$.

Assume first that there exists $g \in G$ such that $gKg^{-1} \leq
H$. By Lemma~\ref{morphism of graphs}, there exists a morphism
$\varphi: \Gamma(gKg^{-1}) \rightarrow \Gamma(H).$ Let $\varphi'$
be the restriction of $\varphi$ to $Type(\Gamma(gKg^{-1}))$, that
is $\varphi': Type(\Gamma(gKg^{-1})) \rightarrow Type(\Gamma(H)).$

Let $K  \nleq A$. Thus either $gKg^{-1}  \nleq A$ or $gKg^{-1}
\leq A$. Hence, by Theorem~\ref{thm:Conjugates=>EqType}, either
$Type(\Gamma(K))=Type(\Gamma(gKg^{-1}))$  or,
$$Type(\Gamma(gKg^{-1}))=C_1 \ast_{\{v_1 \cdot a=v_2 \cdot a \; |
\; a \in A\}} C_2, \  \ \ Type(\Gamma(K))=C_l  \      (l \in
\{1,2\}),$$  where $(C_i,v_i)=(Cayley(G_i,gKg^{-1}),gKg^{-1} \cdot
1)$, $i \in \{1,2\}$. In the first case, we take $\pi=\varphi'$.
In the second one,  there exists an embedding
$\phi:Type(\Gamma(K)) \rightarrow Type(\Gamma(gKg^{-1}))$.
Therefore $\pi=\varphi' \circ \phi$ gives the desired morphism.

%
%

Assume now that  $K  \leq A$. Thus, by Definition~\ref{def:Type}
and Lemma~\ref{lem:FormOfRedPrecover} (iii), $Type(\Gamma(K))=C_1
\ast_{\{v_1 \cdot a=v_2 \cdot a \; | \; a \in A\}} C_2$, where
$(C_i,v_i)=(Cayley(G_i,K),K \cdot 1)$ ($i \in \{1,2\}$).

If $gKg^{-1}  \nleq A$ then, by
Theorem~\ref{thm:Conjugates=>EqType},
$Type(\Gamma(gKg^{-1}))=C_l$, for some $l \in \{1,2\}$. Thus
$\pi=\varphi'$ produces the desired morphism $\pi: C_l \rightarrow
Type(\Gamma(H)).$
If $gKg^{-1}  \leq A$ then, by
Theorem~\ref{thm:Conjugates=>EqType}, $$Type(\Gamma(gKg^{-1}))=C_l
\ast_{\{v \cdot a=(gKg^{-1}) \cdot a \; | \; a \in A\}}
Cayley(G_j,gKg^{-1}),$$ where $1\leq l \neq j \leq 2$. Thus there
exists an embedding  $\phi:C_l \rightarrow
Type(\Gamma(gKg^{-1}))$. Therefore $\pi=\varphi' \circ \phi: C_l
\rightarrow Type(\Gamma(H))$ gives the desired morphism.


Suppose now that $ K \nleq A$ and the morphism $\pi:
Type(\Gamma(K)) \rightarrow Type(\Gamma(H))$ exists. Let $p$ be a
path in $\Gamma(K)$ with $\iota(p)=u_0$, where $u_0$ is the
basepoint of the graph $\Gamma(K)$, such that $ \tau(p) \in
V(Type(\Gamma(K))) $. Let $u=\tau(p)$, $lab(p)\equiv f$ and let
$\vartheta=\pi(u) \in V(Type(\Gamma(H)))$.

Since $Type(\Gamma(H)) \subseteq \Gamma(H)$, we have $\vartheta
\in V(\Gamma(H))$. By Lemma~\ref{lemma2.12}, there exists a normal
path $q$ in $\Gamma(H)$ with $\iota(q)=v_0$ (the basepoint of
$\Gamma(H)$) and $\tau(q)=\vartheta$. Let $lab(q) \equiv c$. By
Lemma~\ref{lem:PropertiesType} (iii), $Lab(Type(\Gamma(K)),u)=Lab(
\Gamma(K) ,u)=f^{-1}Kf$ and $Lab(Type(\Gamma(H)),\vartheta)=Lab(
\Gamma(H) ,\vartheta)=c^{-1}Hc$.

Since $\pi$ can also be considered as a morphism of pointed graphs
$$\pi: (Type(\Gamma(K)),u) \rightarrow (Type(\Gamma(H)),\vartheta),$$
by Lemma~\ref{morphism of graphs}, we have $Lab
(Type(\Gamma(K)),u) \leq Lab(Type(\Gamma(H)),\vartheta)$. Thus
$f^{-1}Kf \leq c^{-1}Hc$. Therefore $g=c \cdot f^{-1}$ and
$gKg^{-1} \leq  H $, as required.


Let $K \leq A$. Thus, by Lemma~\ref{lem:FormOfRedPrecover} (iii),
$ \Gamma(K) =C_1 \ast_{\{v_1 \cdot a=v_2 \cdot a \; | \; a \in
A\}} C_2$, where $(C_i,v_i)=(Cayley(G_i,K),K \cdot 1)$ ($i \in
\{1,2\}$) and the basepoint $u_0$ of $\Gamma(K)$ is the image of
$K \cdot 1$. Thus $Lab(C_1,u_0 \cdot a)=Lab(C_2,u_0 \cdot
a)=Lab(\Gamma(K),u_0 \cdot a)$, for all $a \in A$.

Assume that  there is $l \in \{1,2\}$ such that the morphism $\pi:
C_l \rightarrow Type(\Gamma(H))$ exists. Let $p$ be a path in
$\Gamma(K)$ with $\iota(p)=u_0$, where $u_0$ is the basepoint of
the graph $\Gamma(K)$, such that $ \tau(p) \in V(C_l) $. Let
$u=\tau(p)$, $lab(p)\equiv f$ and let $\vartheta=\pi(u) \in
V(Type(\Gamma(H)))$.

Since $Type(\Gamma(H)) \subseteq \Gamma(H)$, we have $\vartheta
\in V(\Gamma(H))$. By Lemma~\ref{lemma2.12}, there exists a normal
path $q$ in $\Gamma(H)$ with $\iota(q)=v_0$ (the basepoint of
$\Gamma(H)$) and $\tau(q)=\vartheta$. Let $lab(q) \equiv c$. Thus
$Lab(C_l,u)=f^{-1}Lab(C_l,u_0)f=f^{-1}Lab(\Gamma(K),u_0)f=f^{-1}Kf$
and $Lab(Type(\Gamma(H)),\vartheta)=Lab( \Gamma(H)
,\vartheta)=c^{-1}Hc$.

Since $\pi$ can also be considered as a morphism of pointed graphs
$$\pi: (C_l,u) \rightarrow (Type(\Gamma(H)),\vartheta),$$
by Lemma~\ref{morphism of graphs}, we have $Lab (C_l,u) \leq
Lab(Type(\Gamma(H)),\vartheta)$. Thus $f^{-1}Kf \leq c^{-1}Hc$.
Therefore $g=c \cdot f^{-1}$ and $gKg^{-1} \leq  H $, as required.

 \end{proof}



\begin{cor} \label{cor: alg K is conj subgroup of H}
Let $ h_1, \ldots, h_s, k_1, \ldots, k_t  \in G.$ Then there
exists an algorithm which decides whether or not there exists $g
\in G$ such that $gKg^{-1} \leq  H $, where $$H=\langle h_1,
\ldots ,h_s \rangle   \ {\rm and}  \ \ K=\langle k_1, \ldots, k_t
\rangle .$$ Moreover, the algorithm produces one such $g$ if it
exists.
\end{cor}
\begin{proof}
First we construct the graphs $(\Gamma(K),u_0)$ and
$(\Gamma(H),v_0)$, using the generalized Stallings' folding
algorithm. Then we construct $Type(\Gamma(K))$ and
$Type(\Gamma(H))$, according to Definition~\ref{def:Type}.

If $K \nleq A$ then we proceed as follows. Let $u \in
V(Type(\Gamma(K)))$. For each vertex $v \in V(Type(\Gamma(H)))$ we
iteratively check if there exists a morphism $\pi:
(Type(\Gamma(K)), u) \rightarrow (Type(\Gamma(H)),v)$. If no such
morphism can be found then $K$ is not conjugate to any subgroup of
$H$, by Theorem~\ref{K conjugate to a subgroup of H finite grps}.
Otherwise, by the proof of Theorem~\ref{K conjugate to a subgroup
of H finite grps}, $gKg^{-1} \leq H$, where $g=c \cdot f^{-1}$ and
$v=v_0 \cdot f$ and $u=u_0 \cdot c$.

Assume now that  $K  \leq A$. Thus $Type(\Gamma(K))=C_1
\ast_{\{v_1 \cdot a=v_2 \cdot a \; | \; a \in A\}} C_2$, where
$(C_i,v_i)=(Cayley(G_i,K),K \cdot 1)$ ($i \in \{1,2\}$).

For each $i \in \{1,2\}$, let $u_i \in V(C_i)$. For each vertex $v
\in V(Type(\Gamma(H)))$ we iteratively check if there exists a
morphism $\pi: (C_i, u_i) \rightarrow (Type(\Gamma(H)),v)$. If no
such morphism can be found then $K$ is not conjugate to any
subgroup of $H$, by Theorem~\ref{K conjugate to a subgroup of H
finite grps}. Otherwise, by the proof of Theorem~\ref{K conjugate
to a subgroup of H finite grps}, $gKg^{-1} \leq H$, where $g=c
\cdot f^{-1}$ and $v=v_0 \cdot f$ and $u=u_0 \cdot c$.

\end{proof}

%

\subsection*{Complexity.} Similarly to the complexity analysis of
the algorithm presented along with the proof of Corollary \ref
{cor: alg K is conj subgroup of H},  the complexity of the above
algorithm is $O((m+l) ^2)$, where $m$ is the sum of the lengths of
the words $h_1, \ldots h_s$, and  $l$ is the sum of the lengths of
the words $k_1, \ldots, k_t$.
Similarly, when  the subgroup $H$ and $K$ are given by the graphs
$\Gamma(H)$ and $\Gamma(K)$,  the  complexity  is $O \big( \big(
|E(\Gamma(H))|+|E(\Gamma(K))| \big)^2 \big).$


\begin{cor}[The Conjugacy Problem]
The conjugacy problem is solvable in amalgams of finite groups.

Namely, let $G=G_1 \ast_A G_2$ be an amalgam of finite groups.
Given elements $k, h \in G$ one can decide whether exists $g \in
G$ such that $gkg^{-1}=_G h$.
\end{cor}
\begin{proof}
Let $K=\langle k \rangle$ and $H=\langle h \rangle$.  We apply to
$K$ and $H$ the algorithm described along with the proof of
Corollary~\ref{cor: alg K is conj subgroup of H}. If there is no
$g \in G$ such that $gKg^{-1} \leq H$ then the elements $k$ and
$h$ are not conjugate in $G$.

Otherwise, let $g \in G$ such that $gKg^{-1} \leq H$. We have to
check whether $gkg^{-1}=_G h$. To this end we rewrite the element
$gkg^{-1}h^{-1}$ as a normal word. If the resulting word is not
empty then, by the Torsion Theorem (IV.2.7, \cite{l_s}),
$gkg^{-1}h^{-1} \neq_G 1$, that is $gkg^{-1}\neq_G h$. Otherwise,
$gkg^{-1}=_G h$.
\end{proof}


\section{The Normality Problem}
\label{subsec:Normality_finite grps}

The current section is devoted to the solution of the
\emph{normality problem}, which asks to know if a subgroup $H$ of
a group $G$ is normal in $G$, for finitely generated subgroups of
amalgams of finite groups.

The quadratic time algorithm  is presented in Corollary \ref{cor:
algm normal subgroup}. It is based on Theorem~\ref{normality} and
Lemma~\ref{normalityII}, which give  a connection between the
normality of a subgroup $H$ of an amalgam  of finite groups $G=G_1
\ast_A G_2$ and its subgroup graph  $\Gamma(H)$ constructed by the
generalized Stallings' algorithm. The complexity analysis of the
algorithm is given at the end of this section.


We start by presenting the following technical result from
\cite{m-foldings}, which is essential for the proof of
Theorem~\ref{normality}.


\begin{lem}[Lemma 6.10 and Remark 6.11 in
\cite{m-foldings}]   \label{lem: normal elements in precovers}
Let $G=G_1 \ast_A G_2$ be an amalgam of finite groups. Let
$(\Gamma,v)$ be a finite precover of $G$ such that
$Lab(\Gamma,v)=_G H \neq \{1\}$.

Let $w \in H$ be a normal word. Then $w$ labels a path in $\Gamma$
closed at $v$ if one of the following holds
\begin{itemize}
 \item $l(w)>1$,
 \item $l(w)=1$ and $w \in G_i \setminus A$ ($i \in \{1,2\}$),
 \item $l(w)=1$, $w \in G_i \cap A$ and, $v \in
 VB(\Gamma)$ or $v\in VM_i(\Gamma)$.
\end{itemize}
Otherwise, if $w \in G_i \cap A$ and $v\in VM_j(\Gamma)$ ($1 \leq
i \neq j \leq 2$), then there exists a path in $\Gamma$ closed at
$v$ and labelled with $w'$ such that $w' \in G_j \cap A$, $ w=_G
w'.$

\end{lem}



\begin{thm} \label{normality}
Let $H \leq G$ be a nontrivial subgroup of $G$ such that $H
\not\leq G_i$  for all $i \in \{1,2\}$. Then $H$ is normal in $G$
if and only if the following holds.
\begin{itemize}
 \item [(i)] The graph $\Gamma(H)$ is $X^{\pm}$-saturated.
 \item [(ii)] For all vertices $v,u \in V(\Gamma(H))$, the graphs
 $(\Gamma(H),v)$ and  $(\Gamma(H),u)$ are isomorphic.
\end{itemize}
\end{thm}
\begin{proof} Suppose first that conditions (i) and (ii) are satisfied.

Let $g$ be an element of $G$. Since $\Gamma(H)$ is
$X^{\pm}$-saturated, there exists a path $p$ in $\Gamma(H)$ such
that $\iota(p)=v_0$ and $lab(p) \equiv g$. Let $v =\tau(p)$.
Condition (ii) implies $Lab(\Gamma(H),v_0)=Lab(\Gamma(H),v)$. Thus
$H=g^{-1}Hg$, for all $g \in G$. Hence $H \unlhd G$.

Assume now that $\{1\} \neq H \unlhd G$. Then $\Gamma(H)$ is
$X^{\pm}$-saturated. Otherwise, without loss of generality, we can
assume that there exists $v \in VM_1(\Gamma(H))$. Let $C$ be the
$X_1$-monochromatic component of $\Gamma(H)$ such that $v \in
V(C)$.

Let $q$ be the approach path in $\Gamma(H)$ from $v_0$ to $v$ with
$lab(q) \equiv g$. Thus
$$Lab(\Gamma(H),v)=gLab(\Gamma(H),v_0)g^{-1}=gHg^{-1}=H .$$

Since $(\Gamma(H), v)$ is a precover of $G$, each normal element
of $H$, whose syllable length is greater than 1, labels a normal
path in $(\Gamma(H), v)$ closed at $v$, by Lemma~\ref{lem: normal
elements in precovers}.

Let $h \in H$ has the normal decomposition $(h_1, \ldots, h_k)$.
Thus   $k>1$,   since $H \nleq G_i$ ($i \in \{1,2\}$).  Let $p$ be
a normal path in $\Gamma(H)$ such that
$$\iota(p)=\tau(p)=v , \ p=p_1\cdots p_k, \ {\rm where} \ lab(p_l) \equiv h_l, \ 1 \leq l \leq k.$$
Thus $h_1,h_k \in G_1 \setminus A$, because $v \in
VM_1(\Gamma(H))$. Hence $p_kp_1$ is a path in $C$ from $
\iota(p_k)$ to $ \tau(p_1)$, and we have $lab(p_kp_1) \equiv
h_kh_1 \in G_1$.

If $h_kh_1 \not\in A$ then the decomposition $(h_2, \ldots,
(h_kh_1))$ is normal. Moreover,  $h_1^{-1}hh_1=_G h_2 \cdots
(h_kh_1) \in H$, because $H\unlhd G$. Therefore,  by
Lemma~\ref{lem: normal elements in precovers}, there exists a
normal path in $\Gamma(H)$ closed at $v $ and labelled with $h_2
\cdots h_{k-2}(h_kh_1)$. However, this is impossible because $v
\in VM_1(\Gamma(H))$ and $h_2  \in G_2 \setminus A$. We get a
contradiction.

If $h_kh_1  \in G_1 \cap A$, we take $b=_G h_kh_1$ such that $b
\in G_2 \cap A$. Thus the decomposition $(h_2, \ldots, (h_{k-1}b))
$ is normal, since $h_{k-1}b \in G_2 \setminus A$. We get a
contradiction in the similar way.

Therefore the graph $\Gamma(H)$ is $X^{\pm}$-saturated. Moreover,
$$Lab(\Gamma(H),v)=gLab(\Gamma(H),v_0)g^{-1}=gHg^{-1}=H,$$
where $g \equiv lab(q)$ and $q$ is an approach path in $\Gamma(H)$
from $v_0$ to $v$.

Thus, by Lemma~\ref{lemma1.5}, $(\Gamma(H),v)$ is isomorphic to
$(Cayley(G,H),H \cdot 1)$, for all $v \in V(\Gamma(H))$. Therefore
the graphs $(\Gamma(H),v )$ and $(\Gamma(H),u)$ are isomorphic,
for all vertices $v, u \in V(\Gamma(H))$.

\end{proof}


\begin{lem} \label{normalityII}
Let $H $ be a nontrivial subgroup of $G$ such that $H  \leq G_i$
($i \in \{1,2\}$). The following holds.
\begin{itemize}
 \item[(i)]
If $H\unlhd G$ then $H \leq A$.
 \item[(ii)] If $H \leq A$ then $H\unlhd G$ if and only if each monochromatic
component $C$ of $\Gamma(H)$ is a \underline{regular graph}, that
is $(C,v)$ is isomorphic to $(C,u)$, for all $v, u \in V(C)$.
\end{itemize}
\end{lem}
\begin{proof}
To prove (i) suppose that there exists $h \in H \setminus A$. Let
$g \in G_j \setminus A$, where $1 \leq i \neq j \leq 2$. Therefore
$ghg^{-1}$ is a normal word of syllable length 3. Hence $ghg^{-1}
\not\in H$. This contradicts with the assumption that $H\unlhd G$.

Since $H \leq A$, $\Gamma(H) = Cayley(G_1,H) \ast_{\{Ha \; | \; a
\in A\}} Cayley(G_2,H)$,  by Lemma~\ref{lem:FormOfRedPrecover}
(iii). Since $H \leq A$, $H\unlhd G$ if and only if $H \unlhd G_i$
($i \in \{1,2\}$). Therefore $H\unlhd G$ if and only if the graphs
$ Cayley(G_1,H) $ and $ Cayley(G_2,H) $ are regular (see 2.2.7 in
\cite{stil}).
\end{proof}

Recall the following result from \cite{m-algI}.
\begin{thm}[Theorem~{7.1} in \cite{m-algI}] \label{fi}
Let $H$ be a finitely generated subgroup of an amalgam of finite
groups $G=G_1 \ast G_2$.

Then $[G:H] < \infty$ if and only if $\Gamma(H)$ is
$X^{\pm}$-saturated.
\end{thm}

\begin{remark}
By Theorem~\ref{fi}, $H \unlhd G$ implies $[G:H] < \infty$.
\end{remark}

\begin{cor} \label{cor: algm normal subgroup} Let $h_1, \ldots h_k \in
G$. Then there exists an algorithm which decides whether or not
$H=\langle h_1, \ldots, h_k \rangle$   is a normal subgroup (of
finite index) in $G$.
\end{cor}
\begin{proof}
We first construct the graph $\Gamma(H)$ using the generalized
Stallings' algorithm.

If the number of monochromatic components of $\Gamma(H)$ is equal
to $1$ then, by Lemma~\ref{lem:FormOfRedPrecover} (ii), $H \leq
G_i$ and $H \cap A=\{1\}$ ($i\in \{1,2\}$). Hence, by
Lemma~\ref{normalityII} (i), $H$ is not normal in $G$.

If the number of distinct monochromatic components of $\Gamma(H)$
is equal 2 and
$$(\Gamma(H), v_0)=(Cayley(G_1,H) \ast_{\{Ha \; | \; a \in A\}}
Cayley(G_2,H), \vartheta),$$  where $\vartheta$ is the image of $H
\cdot 1$ in the amalgam graph, then $H  \leq A$, by
Lemma~\ref{lem:FormOfRedPrecover} (iii). Thus, by
Lemma~\ref{normalityII} (ii), $H$ is  normal in $G$ if and only if
$Cayley(G_1,H)$ and $Cayley(G_2,H)$ are regular graphs, that is if
and only if $(Cayley(G_i,H),v)$ is isomorphic to $(Cayley(G_i,H),H
\cdot 1)$, for all $v  \in V(Cayley(G_i,H))$, $i \in \{1,2\}$.
Since an isomorphism of pointed labelled  graphs is unique, by
Remark~\ref{unique isomorphism}, we are done.

Assume now that $H \not\leq G_i$, Thus we verify if $\Gamma(H)$ is
$X^{\pm}$-saturated. If it is not then, by
Theorem~\ref{normality}, $H$ is not normal in $G$. Otherwise, if
$\Gamma(H)$ is $X^{\pm}$-saturated, then, by
Theorem~\ref{normality}, $H$ is  normal in $G$ if and only if  the
graphs
 $(\Gamma(H),v_0)$ and  $(\Gamma(H),v)$ are isomorphic, for each vertex $v \in V(\Gamma(H))$.
Since an isomorphism of pointed labelled  graphs is unique, by
Remark~\ref{unique isomorphism}, we are done.

\end{proof}

\begin{ex}
{\rm Let $H_1$ and $H_2$ be subgroups  from Example \ref{example:
graphconstruction}. One can easily verify from Figures \ref{fig:
example of H=xy}
 and  \ref{fig: example of H=xy^2x, yxyx}  that $H_1$ is not normal in $G$, because $\Gamma(H_1)$ is
not $\{x,y\}^{\pm}$-saturated, while $H_2\unlhd G$.} \e
\end{ex}


\subsection*{Complexity} By  Theorem~\ref{thm: properties of subgroup graphs} (5), the complexity of
the construction of $\Gamma(H)$  is $O(m^2)$, where $m$ is the sum
of lengths of the given subgroup generators.

The detecting of monochromatic components in the constructed graph
takes $ \: O(|E(\Gamma(H))|) \: $, that is  $O(m)$, by
Theorem~\ref{thm: properties of subgroup graphs} (5).

Since all the essential information about the amalgam $G=G_1
\ast_A G_2$, $A$, $G_1$ and $G_2$ is given and it is not a part of
the input, the verifications concerning  monochromatic components
of $\Gamma(H)$ takes $O(1)$. Therefore, to check from $\Gamma(H)$
whether $H \leq G_i$, or $H \leq A$ and the monochromatic
components of $\Gamma(H)$ are regular graphs, takes $O(1)$.

To verify that all the vertices of $\Gamma(H)$ are bichromatic
takes $O(|E(\Gamma(H))|)$. The verification of an isomorphism of
the graphs $(\Gamma(H),v_0)$ and $(\Gamma(H),v) $, for all $v \in
V(\Gamma)$, takes time proportional to $|V(\Gamma(H))| \cdot
|E(\Gamma(H))|$ (see the complexity analysis of the conjugacy
problem). Since, by the Theorem~\ref{thm: properties of subgroup
graphs} (5), $|V(\Gamma(H))|$ and $|E(\Gamma(H))|$ are
proportional to $m$, the complexity of the above algorithm  is
$O(m^2)$.

If the subgroup $H$ is given by the graph $(\Gamma(H),v_0)$ and
not by a finite set of subgroup generators, then the complexity is
equal to $|V(\Gamma(H))| \cdot |E(\Gamma(H))|$. Thus in both cases
the algorithm is quadratic in the size of the input.


\section{Intersection Properties}
\label{sec:Intersection Properties}

In this section we study properties of intersections of finitely
generated subgroups of amalgams of finite groups such as the
\emph{Howson property} (the intersection of two finitely generated
subgroups   is finitely generated), \emph{malnormality} and
\emph{almost \ malnormality}. The corresponding algorithmic
problems and their solutions are presented.

\begin{defin} \label{def: malnormality}
Let $H$ be a subgroup of a group $G$.

We say that $H$ is a \underline{malnormal} subgroup of $G$ if and
only if
$$ gHg^{-1} \cap H=\{1\}, \ \forall g \in G \setminus H.$$

$H$  is \underline{almost malnormal} if for all $g \in G \setminus
H$, the subgroup $H \cap gHg^{-1}$ is finite.

\end{defin}

Obviously, $\{1\}$ and $G$ are malnormal subgroups of $G$. If $G$
is Abelian, then $\{1\}$ and $G$ are the only malnormal subgroups
of $G$. The most natural example of a malnormal subgroup is $K$
inside any free product $K \ast L$.

Malnormal subgroups of hyperbolic groups have recently become the
object of intensive studies (see, for instance, \cite{gi_quas,
gi_sep, gi_doub, g-m-s}).  Thus malnormality plays an important
role in the Combination Theorem for hyperbolic groups \cite{b-f,
gi_comb, kharlamovich-m}. For importance of almost malnormality
see, for example, \cite{ashot, wise-RF}.

As is well known \cite{b-wise}, in general, malnormality is
undecidable in hyperbolic groups. However, the results presented
in the current section show that malnormality is \emph{decidable}
for finitely generated subgroups in the class of amalgams of
finite groups.
Below we present a polynomial time algorithm (Corollary \ref{cor:
alg malnormality}) that solves the \emph{malnormality problem},
which asks to decide whether or not a subgroup $H$ of the group
$G$ is malnormal in $G$. The complexity analysis of the presented
algorithm is given at the end of this section.

Product graphs (Definition~\ref{def: product graph})  are used to
compute intersections of subgroups via their subgroup graphs.
We start by studding products of precovers. As an immediate
consequence, a solution of \emph{the intersection problem}, which
asks to find effectively the intersection of two subgroups, is
obtained (Corollary~\ref{cor:IntersectionProblem}). This allows
one to conclude (Corollary~\ref{cor:HowsonProperty}) that amalgams
of finite groups possess \emph{Howson property}, which is known to
be true (see, for instance \cite{gi_quas, sho}).

Then we characterize malnormality of a finitely generated subgroup
$H$ of an amalgam of finite groups  by the properties of the
product graph $\Gamma(H) \times \Gamma(H)$ (Theorem~\ref{H is
malnormal iff (1)-(2)}). This provides the solution of the
malnormality problem (Corollary \ref{cor: alg malnormality}).
These results are naturally extended to detect \emph{almost
malnormality} of finitely generated subgroups of amalgams of
finite groups (Theorem~\ref{thm:AlmostMalnormal} and
Corollary~\ref{cor:AlmostMalnormal}).

\subsection*{Product Graphs}

\begin{defin} \label{def: product graph}
Let $\Gamma$ and $\Delta$ be two labelled with $X_1^{\pm} \cup
X_2^{\pm}$ graphs. The \underline{product graph} $ \Gamma \times
\Delta$ is the graph defined as follows.
\begin{enumerate}
 \item [(1)] $V(\Gamma \times \Delta)=V(\Gamma) \times V(\Delta)$.
 \item [(2)] for each pair of vertices $(v_1,u_1)$, $(v_2,u_2) \in V(\Gamma \times
 \Delta)$ (so that $v_1, v_2 \in V(\Gamma)$ and $u_1, u_2 \in
 V(\Delta)$) and the letter $x \in X$ there exists an edge $e \in E(\Gamma \times
 \Delta)$ with
  $$\iota(e)=(v_1,u_1), \ \tau(e)=(v_2,u_2) \  {\rm and} \ lab(e)
  \equiv x$$
if and only if there exist edges $e_1 \in E(\Gamma)$ and $e_2 \in
E(\Delta)$ such that
$$\iota(e_1)=v_1, \ \tau(e_1)=v_2 \  {\rm and} \ lab(e_1)
  \equiv x$$
  and
$$\iota(e_2)=u_1, \ \tau(e_2)=u_2 \  {\rm and} \ lab(e_2)
  \equiv x$$
\end{enumerate}
\end{defin}

%
%

Along this section we consider  $G=G_1 \ast_A G_2$ to be an
amalgam of finite groups $G_1$ and $G_2$.

\begin{lem} \label{product of precovers=precover}
Let $\Gamma$ and $\Delta$ be finite precovers of $G=G_1 \ast_A
G_2$.

Then nonempty connected components of $\Gamma \times \Delta$ are
finite precovers of $G$.
\end{lem}
\begin{proof} Since $\Gamma$ and $\Delta$ are finite graphs the product graph is
finite, by Definition~\ref{def: product graph}. Thus each of its
connected components is finite as well.

Let $\Phi$ be a nonempty  connected component $\Phi$ of $\Gamma
\times \Delta$, that is $E(\Phi) \neq \emptyset$.

 Let
$p$ be a path in $\Phi$ such that $lab(p)\equiv r$, where $r=_G
1$.
Let $(v,u)=\iota(p)$ and $(v',u')=\tau(p)$ then there exist paths
$p_1$ in $\Gamma$ and $p_2$ in $\Delta$ such that
$$\iota(p_1)=v, \ \tau(p_1)=v' \ {\rm and} \ lab(p_1) \equiv r$$
and
$$\iota(p_2)=u, \ \tau(p_2)=u' \ {\rm and} \ lab(p_2) \equiv r.$$
Since $\Gamma$ and $\Delta$ are $G$-based graphs, we have
$v=\iota(p_1)=\tau(p_1)=v'$ and $u=\iota(p_2)=\tau(p_2)=u'$. Thus
$(v,u)=(v',u')$. Hence $p$ is a closed path in $\Phi$. That is
$\Phi$ is $G$-based.

Finally, we have to show that each $X_i$-monochromatic component
of $\Phi$ is a cover of $G_i$ ($i \in \{1,2\}$). By
Lemma~\ref{lemma1.5},  a $X_i$-monochromatic component is a cover
of $G_i$  if and only if it is $X_i^{\pm}$-saturated and
$G_i$-based ($i \in \{1,2\}$).

Let $C$ be a  $X_i$-monochromatic component of $\Phi$. Since the
graph $\Phi$ is $G$-based, it is, in particular, $G_i$-based.
Hence $C$ is $G_i$-based.

$C$ is $X_i^{\pm}$-saturated. Indeed, let $(v,u) \in V(C)$.
Thus either $(v,u) \in VM_i(C)$ ($i \in \{1,2\}$) or $(v,u) \in
VB(C)$. Definition~\ref{def: product graph} implies that in the
first case at least one of the vertices $v$ or $u$ is
$X_i$-monochromatic in $\Gamma$ and $\Delta$, respectively, and
the other one is bichromatic or $X_i$-monochromatic. If $(v,u) \in
VB(C)$ then $v \in VB(\Gamma)$ and $u \in VB(\Delta)$. %

Since $\Gamma$ and $\Delta$ are precovers, their bichromatic
vertices are $X^{\pm}$-saturated and their $X_i$-monochromatic
vertices are $X_i^{\pm}$-saturated ($i \in \{1,2\}$).

Therefore, by Definition~\ref{def: product graph}, the vertex
$(v,u)$ is  $X_i^{\pm}$-saturated. Thus $C$ is
$X_i^{\pm}$-saturated. Hence it is a cover of $G_i$.

By definition of precover, each  nonempty  connected component
$\Phi$ of $\Gamma \times \Delta$ is a finite precover of $G$.

 \end{proof}

Let $C$ be the connected component of $\Gamma \times \Delta$
containing the vertex $\vartheta$. Therefore $Lab(\Gamma \times
\Delta, \vartheta)=Lab(C, \vartheta) $, because $Loop(\Gamma
\times \Delta, \vartheta)=Loop(C, \vartheta) $. From now on we
allow ourselves to vary between this two notations which define
the same.


\begin{lem} \label{the intersection of  precovers H and K}
Let $\Gamma$ and $\Delta$ be finite precovers of $G$ such that
$Lab(\Gamma,v)=H$ and $Lab(\Delta,u)=K$, where $v \in V(\Gamma)$
and $u \in V(\Delta)$.

 Let
$\vartheta=(v,u) \in V(\Gamma \times \Delta)$. Let $C$ be a
connected component  of $\Gamma \times \Delta$ such that
$\vartheta \in V(C)$.

If $v \in VM_i(\Gamma)$ and $u \in VM_j(\Delta)$, where $1 \leq i
\neq j \leq 2$, then
$$V(C)=\{\vartheta\},  \ E(C)=\emptyset \ {\rm and} \ \{1\} \leq H \cap K \leq
A.$$
Otherwise $E(C) \neq \emptyset$ and $Lab(\Gamma \times \Delta,
\vartheta)=H \cap K$.

\end{lem}
\begin{proof} If $v \in VM_i(\Gamma)$ and $u \in VM_j(\Delta)$, where $1
\leq i \neq j \leq 2$, then, by Definition~\ref{def: product
graph}, $V(C)=\{\vartheta\}$ and $E(C)=\emptyset$.

Let $H \cap K \neq \{1\}$. Then there exists $1 \neq w \in H \cap
K$.  Without loss of generality, we can assume that the word $w$
is in normal form.

Assume first that the syllable length  of the normal word $w$ is
greater than $1$. Let $(x_1,x_2, \ldots,x_{k-1},x_k)$ be a normal
decomposition of $w$. By Lemma~\ref{lem: normal elements in
precovers}, this word  labels a normal path closed at the
basepoint in both graphs $\Gamma$ and $\Delta$.

Hence $x_1,x_k \in G_i \setminus A$, because $v \in VM_i(\Gamma)$.
On the other hand, $u \in VM_j(\Delta)$ and therefore $x_1,x_k \in
G_j \setminus A$ ($1 \leq i \neq j \leq 2$). This is a
contradiction. Therefore the syllable length of the normal word $w
\in H \cap K$ is equal to $1$. By Lemma~\ref{lem: normal elements
in precovers}, similar arguments show that in this case $w \in A$.

Assume now that the vertices $v$ and $u$ are not of different
colors. Using the same ideas as in the proof of Lemma~\ref{product
of precovers=precover}, we can assume, without loss of generality,
that the vertices $v$ and $u$ are $X_i^{\pm}$-saturated ($i \in
\{1,2\}$). Therefore, by Definition~\ref{def: product graph},
$E(C) \neq \emptyset$.

Let $w \in H \cap K$ be a  normal word.

By Lemma~\ref{lem: normal elements in precovers}, if either
$l(w)>1$, or $l(w)=1$ and $w \not\in A$ then the word $w$ labels a
path in $\Gamma$ closed at $v$ and also labels a path in $\Delta$
closed at $u$. Hence, by Definition~\ref{def: product graph},
there exists a path $p$ closed at $\vartheta$ in $C \subseteq
\Gamma(H) \times \Gamma(K)$ such that $lab(p) \equiv w$. Thus $w
\in Lab(\Gamma(H) \times \Delta, \vartheta)=Lab(C,\vartheta)$.

Assume now that the syllable length of $w$ is equal to $1$ and $w
\in A$. Let $w' \in G_i \cap A$ such that $w'=_G w$. Since, by our
assumption, the vertices $v$ and $u$ are $X^{\pm}_i$ saturated,
Lemma~\ref{lem: normal elements in precovers} implies that the
normal word $w'$ labels a path closed at the basepoint in both
graphs $\Gamma$ and $\Delta$. Therefore $w \in Lab(\Gamma(H)
\times \Delta, \vartheta)$. Thus
$$H \cap K  \subseteq Lab(\Gamma \times \Delta,
\vartheta)=Lab(C, \vartheta).$$

Now let $p$ be a path in $C \subseteq \Gamma \times \Delta$ closed
at $\vartheta$. Hence, by Definition~\ref{def: product graph},
there exists a path $p_1$ in $\Gamma$ closed at $v$ and there
exists a path $p_2$ in $\Delta$ closed at $u$, such that
$$lab(p) \equiv lab(p_1) \equiv lab(p_2).$$
Since $lab(p_1) \in Lab(\Gamma,v_0)=H$ and $lab(p_2) \in
Lab(\Delta,u)=K$, we have $Lab(C, \vartheta) \subseteq
Lab(\Gamma,v) \cap Lab(\Delta,u)=H \cap K.$

Hence  $Lab(C, \vartheta)=H \cap K$.

\end{proof}


\begin{remark} \label{rem: H intersects K on A}
{\rm The above proof implies that if $H \cap K \leq G_i$  ($i \in
\{1,2\}$) then
$$H \cap K=Lab(C,v) \cap Lab(D,u),$$ where $C$ and $D$ are
$X_i$-monochromatic component of $\Gamma$ and $\Delta$,
respectively, such that $v \in V(C)$ and $u \in V(D)$.
 \e }
\end{remark}

Recall that $(\Gamma(H),v_0)$ is the subgroup graph of $H$
constructed by the generalized Stallings' algorithm.


\begin{cor} \label{cor: intersection of reduced precovers}
Let $H$ be a finitely generated subgroup of $G$.

Let $\Delta$ be a finite precover of $G$ such
 $Lab(\Delta,u)=K$.

 Let
$\vartheta=(v_0,u) \in V(\Gamma(H) \times \Delta)$. Let $C$ be a
connected component  of $\Gamma(H) \times \Delta$ such that
$\vartheta \in V(C)$.

Then $Lab(\Gamma(H) \times \Delta, \vartheta)=Lab(C,\vartheta)=H
\cap K$.

\end{cor}
\begin{proof}  If $C$ is a nonempty (i.e., $E(C) \neq \emptyset$)
connected component  of $\Gamma(H) \times \Delta$ then, by
Lemma~\ref{the intersection of  precovers H and K}, $Lab(\Gamma(H)
\times \Delta, \vartheta)=Lab(C,\vartheta)=H \cap K$.

Otherwise, Lemma~\ref{the intersection of  precovers H and K}
implies that $\{1\} \leq H \cap K \leq A$. If $H \cap K=\{1\}$
then, since $Lab(C,\vartheta)=\{1\}$, the desired equality holds.

Assume now that $\{1\} \neq H \cap K \leq A$. Hence $H \cap A \neq
\{1\}$. Therefore, since $\Gamma(H)$ is a reduced precover of $G$,
Definition~\ref{def:ReducedPrecover} (ii) implies $v \in
VB(\Gamma(H))$. Therefore, by Lemma~\ref{the intersection of
precovers H and K}, $E(C) \neq \emptyset$. That is $C$ is
nonempty. This is a contradiction. Thus $Lab(\Gamma(H) \times
\Delta, \vartheta)=Lab(C,\vartheta)=H \cap K$.

 \end{proof}




\begin{cor}[The Intersection Problem]
\label{cor:IntersectionProblem}
Let $H=\langle h_1, \cdots, h_n \rangle$ and $K=\langle k_1,
\cdots, k_m \rangle$ be finitely generated subgroups of an amalgam
of finite groups $G=G_1 \ast_A G_2$.

Then there exists an algorithm which finds the generators of $H
\cap K$, which is finitely generated.
\end{cor}
\begin{proof}
We first use the generalized Stallings' folding  algorithm to
construct the subgroup graphs $(\Gamma(H),v_0)$ and
$(\Gamma(K),u_0)$.

Since, by Theorem~\ref{thm: properties of subgroup graphs}, these
graphs are finite, the product graph $\Gamma(H) \times \Gamma(K)$
can be effectively constructed, and it is finite.

By Corollary~\ref{cor: intersection of reduced precovers}, $H \cap
K=Lab(C,(v_0,u_0))$, where $C$ is a connected component of
$\Gamma(H) \times \Gamma(K)$ such that $(v_0,u_0) \in V(C)$.
Therefore it is sufficient to construct only the component $C$.

Let $\vartheta=(v_0,u_0)$. Recall that $G=G_1 \ast_A G_2=gp\langle
X \; | \; R \rangle$.

 By Lemma~\ref{product of
precovers=precover}, $(C,\vartheta)$ is a precover of $G$, which
is finite by the construction. Therefore, in particular, it is a
finite well-labelled graph.
Hence, the subgroup $\widetilde{L}$ of $FG(X)$ determined by
$(C,\vartheta)$ is finitely generated  (\cite{kap-m, mar_meak,
stal}). Since
$Lab(C,\vartheta)=\widetilde{L}/N=\widetilde{L}/\widetilde{L} \cap
N$, where $N$ is the normal closure of $R$ in $FG(X)$, we conclude
that $Lab(C,\vartheta)=H \cap K$ is finitely generated.

To find the generating set we proceed as follows.
Let $T$ be a fixed spanning tree  of $C$. For all $v \in V(C)$, we
consider $t_v$ to be the unique freely reduced path in $T$ from
the basepoint $\vartheta$ to the vertex $v$.

For each $e \in E(C)$ we consider $t(e)=t_{\iota(e)}e
\overline{t_{\tau(e)}}$. Thus if $e \in E(T)$ then $t(e)$ can be
freely reduced to an empty path, that is $lab(t(e))=_{FG(X)} 1$.

Let $E^+$ be the set of positively oriented edges of $C$. Let
$$X_L=\{lab(t(e)) \; | \; e \in E^+ \setminus E(T) \},$$
As is well known \cite{kap-m, mar_meak, stal},
$\widetilde{L}=FG(X_H)$. Therefore $L=\langle X_L \rangle$.

\end{proof}

\begin{remark}
{\rm In order to compute a finite group presentation of the
subgroup $H \cap K$ one can apply to $(C,\vartheta)$ the
restricted version of the Reidemeister-Schreier algorithm
presented in \cite{m-algI}. This is possible, because $C$ is a
finite precover of $G$, by Lemma~\ref{product of
precovers=precover}.} \e

\end{remark}


\begin{cor}[Howson  Property] \label{cor:HowsonProperty}
Let $G=G_1 \ast_A G_2$ be an amalgam of finite groups.

The intersection of two finitely generated subgroups of $G$ is
finitely generated in $G$. That is $G$ possesses the
\underline{Howson property}.
\end{cor}

\medskip


\subsection*{Malnormality}

\begin{lem} \label{for any element g in G there exist a component in the product graph}
Let $H$ and $K$ be finitely generated subgroups of the group $G$.
Let  $g$ be an element of $G$.

Then  $H \cap gKg^{-1}$ conjugates to a subgroup of $A$ or it
conjugates to the subgroup $Lab(C,\vartheta)$, where $C$ is a
nonempty  connected component of  $\Gamma(H) \times \Gamma(K)$
such that $\vartheta=(v,u) \in V(C)$, and $v \in V(\Gamma(H))$, $u
\in V(\Gamma(K))$.
\end{lem}
\begin{proof}
Without loss of generality, assume that $g$ is a normal word. Then
either there exists a path $p$ in $\Gamma(K)$ such that
$\iota(p)=u_0$ and $lab(p) \equiv g^{-1}$ or such a path doesn't
exist in $\Gamma(K)$.

In the first case, let $u=\tau(p)$ then
$Lab(\Gamma(K),u)=gKg^{-1}$. By Corollary \ref{cor: intersection
of reduced precovers}, $Lab(C,\vartheta)=H \cap gKg^{-1}$, where
$\vartheta=(v_0,u)$.

Assume now that $p'$ is the longest path in $\Gamma(K)$ such that
$\iota(p')=u_0$ and $lab(p') \equiv g_2^{-1}$, where $g \equiv g_1g_2$. Let
$u=\tau(p')$. Then either there exists a path $q$ in $\Gamma(H)$ such that
$\iota(q)=v_0$ and $lab(q) \equiv g_1$ or
 such a path doesn't exist in $\Gamma(H)$.

\begin{figure}[!h]
\begin{center}
\psfrag{A }[][]{{\large $\Gamma(H)$}} \psfrag{B }[][]{{\large
$\Gamma(K)$}}
\psfrag{v }[][]{\small $v$} \psfrag{v0 }[][]{\small $v_0$}
\psfrag{u }[][]{\small $u$} \psfrag{u0 }[][]{\small $u_0$}
\psfrag{g1 }[][]{$g_1$} \psfrag{g2 }[][]{$g_2$}
\psfrag{p }[][]{$\overline{p'}$} \psfrag{q }[][]{$q$}
\includegraphics[width=0.7\textwidth]{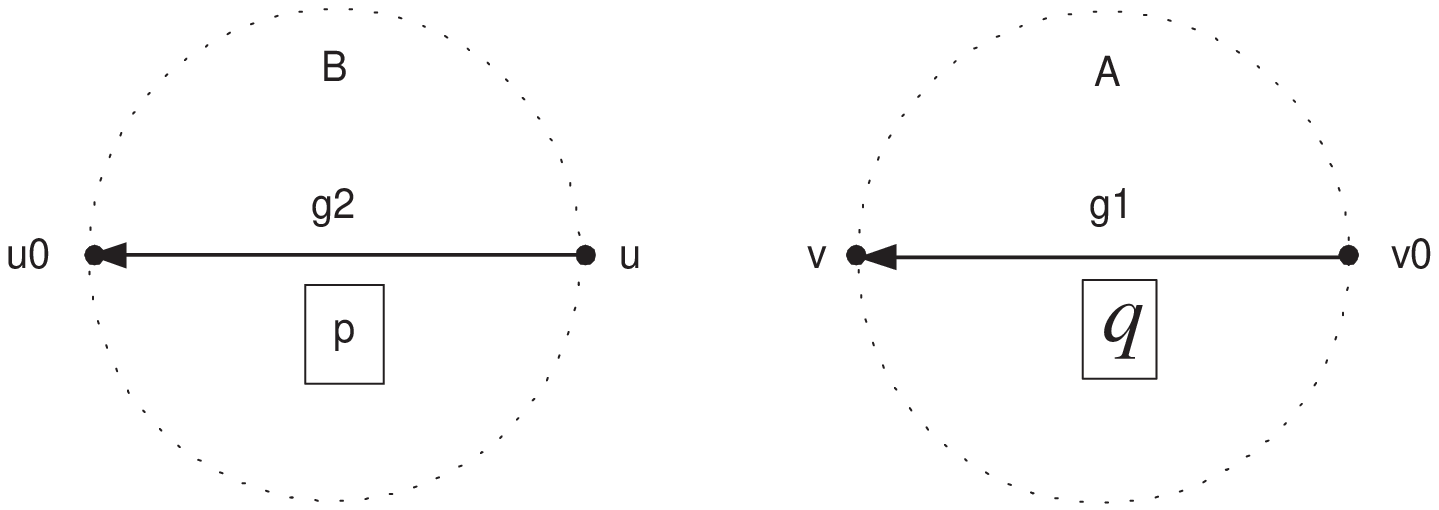}
\caption[The first auxiliary figure for the proof of
Lemma~\ref{for any element g in G there exist a component in the
product graph}]{ \label{fig:IntersectionOfGraphs}}
\end{center}
\end{figure}

First we assume that $q$ exists in $\Gamma(H)$, see Figure
\ref{fig:IntersectionOfGraphs}. Let $v=\tau(q)$. Thus
$Lab(\Gamma(K),u)=g_2Kg_2^{-1}$ and
$Lab(\Gamma(H),v)=g_1^{-1}Hg_1$. By Lemma~\ref{the intersection of
precovers H and K}, if $u$ and $v$ are monochromatic vertices of
different colors then
$$\{1\} \leq g_1^{-1}Hg_1 \cap g_2Kg_2^{-1} \leq A.$$
Otherwise, $Lab(C, (u,v))=g_1^{-1}Hg_1 \cap g_2Kg_2^{-1}$, where
$C$ is a nonempty  connected component of the product graph
$\Gamma(H) \times \Gamma(K)$ containing the vertex
$\vartheta=(v,u)$,  $v \in V(\Gamma(H))$, $u \in V(\Gamma(K))$.

Since $$H \cap gKg^{-1}=g_1(g_1^{-1}Hg_1 \cap
g_2Kg_2^{-1})g_1^{-1},$$ we have  $\{1\} \leq H \cap gKg^{-1} \leq
g_1Ag_1^{-1}$ or  $H \cap gKg^{-1}=g_1Lab(C,\vartheta)g_1^{-1}$,
respectively.

Assume now that there is no  path   in $\Gamma(H)$ starting at
$v_0$ and labelled with $g_1$. Below we prove that in this case $H
\cap gKg^{-1}=\{1\}$.

 Suppose  that $H \cap gKg^{-1} \neq \{1\}$.
Let $(\Gamma',u')$ be the graph obtained from $\Gamma(K)$ by
attaching  a path $t$ at the vertex $u$, such that $\iota(t)=u$
and $lab(t) \equiv g_1^{-1}$. Let $\tau(t)=u'$,  see Figure
\ref{fig:IntersectionOfGraphs1}.
\begin{figure}[!h]
\begin{center}
\psfrag{B }[][]{{\large $\Gamma(K)$}}
\psfrag{u }[][]{$u$} \psfrag{u1 }[][]{$u'$} \psfrag{u0
}[][]{$u_0$}
\psfrag{g1 }[][]{$g_2$} \psfrag{g2 }[][]{$g_1$}
\psfrag{x }[][]{$x$} \psfrag{y }[][]{$y$}
\psfrag{p }[][]{$\overline{p'}$} \psfrag{t }[][]{$\overline{t}$}
\includegraphics[width=0.7\textwidth]{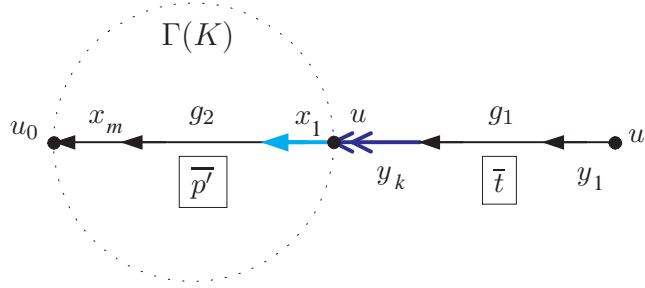}
\caption[The second auxiliary figure for the proof of
Lemma~\ref{for any element g in G there exist a component in the
product graph}]{{\footnotesize The graph $\Gamma'$}
\label{fig:IntersectionOfGraphs1}}
\end{center}
\end{figure}
The graph $(\Gamma',u')$ is finite, because $\Gamma(K)$ is finite
and the ``stem'' $t$ is also finite. It is well-labelled, because
$\Gamma(K)$ is well-labelled and $g_2^{-1}$ is the maximal prefix
of the word $g^{-1}$ that is readable in $\Gamma(K)$ starting at
$u_0$.

Thus $Lab(\Gamma',u')=gKg^{-1}$, and for each nontrivial element
in $gKg^{-1}$ and ,in particular, for each $1 \neq z \in H \cap
gKg^{-1}$ there exists a nonempty path $\gamma$ in $(\Gamma',u')$
such that $lab(\gamma)=_G z$ and $\iota(\gamma)=\tau(\gamma)=u'$.
The above construction of $(\Gamma',u')$ implies that $lab(\gamma)
\equiv g_1wg_1 ^{-1}$, where $w \in Lab(\Gamma(K),u)$. Since
$lab(\gamma)=_G z \neq 1$, the word $w$ is nonempty. Thus, by
Lemma~\ref{lem: normal elements in precovers}, we can assume that
the word $w$ is in normal form, because $(\Gamma(K),u)$ is a
finite precover of $G$ and $w \in Lab(\Gamma(K),u)$.

Since each $X_i$-monochromatic component of $\Gamma(K)$ is a cover
of $G_i$, $i \in \{1,2\}$ (thus, in particular, it is
$X_i^{\pm}$-saturated) and because $g_2^{-1}$ is the maximal
prefix of the word $g^{-1}$ such that there is a path $p'$ in
$\Gamma(K)$ with $\iota(p')=u_0$ and $lab(p') \equiv g_2^{-1}$,
there exists a normal decomposition of the word $g$
$$(y_1, \ldots, y_k,x_1, \ldots, x_m)$$   such that $g_1 \equiv y_1 \cdots y_k$ and
$g_2 \equiv x_1 \cdots x_m$, where $y_k \in G_i$ and $x_1 \in
G_j$, $1 \leq i \neq j \leq 2$.

Note that $u$ is a $X_j$-monochromatic vertex of $\Gamma(K)$.
Otherwise there exists a $G_i$-monochromatic component $D$ in
$\Gamma(K)$, such that $u \in V(D)$. Since it is $X_i$-saturated,
$y_k$ is readable from $u$ in $D$ and therefore in $\Gamma(K)$.
This contradicts  the maximality of the word $g_2$. Thus
  the word $g_1wg_1 ^{-1}$ is in normal form.

On the other hand, since $z=_G g_1wg_1 ^{-1} \in H$,
Theorem~\ref{thm: properties of subgroup graphs} (4) implies that
there exists a normal path $\gamma'$ in $\Gamma(H)$ closed at
$v_0$ with $lab(\gamma') \equiv g_1wg_1 ^{-1}$. Therefore there
exists a path in $\Gamma(H)$ starting at $v_0$ and labelled with
$g_1$.
 This contradicts with our assumption that such a path doesn't exist in $\Gamma(H)$.
 Hence $H \cap gKg^{-1} = \{1\}$.

\end{proof}


\begin{thm} \label{H is malnormal iff (1)-(2)}
Let $H$ be a finitely generated subgroup of $G$. Then $H$ is malnormal in $G$ if
and only if the following holds
\begin{enumerate}
 \item [(1)]
$H \cap gHg^{-1} \cap fAf^{-1}=\{1\}$ for all $g \in G \setminus
H$, $f \in G$;
 \item [(2)] each connected nonempty component $C$ of $\Gamma(H) \times \Gamma(H)$
which doesn't contain the vertex $(v_0,v_0)$ satisfies $Lab(C,\vartheta)=\{1\}$
for all $\vartheta \in V(C)$.
\end{enumerate}
\end{thm}
\begin{proof}
Suppose that $H$ is malnormal in $G$. Then $H \cap gHg^{-1}=\{1\}$
for all $g \in G \setminus H$. Hence $H \cap gHg^{-1} \cap
fAf^{-1}=\{1\}$ for all $g,f \in G$.

Let $C$ be a nonempty  connected component  of $\Gamma(H) \times
\Gamma(H)$ such that $(v_0,v_0) \not\in V(C)$. Let
$\vartheta=(v_1,v_2) \in V(C)$. Hence $v_1 \neq v_2 \in
V(\Gamma(H))$. Indeed, if $v_1=v_2$ then $(v_1,v_2) \in V(C_0)$,
where $C_0$ is a connected component of the product graph
$\Gamma(H) \times \Gamma(H)$, containing the vertex $(v_0,v_0)$.

Lemma~\ref{the intersection of precovers H and K} implies that
$$Lab(C,\vartheta)=Lab(\Gamma(H),v_1) \cap
Lab(\Gamma(H),v_2)=g_1^{-1}Hg_1 \cap g_2^{-1}Hg_2,$$ where $g_1$
and $g_2$ label  paths in $\Gamma(H)$ from $v_0$ to $v_1$ and to
$v_2$, respectively.

Since $\Gamma(H)$ is a $G$-based graph and $v_1 \neq v_2$, we have
$g_1g_2^{-1} \not\in H$. Indeed, otherwise
$$v_0=v_0 \cdot (g_1g_2^{-1})=(v_0 \cdot g_1) \cdot g_2^{-1}=v_1 \cdot g_2^{-1}.$$ Thus
$$v_1=v_1 \cdot (g_2^{-1}g_2)=(v_1 \cdot g_2^{-1}) \cdot g_2=v_0 \cdot g_2=v_2.$$

However $g_1Lab(C,\vartheta)g_1^{-1}= H \cap
g_1g_2^{-1}Hg_2g_1^{-1}=\{1\}$, because $H$ is malnormal in $G$.
Therefore $Lab(C,\vartheta) = \{1\}$.

Assume now that the conditions (1)-(2) are satisfied.
By Lemma~\ref{for any element g in G there exist a component in
the product graph}, the  subgroup $H \cap gHg^{-1}$ conjugates to
a subgroup of $A$ or it conjugates to the subgroup
$Lab(C,\vartheta)$, where $C$ is a nonempty (i.e. $E(C) \neq
\emptyset$) connected component of the product graph $\Gamma(H)
\times \Gamma(H)$.

In the first case, $1 \leq f^{-1}(H \cap gHg^{-1})f \leq A$ for
some $f \in G$.  Therefore $H \cap gHg^{-1} \cap fAf^{-1} \neq
\{1\}$. This contradicts  condition (1).

 Condition (2) implies
$Lab(C,\vartheta)=\{1\}$, hence  $H \cap gHg^{-1}=\{1\}$.
Therefore $H$ is malnormal in $G$.

\end{proof}


\begin{cor}[The Malnormality Problem] \label{cor: alg malnormality}
Let $h_1, \ldots h_k  \in G$. Then there exists an algorithm which
decides whether or not the subgroup $H = \langle h_1, \ldots h_k
\rangle$ is malnormal in $G$.

If $H$ is not malnormal, the algorithm  produces a nontrivial
element $g \in G \setminus H$ such that $H \cap gHg^{-1} \neq
\{1\}$.
\end{cor}
\begin{proof} First we construct the subgroup graph
$\Gamma(H)$ using the generalized Stallings' algorithm. Since, by
Theorem~\ref{thm: properties of subgroup graphs}, it is finite,
the product graph  $\Gamma(H) \times \Gamma(H)$ can be constructed
effectively. Now we check whether each connected nonempty
component $C$ of $\Gamma(H) \times \Gamma(H)$ which doesn't
contain the vertex $(v_0,v_0)$ satisfies $Lab(C,\vartheta)=\{1\}$
for some $\vartheta \in V(C)$.  If there exists a component with
$Lab(C,\vartheta)\neq \{1\}$ ($\vartheta \in V(C)$), then, by
Theorem~\ref{H is malnormal iff (1)-(2)}, $H$ is not malnormal in
$G$.

Moreover, by the proof of Theorem~\ref{H is malnormal iff
(1)-(2)}, a nontrivial element $g \in G \setminus H$ such that $H
\cap gHg^{-1} \neq \{1\}$ is $g=_G g_1g_2^{-1}$, where
$$v_1=v_0 \cdot g_1, \ \ v_2=v_0 \cdot g_2 \ \ {\rm and} \ \
\vartheta=(v_1,v_2).$$

Note that it is sufficient to check whether
$Lab(C,\vartheta)=\{1\}$ only for some $\vartheta \in V(C)$.
Indeed, if $v' \in V(C)$ such that $v' \neq \vartheta$ then
$Lab(C,v')=xLab(\Gamma,\vartheta)x^{-1}$, where $x \in G$ and $v'
\cdot x=\vartheta$. Thus $Lab(\Gamma,v')=\{1\}$.

Since, by Lemma~\ref{product of precovers=precover}, $C$ is a
precover, the above verification can be done as follows.

By Lemma~\ref{lem:FormOfRedPrecover}(i), a reduced precover
$(\Delta,u)$ has $Lab(\Delta,u)=\{1\}$ if and only if
$V(\Delta)=u$ and $E(\Delta)=\emptyset$. Thus
$Lab(C,\vartheta)=\{1\}$ if and only if  the iterative  removal of
the unique sequence of redundant components from $(C,\vartheta)$
yield the empty graph $(\Delta,u)$ with the above properties.

%
Assume now that all connected nonempty components of $\Gamma(H)
\times \Gamma(H)$ satisfy condition (2) from Theorem~\ref{H is
malnormal iff (1)-(2)}. Then  $H$ is malnormal in $G$ if and only
if condition (1) is satisfied. In order to verify this we proceed
as follows.

Let $D$ be an arbitrary single vertex of the product graph
$\Gamma(H) \times \Gamma(H)$, i.e. $D$ is an empty component of
$\Gamma(H) \times \Gamma(H)$ such that $V(D)=\{(v_1,v_2)\}$ and
$E(D)=\emptyset$. Then $v_1 \neq v_2 \in V(\Gamma(H))$ such that
$v_0 \cdot g_i=v_i$, $i \in \{1,2\}$. Since $E(D)=\emptyset$, by
 Lemma~\ref{the intersection of  precovers H and K},
$v_1$ and $v_2$ are monochromatic vertices of $\Gamma(H)$ of
different colors. Without loss of generality, assume that $v_i \in
VM_i(\Gamma(H))$, $i \in \{1,2\}$.
By Lemma~\ref{the intersection of  precovers H and K}, $\{1\} \leq
g_1^{-1}Hg_1 \cap g_2^{-1}Hg_2 \leq A$.

Let $C_i$ be a $X_i$-monochromatic component of
 $\Gamma(H)$ such that $v_i \in V(C_i)$.
By Remark \ref{rem: H intersects K on A},
%
$$g_1^{-1}Hg_1 \cap g_2^{-1}Hg_2=Lab(C_1,v_1) \cap Lab(C_2,v_2) \neq \{1\}.$$
 Thus we have to check if
$Lab(C_1,v_1) \cap Lab(C_2,v_2)$ is a nontrivial subgroup of $A$.
If so then $g_1^{-1}Hg_1 \cap g_2^{-1}Hg_2$ is a nontrivial
subgroup of $A$, otherwise $g_1^{-1}Hg_1 \cap g_2^{-1}Hg_2=\{1\}$.

Let $S=A \cap Lab(C_1,v_1)$. We consider $(Cayley(G_2,S), S \cdot
1)$. Thus  $Lab(Cayley(G_2,S), S \cdot 1)=S $. Let $E$ be a
nonempty connected component of the product graph $Cayley(G_2,S)
\times C_2$ containing the vertex $(S \cdot 1, v_2)$. Then, by
Lemma~\ref{the intersection of precovers H and K},
$$Lab(E,(S \cdot 1, v_2))=Lab(Cayley(G_2,S), S \cdot 1) \cap Lab(C_2,v_2)=$$ $$=S \cap Lab(C_2,v_2)=A \cap Lab(C_1,v_1) \cap
Lab(C_2,v_2)=g_1^{-1}Hg_1 \cap g_2^{-1}Hg_2.$$

Thus $g_1Lab(E,(S \cdot 1, v_2))g_1^{-1}=H \cap gHg^{-1}$, where
$g=_G g_1g_2^{-1}$. Hence  $Lab(E,(S \cdot 1, v_2)) \neq \{1\}$
implies $H$ is not malnormal in $G$.

Otherwise if $Lab(E,(S \cdot 1, v_2)) = \{1\}$ for each  component
$D$ of the product graph $\Gamma(H) \times \Gamma(H)$, where $E$
is constructed as described above, then $H$ is a malnormal
subgroup of $G$.

\end{proof}

\begin{ex}
{\rm

Let $G=gp\langle x,y | x^4, y^6, x^2=y^3 \rangle=G_1 \ast_A G_2$,
where $G_1=gp\langle x | x^4 \rangle$, $G_2=gp\langle y | y^6
\rangle$ and $A=\langle x^2 \rangle=\langle y^3 \rangle$.

Let $H$ be a finitely generated subgroup of $G$ given by its
subgroup graph $\Gamma(H)$ which is presented on
Figure~\ref{fig:Malnormality}.

We compute $\Gamma(H) \times \Gamma(H)$ (see
Figure~\ref{fig:Malnormality}). Using the method described along
with the proof of Corollary~\ref{cor: alg malnormality}, we
conclude that $Lab(C_1, (v_0,v_1))=\{1\}$, $Lab(C_3,
(v_0,v_3))=\{1\}$, but $Lab(C_2, (v_0,v_2)) \neq \{1\}$.
Therefore, by Theorem~\ref{H is malnormal iff (1)-(2)}, $H$ is not
malnormal in $G$.

} \e
\end{ex}

\begin{figure}[!h]
\begin{center}
\psfrag{B }[][]{$\Gamma(H) \times \Gamma(H)$}
\psfrag{A }[][]{$\Gamma(H)$}
\psfrag{c1 }[][]{$C_1$}
\psfrag{c2 }[][]{$C_2$} \psfrag{c3 }[][]{$C_3$}
\psfrag{v0 }[][]{$v_0$} \psfrag{v1 }[][]{$v_1$}
\psfrag{v2 }[][]{$v_2$} \psfrag{v3 }[][]{$v_3$}
\psfrag{v00 }[][]{$(v_0,v_0)$} \psfrag{v11 }[][]{$(v_1,v_1)$}
\psfrag{v22 }[][]{$(v_2,v_2)$} \psfrag{v33 }[][]{$(v_3,v_3)$}
\psfrag{v01 }[][]{$(v_0,v_1)$} \psfrag{v12 }[][]{$(v_1,v_2)$}
\psfrag{v23 }[][]{$(v_2,v_3)$} \psfrag{v30 }[][]{$(v_3,v_0)$}
\psfrag{v02 }[][]{$(v_0,v_2)$} \psfrag{v13 }[][]{$(v_1,v_3)$}
\psfrag{v20 }[][]{$(v_2,v_0)$} \psfrag{v31 }[][]{$(v_3,v_1)$}
\psfrag{v03 }[][]{$(v_0,v_3)$} \psfrag{v10 }[][]{$(v_1,v_0)$}
\psfrag{v21 }[][]{$(v_2,v_1)$} \psfrag{v32 }[][]{$(v_3,v_2)$}
\includegraphics[width=0.8\textwidth]{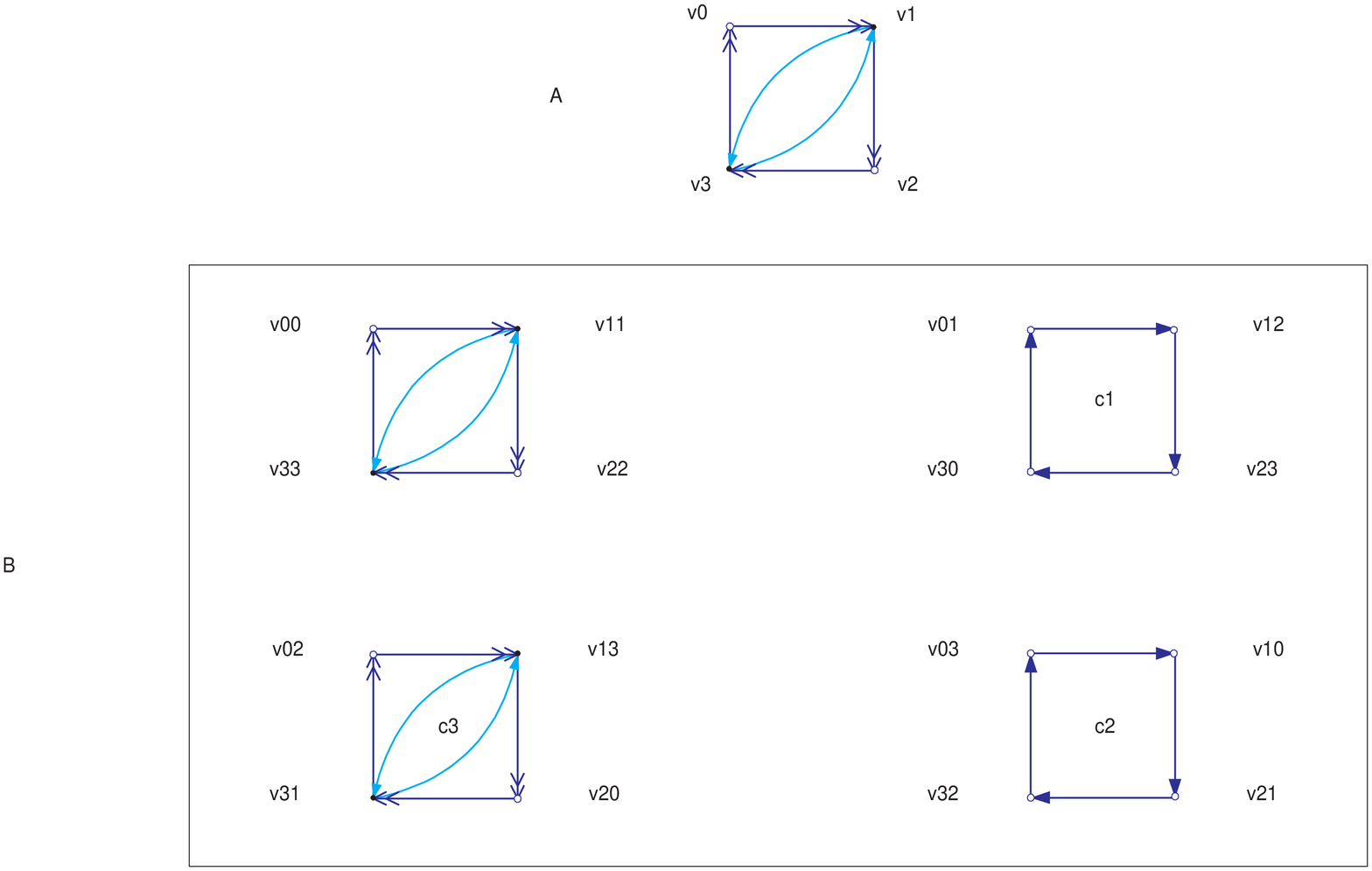}
\caption{ \label{fig:Malnormality}}
\end{center}
\end{figure}
%

\subsection*{Complexity}  By Theorem~\ref{thm: properties of subgroup graphs} (5), the complexity of
the construction of $\Gamma(H)$ for a subgroup $H$ of $G$ given by
a finite set of generators is $O(m^2)$, where $m$ is the sum of
lengths of the input subgroup generators.

The construction of $\Gamma(H) \times \Gamma(H)$ takes
$$O(|V(\Gamma(H))|^2+|V(\Gamma(H))| \cdot |E(\Gamma(H))|).$$

 Let $C$ be a connected component of
$\Gamma(H) \times \Gamma(H)$. To verify whether $Lab(C,v)=1$, $v
\in V(C)$, takes time proportional to $|E(C)|^2$, by  the
complexity analysis of the generalized Stallings algorithm (see
Lemma 8.7 in \cite{m-foldings}). Since
$$\sum_{C \subseteq \Gamma(H) \times
\Gamma(H)} |E(C)|=|E(\Gamma(H) \times \Gamma(H))| \leq
|E(\Gamma(H))|^2,$$ the above verification for all connected
components of $\Gamma(H) \times \Gamma(H)$ takes
$O(|E(\Gamma(H))|^4)$.

Since all the information about the free factors of the amalgams,
as well as the relative Cayley graphs of the free factors are not
a part of the input, the verifications concerning the empty
components of the product graph $\Gamma(H) \times \Gamma(H)$ takes
time $O(|V(\Gamma)|^2)$.

Since, by Theorem~\ref{thm: properties of subgroup graphs} (5),
$|E(\Gamma(H))|$ and $|V(\Gamma(H))|$ are proportional to $m$,
algorithm given by the proof of Corollary \ref{cor: alg
malnormality} takes $O(m^4)$. Thus the  algorithm is polynomial
in the size of the input.


\subsection*{Almost Malnormality}

\begin{thm} \label{thm:AlmostMalnormal}
Let $H$ be a finitely generated subgroup of $G$. Then $H$ is
almost malnormal in $G$ if and only if $Lab(C,\vartheta)$
conjugates to a subgroup of $G_1$ or $G_2$ ($\vartheta \in V(C)$),
for each  nonempty connected component $C$ of $\Gamma(H) \times
\Gamma(H)$, which doesn't contain the vertex $(v_0,v_0)$.
\end{thm}
\begin{proof}
Suppose that $H$ is almost malnormal in $G$. Then $H \cap
gHg^{-1}$ is finite for all $g \in G \setminus H$.

Let $C$ be a nonempty  connected component  of $\Gamma(H) \times
\Gamma(H)$ such that $(v_0,v_0) \not\in V(C)$. Let
$\vartheta=(v_1,v_2) \in V(C)$. By the proof of Theorem~\ref{H is
malnormal iff (1)-(2)}, $v_1 \neq v_2 \in V(\Gamma(H))$ and
$g_1g_2^{-1} \not\in H$, where $v_i=v_0 \cdot g_i$ ($i \in
\{1,2\}$). Moreover,
$$Lab(C,\vartheta)=Lab(\Gamma(H),v_1) \cap
Lab(\Gamma(H),v_2)=g_1^{-1}Hg_1 \cap g_2^{-1}Hg_2.$$

However
 $g_1Lab(C,\vartheta)g_1^{-1}= H \cap
g_1g_2^{-1}Hg_2g_1^{-1}$ is finite, because $H$ is almost
malnormal in $G$. Therefore, by the Torsion Theorem (IV.2.7 in
\cite{l_s}), $g_1Lab(C,\vartheta)g_1^{-1}$ conjugates to a
subgroup of $G_1$ or $G_2$.

Assume now that the condition is satisfied.
By Lemma~\ref{for any element g in G there exist a component in
the product graph}, for all $g \in G \setminus H$ the  subgroup $H
\cap gHg^{-1}$ conjugates to a subgroup of $A$ or it conjugates to
the subgroup $Lab(C,\vartheta)$, where $C$ is a nonempty connected
component of the product graph $\Gamma(H) \times \Gamma(H)$.
Therefore if $Lab(C,\vartheta)$ conjugates to a subgroup of $G_1$
or $G_2$, then, since $G_i$ ($i \in \{1,2\}$) is finite, $H$ is
almost malnormal.

\end{proof}

The \emph{almost malnormality problem} asks to decide whether or
not a subgroup $H$ of the group $G$ is almost malnormal in $G$.

\begin{cor}[The Almost Malnormality Problem] \label{cor:AlmostMalnormal}
Let $h_1, \ldots h_k  \in G$. Then there exists an algorithm which
decides whether or not the subgroup $H = \langle h_1, \ldots h_k
\rangle$ is almost malnormal in $G$.

If $H$ is not almost malnormal, the algorithm  produces a
nontrivial element $g \in G \setminus H$ such that $H \cap
gHg^{-1}$ is not finite.
\end{cor}
\begin{proof} The proof is similar to that of Corollary~\ref{cor: alg malnormality}.
First we construct the subgroup graph $\Gamma(H)$ using the
generalized Stallings' algorithm. Since, by Theorem~\ref{thm:
properties of subgroup graphs}, it is finite, the product graph
$\Gamma(H) \times \Gamma(H)$ can be constructed effectively. Now
for  each nonempty connected  component $C$ of $\Gamma(H) \times
\Gamma(H)$ which doesn't contain the vertex $(v_0,v_0)$, we check
whether $Lab(C,\vartheta)$  ($\vartheta \in V(C)$) conjugates to a
subgroup of $G_1$ or $G_2$. By Theorem~\ref{thm:AlmostMalnormal},
$H$ is almost malnormal in $G$ if and only if each such component
$C$ possesses this property.

We proceed  as follows. If $C$ consists of a unique
$X_i$-monochromatic component ($i \in \{1,2\}$) then  $\{1\} \leq
Lab(C,v) \leq G_i$.
Otherwise, let $\vartheta \in VB(C)$  be a basepoint of $C$.

By Lemma~\ref{product of precovers=precover}, $C$ is a finite
precover of $G$. If $(C,\vartheta)$ is not a reduced precover then
we remove from $C$ all the redundant components w.r.t. the
basepoint $\vartheta$. Let $(C',\vartheta')$ be the resulting
graph, where $\vartheta'$ is the image of $\vartheta$ in $C'$.
Thus $(C',\vartheta')$ is a reduced precover such that
$Lab(C,\vartheta)=Lab(C',\vartheta')$. Let $L=Lab(C',\vartheta')$.

By Lemma~\ref{lem:FormOfRedPrecover}, $L \leq G_i$ such that $L
\cap A=\{1\}$ if and only if $C'$ consists of a unique
$X_i$-monochromatic component, and  $L \leq A$  if and only if
$(C',\vartheta')= Cayley(G_1,L) \ast_{\{La \; | \; a \in A\}}
Cayley(G_2,L)$.
Thus $Lab(C,\vartheta)=L$ conjugates to a subgroup of $G_1$ or
$G_2$ if and only if $C'$ satisfies one of the above properties.

Note that   if there exists a connected component $C$ such that
none of the above properties is satisfied then, by the proof of
Theorem~\ref{thm:AlmostMalnormal}, a nontrivial element $g \in G
\setminus H$ such that  $H \cap gHg^{-1}$ is not finite is $g=_G
g_1g_2^{-1}$, where
$$v_1=v_0 \cdot g_1, \ \ v_2=v_0 \cdot g_2 \ \ {\rm and} \ \
\vartheta=(v_1,v_2).$$

\end{proof}

\subsection*{Complexity}
Similarly to the complexity analysis of the solution of the
malnormality problem, presented along with the proof of
Corollary~\ref{cor: alg malnormality}, the above solution of the
almost malnormality problem takes $O(|E(\Gamma(H))|^4)$, that is
$O(m^4)$, where $m=\sum_{i=1}^k |h_i|$.


\section{The Power Problem}
\label{subsec:PowerProblem}

The \emph{power problem} asks for an algorithm that decides
whether or not some \emph{nontrivial power} of a word $g$ in the
generators of a group $G$ belongs to the subgroup $H$ of $G$.

By a \emph{nontrivial power} of $g$ we mean an element $g^n \in G$
such that $n \geq 1$ and $g^n \neq_G 1$ (otherwise $g^n \in H$ for
each torsion element $g \in G$ and all $o(g) \: | \: n$).

This problem is an extension of the membership problem for $H$ in
$G$. The membership problem  for finitely generated subgroups in
amalgams of finite groups was (successfully) solved in
\cite{m-foldings} using subgroup graphs constructed by the
generalized Stallings' algorithm. Below we employ same technics to
solve the power problem in this class of groups
(Corollaries~\ref{cor: alg PI} and \ref{cor: alg PII}).
Theorem~\ref{order problem finite grp} provides the solution.  The
complexity analysis of the described algorithm is given at the end
of the section.

We split the power problem into two instances. The first one,
(PI), asks for an answer ``Yes'' or ``No'' on the question whether
some nonzero power of a word $g$ in the generators of $G$ belongs
to the subgroup $H$.

The second one, (PII), asks to find the minimal power $n>0$ such
that $g^n \in H$. Evidently, (PII) implies (PI).


\begin{cor}[The Power Problem] \label{cor: alg PI}
Let $G=G_1 \ast_A G_2$ be an amalgam of finite groups. Then there
exists an algorithm which solves (PI).

That is, given finitely many subgroup generators $h_1, \ldots h_k
\in G$ and normal word $g \in G,$ the algorithm decides whether or
not some nonzero power of $g$ is in the subgroup $H = \langle h_1,
\ldots h_k \rangle$.
\end{cor}
\begin{proof}
Let $K=\langle g \rangle$.
Construct the subgroup graphs $(\Gamma(H),v_0)$ and
$(\Gamma(K),u_0)$ using the generalized Stallings' algorithm.

By Corollary \ref{cor: intersection of reduced precovers}, $Lab(C,
\vartheta)=H \cap K=\langle g^n \rangle$, where $C$ is the
connected component  of $\Gamma(H) \times \Gamma(K)$ such that
$\vartheta=(v_0,u_0) \in V(C)$. Therefore $Lab(C,
\vartheta)=\{1\}$ implies no nonzero power of $g$ is in $H$.

Thus we construct the connected component $C$ of the product graph
$\Gamma(H) \times \Gamma(K)$. The verification whether or not
$Lab(C, \vartheta)=\{1\}$ can be done as is explained in the proof
of Corollary~\ref{cor: alg malnormality}.

\end{proof}


\subsection*{Complexity} By the complexity analysis of the
``malnormality'' algorithm given along with the proof of Corollary
\ref{cor: alg malnormality}, the complexity of the above algorithm
given by Corollary \ref{cor: alg PI} is $O \big( |E(\Gamma(H))|^2
\cdot |E(\Gamma(K))| ^2 \big).$ That is $O \big( m^2 \cdot |g|^2
\big), $ where $m$ is the sum of lengths of $h_1, \ldots h_k$.


\bigskip

Following \cite{l_s}, we say that a word $g \equiv g_1g_2 \cdots
g_k \in G=G_1 \ast_A G_2$ given by the normal decomposition $(g_1,
g_2, \ldots, g_k)$ is \emph{cyclically reduced}  if $k \leq 1$ or
if $g_1$ and $g_k$ are in different factors of $G$. Hence if $g$
is cyclically reduced  then all cyclic permutations of $(g_1, g_2,
\ldots, g_k)$ define normal words.
Obviously, if $g \in G$ is cyclically reduced then $g$ is
\emph{freely cyclically reduced}, that is $g \not\equiv xg'x^{-1}$
($x  \in X^{\pm}$).

\begin{lem} \label{lem: cyclically normal}

Let $g \in G$ be a normal word given by the normal decomposition
$(g_1, g_2, \ldots, g_k)$. Then there exists a  normal word $x \in
G$ and a  cyclically reduced word $g' \in G$ such that $g=_G
xg'x^{-1}$ and the word $xg'x^{-1}$ is in normal form.

\end{lem}
\begin{proof}
If $k=1$ then the statement is trivial: $x \equiv 1$ and $g'
\equiv g$. If $k$ is an even number then the syllables $g_1$ and
$g_k$ are, evidently, in different free factors. Therefore $g$ is
cyclically reduced. Thus the statement is trivial.

Assume now that  $k$ is odd. The proof is by induction on the
syllable length of $g$, that is on $k$.

If $g_kg_1 \in G_i \setminus A$ then we put $g_k'=_{G_i} g_kg_1$
($i \in \{1,2\}$). Thus $g=_Gg_1(g_2 \cdots g_k')g_1^{-1}$. Since
$g_2 \in G_j \setminus A$ ($1 \leq i \neq j \leq 2$), the word $g'
\equiv g_2 \cdots g_k'$ is normal and cyclically reduced.
Moreover, the words $x \equiv g_1$ and   $xg'x \equiv g_1g_2
\cdots g_{k-1}(g_k'g_1^{-1})$ are normal.

If  $g_kg_1 \in G_i \cap A$ ($i \in \{1,2\}$) we take $b \in G_j
\cap A$ ($1 \leq i \neq j \leq 2$)  such that $b=_G g_kg_1$. Since
$g_{k-1} \in G_j \setminus A$, we have $g_{k-1}b \in G_j \setminus
A$. Let $g'_{k-1}=_{G_j} g_{k-1}b$.
Then
$$g=_G g_1(g_2 \cdots g_{k-2}g'_{k-1})g_1^{-1}.$$
We put  $x \equiv g_1$, $g' \equiv g_2 \cdots g_{k-2}g'_{k-1}$.
Thus $x$, $g'$ and $xg'x^{-1}$ are normal words. Moreover, by the
inductive assumption, there exists a  normal word $x' \in G$ and a
cyclically reduced word $g'' \in G$ such that $g'=_G
x'g''(x')^{-1}$, where $x' \equiv g_2 \cdots g_l$ ($l \leq
(n-1)/2$) and the word $x'g''(x')^{-1}$ is in normal form.

Thus the words $xx'$ and $(xx')g''(xx')^{-1}$ are normal. Since
$g=_G (xx')g''(xx')^{-1}$, we are done.

\end{proof}


\begin{thm} \label{order problem finite grp}
Let $H$ be a finitely generated subgroup of an amalgam of finite
groups $G=G_1 \ast_A G_2$.
Let $g \in G$ be a non torsion element such that $g^n \in H$ for
some $n \geq 1$.

Then there exists $1 \leq z \leq |V(\Gamma(H))|$ such that $g ^ z
\in H$.
\end{thm}
\begin{proof}
Assume that  $g \not\in H$ otherwise the statement is trivial. Let
$n \geq 1$ be the smallest positive integer such that $g^n \in H$.
Since $g \not\in H$, we have $n >1$. Suppose that $n
> |V(\Gamma(H))|>1$ otherwise the statement is trivial.

Without loss of generality we can assume that $g$ is a normal word
given by the normal decomposition $g \equiv g_1 \cdots g_k$, where
$k >1$ since $g$ is non torsion.

By Lemma~\ref{lem: cyclically normal}, there exits a normal word
$x \in G$ and a  cyclically reduced word $ g' \in G$,
such that $g=_G xg'x^{-1}$ and the word $x g' x^{-1}$ is normal.

Note that  the syllable length of  $g'$ is greater than 1.
Otherwise $g' $ is an element of  either $G_1$ or of $G_2$. Thus
$g$ is a conjugate of an element of either $G_1$ or of $G_2$.
Therefore, by the Torsion Theorem (IV.2.7 in \cite{l_s}),
$g$ is a torsion element of $G$, which contradicts our assumption.

Therefore  $g^n=_G x (g')^n x^{-1}$ and the word $x (g')^n x^{-1}$
is normal. Hence, by Theorem~\ref{thm: properties of subgroup
graphs} (4), there exists a normal path $p$ in $\Gamma(H)$ with
$\iota(p)=\tau(p)=v_0$ and $lab(p) \equiv x (g')^n x^{-1}$. Since
the graph $\Gamma(H)$ is well-labelled, there is a decomposition
$p=tq\overline{t}$, where $$\iota(t)=v_0, \ \tau(t)=v, \ lab(t)
\equiv x, \ {\rm and} \ \iota(q)=\tau(q)=v, lab(q)\equiv (g')^n.$$

Since the word $g'$ is freely cyclically reduced, we have $|(g')^n
|=|g'| \cdot n.$ Hence we can set $v_m=u \cdot (g' )^m$, $1 \leq m
\leq n$. Since $n
> |V(\Gamma(H))|$, there exist $1 \leq i < j \leq n$ such that
$v_i=v_j$. Thus $v_i \cdot (g' )^{j-i}=v_j=v_i$. Therefore $v=v
\cdot (g' )^n=v \cdot (g'v)^{n-(j-i)}$. Hence
$$v_0 \cdot \big( x (g')^{n-(j-i)} x^{-1} \big) =v \cdot \big((g' )^{n-(j-i)}  x^{-1} \big)=
 v \cdot x^{-1}=v_0.$$
Thus $x(g')^{n-(j-i)}x^{-1} \in H$. Hence $g^{n-(j-i)} \in H$.
Since $1 \leq i <j $,  we have $1 \leq n-(j-i) < n$. This
contradicts with the choice of $n$.

\end{proof}

\begin{cor}[The Power Problem] \label{cor: alg PII}
Let $G=G_1 \ast_A G_2$ be an amalgam of finite groups. Then there
exists an algorithm which solves (PII).

That is, given finitely many subgroup generators $h_1, \ldots h_k
\in G$ and normal word $g \in G,$ the algorithm finds the minimal
nonzero power $n$ such that $g^n \in H = \langle h_1, \ldots h_k
\rangle$.
\end{cor}
\begin{proof}
We begin by rewriting the word $g$ as a normal word $xg'x^{-1}$,
where  $x \in G$ is a normal word and $ g' \in G$ is a  cyclically
reduced word. This is possible  by Lemma~\ref{lem: cyclically
normal} and can be done according to the process described in the
proof of this lemma. Thus $g=_G xg'x^{-1}$.

If $l(g')=1$ then $g' \in G_i$ ($i \in \{1,2\}$). Thus $g$ is a
torsion element of $G$.
Let $ o(g')$ be the order of $g'$. Since $1 < o(g)=o(g')  \leq |
G_i|$, we have to verify whether $g^m \in H$, for all $1 \leq m
\leq |G_i|-1$, and to stop when the first such power is found or
when $g^m=_G 1$, that is no nontrivial power of $g$ is in $H$.

By Theorem~\ref{thm: properties of subgroup graphs} (4), such a
verification can be done using the subgroup graph
$(\Gamma(H),v_0)$ constructed by the generalized Stallings'
algorithm. That is $g^m \in H$ if and only if its normal form
labels a normal path in $\Gamma(H)$ closed at the basepoint $v_0$.
If   $(g')^m \not\in A$  ($1 \leq m \leq |G_i|-1$), then $x(g')^m
x$ is a normal word. Otherwise we just rewrite it as a normal
word.

If $l(g')>1$, then, by the proof of Theorem~\ref{order problem
finite grp}, $g^m \in H$ if  and only if there exists a path $p$
in $\Gamma(H)$ closed at $v_0$ with $lab(p) \equiv x(g')^m x^{-1}$
such that $1 \leq m \leq |V(\Gamma(H))|$.

Hence we   try to read $x(g')^mx^{-1}$ on $\Gamma(H)$ starting at
$v_0$, for all $1 \leq m \leq |V(\Gamma(H))|$. That is we begin
with $m=1$ and stop when we succeed to read $x(g')^mx^{-1}$   at
the first time. If no such $m$ is found then no nonzero power of
$g$ is in $H$.

\end{proof}

\subsection*{Complexity} By Theorem~\ref{thm: properties of subgroup graphs} (5), the
construction of $\Gamma(H)$ takes $O(m^2)$, where $m$ is the sum
of the lengths of  $h_1, \ldots h_k$. To find the desired normal
form of $g$, which is $xg'x^{-1}$, takes $O(|g|)$. A verification
of whether or not $x(g')^i x^{-1}$ can be read on $\Gamma(H)$
starting at $v_0$ ($1 \leq i \leq |V(\Gamma(H)|$) takes $O(|g|
\cdot |V(\Gamma(H)|)$, when $g$ is non torsion. Otherwise it takes
$O(|g| \cdot |G_i|)$ ($i\in \{1,2\}$). Since the information about
the factors, $G_1$ and $G_2$, is given and it is not a part of the
input, it takes $O(|g|)$.

Since, by Theorem~\ref{thm: properties of subgroup graphs} (5),
$|V(\Gamma(H))|$ is proportional to $m$, we conclude that the
complexity of the algorithm given along with the proof of
Corollary \ref{cor: alg PII} is $O ( m^2 + m \cdot |g| )$. Thus
the algorithm is quadratic in the size of the input. Moreover, it
is faster than the algorithm presented in Corollary~\ref{cor: alg
PI} which solves (PI).


\appendix
\section{}


Below we follow the notation of Grunschlag \cite{grunschlag},
distinguishing between the ``\emph{input}'' and the ``\emph{given
data}'', the information that can be used by the algorithm
\emph{``for free''}, that is it does not affect the complexity
issues.

\begin{center}
\large{\emph{\underline{\textbf{Algorithm}}}}
\end{center}

\begin{description}
\item[Given] Finite groups $G_1$, $G_2$, $A$ and the amalgam
$G=G_1 \ast_{A} G_2$ given via $(1.a)$, $(1.b)$ and $(1.c)$,
respectively.

We assume that the Cayley graphs and all the relative Cayley
graphs of the free factors are given.
\item[Input]  A finite set $\{ g_1, \cdots, g_n \} \subseteq G$.
\item[Output] A finite graph $\Gamma(H)$ with a basepoint $v_0$
which is a reduced precover of $G$ and the following holds
\begin{itemize}
 \item
$Lab(\Gamma(H),v_0)=_{G} H$;
 \item $H=\langle g_1, \cdots, g_n \rangle$;
 \item a normal word $w$ is in $H$ if and only if
  there is a loop (at $v_0$) in $\Gamma(H)$
labelled by the word $w$.
 \end{itemize}
\item[Notation] $\Gamma_i$ is the graph obtained after the
execution of the $i$-th step.

%
%
\medskip

    \item[\underline{Step1}] Construct a based set of $n$ loops around a common distinguished
vertex $v_0$, each labelled by a generator of $H$;
    \item[\underline{Step2}] Iteratively fold edges and cut hairs %
    \footnote{A \emph{hair} is an edge one of whose endpoint has degree 1};
  \item[\underline{Step3}] { \ }\\
\texttt{For} { \ } each $X_i$-monochromatic component $C$ of
$\Gamma_2$ ($i=1,2$) { \ } \texttt{Do} \\
\texttt{Begin}\\
    pick an edge $e \in E(C)$; \\
    glue a copy  of $Cayley(G_i)$   on $e$ via identifying $ 1_{G_i} $  with $\iota(e)$ \\
    and identifying the two copies of $e$ in $Cayley(G_i)$ and in $\Gamma_2$; \\
    \texttt{If}  { \ } necessary  { \ } \texttt{Then} { \ } iteratively fold
    edges; \\
 \texttt{End;}

 \item[\underline{Step4}] { \ } \\
\texttt{ For}  { \ } each $v \in VB(\Gamma_3)$ { \ } \texttt{ Do} \\
 \texttt{If} { \ } there are paths $p_1$ and $p_2$, with $\iota(p_1)=\iota(p_2)=v$
 and $\tau(p_1)~\neq~\tau(p_2)$  such that
 $$lab(p_i) \in G_i \cap A \ (i=1,2) \ {\rm and} \  lab(p_1)=_G
 lab(p_2)$$
\texttt{ Then} { \ } identify $\tau(p_1)$ with $\tau(p_2)$; \\
 \texttt{If}  { \ } necessary  { \ } \texttt{Then} { \ } iteratively fold
    edges; \\

 \item[\underline{Step5}]
%
Reduce  $\Gamma_4$ by an iterative  removal of all
(\emph{redundant})
 $X_i$-monochromatic components $C$ such that
\begin{itemize}
 \item $(C,\vartheta)$ is isomorphic to $Cayley(G_i, K, K \cdot 1)$, where $K \leq A$ and
$\vartheta \in VB(C)$;
 \item  $|VB(C)|=[A:K]$;
 \item one of the following holds
    \begin{itemize}
        \item  $K=\{1\}$ and $v_0 \not\in VM_i(C)$;
        \item $K$ is  a nontrivial subgroup of $A$ and $v_0  \not\in V(C)$.\\
    \end{itemize}
\end{itemize}

Let $\Gamma$ be the resulting graph;\\

\texttt{If}  { \ }
$VB(\Gamma)=\emptyset$ and $(\Gamma,v_0)$ is isomorphic to $Cayley(G_i, 1_{G_i})$ \\
\texttt{Then} { \ } we set $V(\Gamma_5)=\{v_0\}$ and
$E(\Gamma_5)=\emptyset$; \\
\texttt{Else} { \ } we set $\Gamma_5=\Gamma$.

 \item[\underline{Step6}] { \ } \\
 \texttt{If} { \ }
 \begin{itemize}
  \item $v_0 \in VM_i(\Gamma_5)$ ($i \in \{1,2\}$);
  \item $(C,v_0)$ is isomorphic to $Cayley(G_i,K,K \cdot 1)$, where $L=K \cap A$ is a nontrivial
  subgroup of
 $A$ and $C$ is a $X_i$-monochromatic component of $\Gamma_5$ such that $v_0 \in V(C)$;
  \end{itemize}
\texttt{Then} { \ } glue to $\Gamma_5$ a $X_j$-monochromatic
component ($1 \leq i \neq j \leq 2$) $D=Cayley(G_j,L,L \cdot 1)$
via identifying $L \cdot 1$ with $v_0$ and \\
identifying the vertices $L \cdot a$ of  $Cayley(G_j,L,L \cdot 1)$
with the vertices $v_0 \cdot a$ of $C$, for all $a \in A \setminus
L$.

Denote $\Gamma(H)=\Gamma_6$.

\end{description}


\begin{remark} \label{stal-mar-meak-kap-m}
{\rm Note that the first two steps of the above algorithm
correspond precisely to the Stallings' folding algorithm for
finitely generated subgroups of free groups  \cite{stal, mar_meak,
kap-m}.} \e
\end{remark}


\begin{figure}[!h]
\psfrag{x }[][]{$x$} \psfrag{y }[][]{$y$} \psfrag{v }[][]{$v$}
\psfrag{x1 - monochromatic vertex }[][]{{\footnotesize
$\{x\}$-monochromatic vertex}}
\psfrag{y1 - monochromatic vertex }[][]{\footnotesize
{$\{y\}$-monochromatic vertex}}
\psfrag{ bichromatic vertex }[][]{\footnotesize {bichromatic
vertex}}
\includegraphics[width=0.8\textwidth]{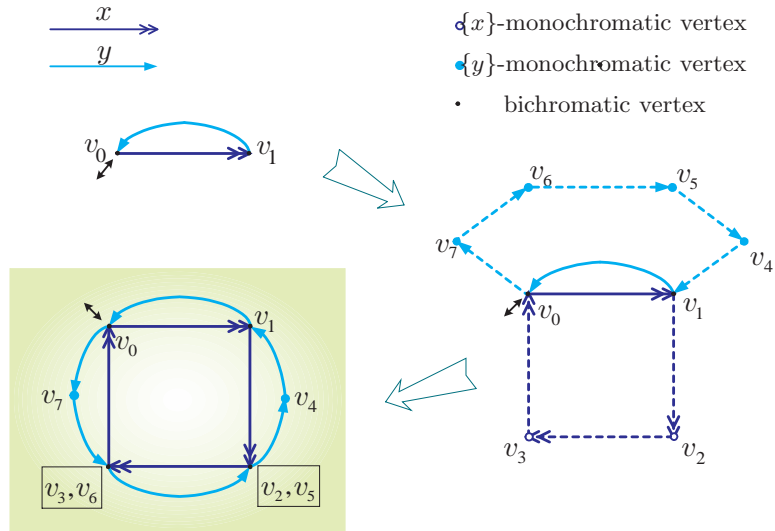}
\caption[The construction of $\Gamma(H_1)$]{ \footnotesize {The
construction of $\Gamma(H_1)$.}
 \label{fig: example of H=xy}}
\end{figure}

\begin{figure}[!hb]
\psfrag{v }[][]{$v$}
\includegraphics[width=0.8\textwidth]{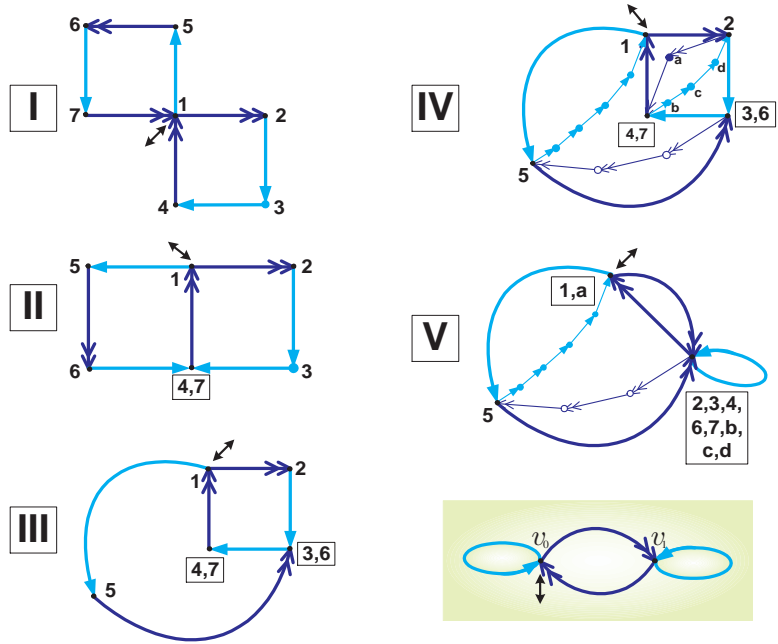}
\caption[The construction of $\Gamma(H_2)$]{ \footnotesize {The
construction of $\Gamma(H_2)$.} \label{fig: example of H=xy^2x,
yxyx}}
\end{figure}

\begin{ex} \label{example: graphconstruction}
{\rm Let $G=gp\langle x,y | x^4, y^6, x^2=y^3 \rangle$.

Let $H_1$ and $H_2$ be  finitely generated subgroups of $G$ such
that
$$H_1=\langle xy \rangle \ {\rm and} \ H_2=\langle xy^2, yxyx \rangle.$$

The construction of $\Gamma(H_1)$ and $\Gamma(H_2)$ by the
algorithm presented above is illustrated on Figures \ref{fig:
example of H=xy}
 and  \ref{fig: example of H=xy^2x, yxyx}.}
\e
\end{ex}


\end{document}